\documentclass[12pt,reqno]{amsart}

\usepackage{amsmath,amssymb,amsthm}
\usepackage{amsaddr}
\usepackage{bm}
\usepackage{natbib}
\usepackage{graphicx,here,comment}
\usepackage{algorithm,algorithmic}
\usepackage{xcolor}
\usepackage{fullpage}
\usepackage{booktabs}
\usepackage{amsmath}
\usepackage{mathtools}
\usepackage{enumerate}
\usepackage[colorlinks=true, allcolors=blue]{hyperref}

\newtheorem{theorem}{Theorem}
\newtheorem{corollary}[theorem]{Corollary}
\newtheorem{lemma}[theorem]{Lemma}
\newtheorem{proposition}[theorem]{Proposition}
\newtheorem{assumption}{Assumption}
\theoremstyle{definition}
\newtheorem{definition}{Definition}
\newtheorem{remark}{Remark}
\newtheorem{example}{Example}

\newcommand{\R}{\mathbb{R}}
\newcommand{\N}{\mathbb{N}}

\newcommand{\mS}{\mathcal{S}}

\newcommand{\mM}{\mathcal{M}}

\newcommand{\bfX}{\mathbf{X}}

\newcommand{\bfG}{\mathbf{G}}
\newcommand{\Ep}{\mathbb{E}}
\renewcommand{\Pr}{\mathbb{P}}

\newcommand{\mW}{\mathcal{W}}

\renewcommand{\hat}{\widehat}

\newcommand{\argmin}{\operatornamewithlimits{argmin}}

\newcommand{\mone}{\textbf{1}}

\newcommand{\indep}{\perp \!\!\! \perp}

\newcommand{\set}{B}
\newcommand{\loss}{\ell}
\newcommand{\regularizer}{R}

\usepackage{color}

\usepackage{newtxtext,newtxmath}

\DeclareMathOperator*{\esssup}{ess\,sup}

\usepackage{mathtools}
\mathtoolsset{showonlyrefs=true}

\allowdisplaybreaks

\title{Universality of estimators for high-dimensional linear models\\with block dependency}

\author{Toshiki Tsuda$^1$, Masaaki Imaizumi$^{2,3}$}

\address{$^1$Yale University, $^2$The University of Tokyo, $^3$ RIKEN Advanced Intelligence Project}

\date{ \today, \textit{Mail}: \url{toshiki.tsuda@yale.edu}, \url{imaizumi@g.ecc.u-tokyo.ac.jp}}

\allowdisplaybreaks

\begin{document}

\begin{abstract}
We study the universality property of estimators for high-dimensional linear models, which implies that the distribution of estimators is independent of whether the covariates follow a Gaussian distribution. Recent developments in high-dimensional statistics typically require covariates to strictly follow a Gaussian distribution to precisely characterize the properties of estimators. To relax this Gaussianity requirement, the existing literature has examined conditions under which estimators achieve universality. In particular, independence among the elements of the high-dimensional covariates has played a critical role. In this study, we focus on high-dimensional linear models with covariates exhibiting block dependence, where covariate elements can only be dependent within each block, and show that estimators for such models retain universality. Specifically, we prove that the distribution of estimators with Gaussian covariates can be approximated by the distribution of estimators with non-Gaussian covariates having the same moments under block dependence. To establish this result, we develop a generalized Lindeberg principle suitable for handling block dependencies and derive new error bounds for correlated covariate elements. We further demonstrate the universality result across several different estimators. 
\end{abstract}
\maketitle

\section{Introduction}

We consider an estimator for a linear model obtained by minimizing an empirical risk. 
Let $p \in \N$ denote the dimension and let $\mathbb{R}$ denote a space of output variables.
Let $n \in \N$ be the number of observations, and $(X_1,Y_1),...,(X_n,Y_n)$ be a set of independently and identically distributed $\R^p \times \mathbb{R}$-valued random observations generated from a statistical model:
\begin{align}
   Y_i = X_i^\top \theta_0 + \xi_i, ~ i= 1,...,n, \label{eq:linear_model}
\end{align}
where $\theta_0 \in \R^p$ is an unknown parameter and $\xi_i$ are independent noise variables. 
Given a convex loss function $\loss_0: \R \to \R$ and a convex regularizer $\regularizer_0: \R^p \to \R$, we consider the estimator $\hat{\theta}_{\mathbf{X}}$ via the empirical risk minimization problem:
\begin{equation}\label{def:theta_x}
    \hat{\theta}_{\mathbf{X}}\in\argmin_{\theta\in\R^p}\left\{\frac{1}{n}\sum_{i=1}^{n}\loss_{0}(Y_{i}-X_{i}^{\top}\theta)+\regularizer_{0}(\theta)\right\}.
\end{equation}
In this study, we show the universality of the distribution of this estimator $\hat{\theta}_{\mathbf{X}}$, that is, the distribution of the estimator does not depend on whether the covariates $X_i$ follow a Gaussian distribution.
Specifically, we will show that the universality holds even in the presence of weak dependencies among elements of the random covariate $X_i$.

The Gaussianity of covariates is a particularly important property in the analysis of high-dimensional statistics, which has been rapidly developing in recent years. Here, the term \textit{high-dimensional} refers to the so-called proportional high-dimensional regime, where both the number of parameters \( p \) in the statistical model and the amount of data \( n \) used for estimation diverge while maintaining a certain ratio. In this setting, several advanced frameworks allow the precise determination of the risk and asymptotic distribution of estimators. These tools include not only the random matrix theory \citep{marchenko1967distribution,dobriban2018high,hastie2022surprises} but also frameworks such as the convex Gaussian minimax (CGMT) theorem \citep{stojnic2013framework,thrampoulidis2015regularized,thrampoulidis2018precise,miolane2021distribution,salehi2019impact,montanari2019generalization,loureiro2021learning,liang2022precise,celentano2023lasso,koehler2021uniform,tsuda2023benign}, the approximate message passing \citep{donoho2009message,bayati2011dynamics,bayati2011lasso,rangan2011generalized,javanmard2013state,donoho2016high,sur2019modern,barbier2019optimal,mondelli2021approximate,zhao2022asymptotic,feng2022unifying,sawaya2023moment}, the leave-one-out \citep{el2013robust,bean2013optimal,el2018impact,sur2019likelihood,yadlowsky2021sloe}, and others \citep{mezard1987spin,mezard2009information,bellec2022observable,sawaya2024high}, which have been applied to analyze the properties of statistical models across a wide range of domains. Although these developments are remarkable, many of these frameworks heavily rely on the strict Gaussianity of covariates, which poses a practical limitation.

Relaxing the Gaussianity assumption on covariates in high-dimensional statistics is crucial for enhancing the practicality of high-dimensional methods. The property whereby the behavior of statistics or estimators does not depend on Gaussianity is known as \textit{universality}. Specifically, universality implies that properties of statistics under non-Gaussian covariates asymptotically match those under Gaussian covariates sharing the same moments. Several studies \citep{korada2011applications,montanari2017universality,hu2022universality,goldt2022gaussian,montanari2022universality,dudeja2022spectral, lahiry2023universality} have demonstrated that, in risk minimization problems, the minimized risk values using high-dimensional models exhibit universality.
The universality of approximate message passing algorithms has also been extensively studied in the context of rotationally invariant matrices \citep{fan2022approximate,dudeja2023universality,wang2024universality}.
Moreover, universality has been examined not only in terms of risks but also in terms of detailed properties of estimators required for statistical inference.
\cite{han2023universality} shows that estimators for high-dimensional parameters obtained through risk minimization exhibit universality.
\cite{han2024entrywise} demonstrates universality for estimators derived from first-order algorithms. Additionally, universality has been shown to hold within more specific high-dimensional models.

A challenge in the universality of estimators lies in the conditions required for non-Gaussian covariates, particularly the independence among their elements.
\citet{han2023universality} and \citet{han2024entrywise} state that, for estimators to exhibit universality, each element of the high-dimensional covariates must be independent. However, this condition often fails to hold for real high-dimensional covariates, posing an obstacle to applying universality results in various statistical models.
To address this issue, \cite{lahiry2023universality} demonstrates that the universality of estimators and risks holds under weaker assumptions, specifically, block dependence among elements of $p$-dimensional covariates with a diagonal covariance matrix.
Despite these advancements, relaxing independence for universal estimators remains a significant challenge.

In this study, we show that universality holds for estimators obtained by risk minimization even when elements of covariates exhibit dependencies. Specifically, we consider block dependence, where each element of the covariate vectors depends on at most $d$ other elements.
We demonstrate that when the dependency coefficient $d$ has a certain polynomial order in
$n$ or $p$, the distribution of estimators does not depend on the Gaussianity of covariates in the limit where $n$ and $p$ diverge while maintaining a proportional relationship. Our research extends the universality framework established by \cite{han2023universality}, incorporating the dependency structure considered in \cite{lahiry2023universality}.
As an application of our results, we establish universality for the distribution of robust estimators obtained by minimizing the absolute loss and the Huber loss, as well as of the lasso and ridge estimators.

As technical aspects, we develop the following proof technique to demonstrate the universality results presented above. Specifically, we introduce a generalized Lindeberg principle tailored to handle block dependence. In our approach, we measure the discrepancy between two design matrices by focusing on cells of dependent covariate elements. We then show that this error is bounded by a cubic polynomial in the number of correlated elements under block dependence, which is crucial for controlling the order of block dependence.
This technique enables a more precise evaluation of errors compared to previous studies, thus facilitating universality under higher-order block dependence.

We highlight the differences between our research and related studies.
\cite{lahiry2023universality} is closest to our work, introducing the notion of block dependence and deriving universality results for estimators and risks under some dependence. 
A key difference of ours from 
\cite{lahiry2023universality} is the universality under general covariance structure. 
This contribution is based on the leverage of the new CGMT by  \cite{akhtiamov2024novel}.
Another difference is the proof technique: while \cite{lahiry2023universality} employs Stein's lemma to establish the universality, we develop a proof based on the Lindeberg method.

\subsection{Notation}
For $x,x' \in \R$, we define their maximum as $x \vee x' := \max\{x,x'\}$.
For $z \in \N$, $[z] := \{1,2,...,z\}$ is a set of natural numbers no more than $z$.
For a vector $x = (x_1,...,x_d)^\top \in \R^d$, $x_j$ denotes the $j$-th element of $x$ for $j \in [d]$.
For a random variable $X$ and $q \geq 1$, $\|X\|_q := E[|X|^q]^{1/q}$ denotes its $q$-norm.
For a positive semidefinite matrix $A$, $\|x\|_{A}^{2}:=\langle x,Ax\rangle$ denotes the Mahalanobis (semi-)norm.
For $x \in \R^d$ and a function $h: \R^d \to \R$, $\partial_j h(x) := \partial h(x)  / \partial {x_j}$ denotes a partial derivative in terms of $x_j$ for $j \in [d]$. 
For a matrix $A \in \R^{d_1 \times d_2}$ and a function $h: \R^{d_1 \times d_2} \to \R$, $\partial_{j_1, j_2} h(A) := \partial h(A)  / \partial {A_{j_1,j_2}}$ denotes a partial derivative in terms of a $(j_1,j_2)$-element $A_{j_1,j_2}$ of $A$.  
For a function $f:\Omega \to \R$ with $\Omega \subset \R^d$, $\|f\|_{L^q} := (\int_\Omega |f(x)|^q dx)^{1/q}$ denotes the $L^q$-norm with $q \geq 1$. 
For a variable $x$, $C_x$ denotes a positive constant depending polynomially on $x$, whose value may change from line to line.
In addition, $a \lesssim_x b$ denotes $a \leq C_x b$, and  $a \lesssim b$ denotes $a \leq C b$ with a universal constant $C>0$.
$a \asymp b$ denotes both $a \lesssim b$ and $b \lesssim a$ holds.
For a matrix $A \in \R^{d \times d}$,  $\lambda_{\max}(A)$ denotes the largest eigenvalue of $A$.
We define the Moreau envelope of a function $f: \R \to \R$ as $e_f(x; \lambda) := \min_{z} (2\lambda)^{-1} \|x - z\|_2^2 + f(z)$ with $\lambda > 0$.

\section{Setup with block sparsity}

\subsection{Setup: estimator with convex risk minimization}
We consider an estimator with empirical loss minimization under general convex loss and convex regularization.
Recall that the random $n$ observations $(X_1,Y_1),...,(X_n,Y_n)$ follow a statistical model.
We assume that the $\R^p$-valued random vectors $X_1,...,X_n$ follow $\Ep[X_i] = 0$ and $\Ep[X_i X_i^\top] = \Sigma / n $ with a positive definite matrix $\Sigma \in \R^{p \times p}$.
Also, recall that the estimator $\hat{\theta}_\mathbf{X}$ is defined in \eqref{def:theta_x} with the convex loss function $\loss_0$ and the regularizer $f_0$.
We further define design matrices $\bfX = (X_1,...,X_n)^\top  \in \R^{n \times p}$ of the covariates.

For simplicity, we rewrite the minimization problem \eqref{def:theta_x} and consider an equivalent problem.
We define another minimizer $\hat{w}_{\mathbf{X}}$ by the redefined version of the minimization problem as follows:
\begin{equation}\label{loss::redef}
    \hat{w}_{\mathbf{X}}\in\argmin_{w\in \R^{p}}H_{\loss}(w,\mathbf{X}):=\argmin_{w\in \R^{p}}\left\{\frac{1}{n}\sum_{i=1}^{n}\loss_{i}(X_{i}^{\top}w)+\regularizer (w)\right\}.
\end{equation}
The minimizer has an explicit relation $\hat{w}_{\mathbf{X}} = \hat{\theta}_{\mathbf{X}}-\theta_{0}$ and
$\loss_{i}(X_{i}^{\top}w)$ and $\regularizer (w)$ in \eqref{loss::redef} correspond to $\loss_{0}(Y_{i} - X_{i}^{\top}\theta) = \ell_i(X_i^\top w)$ and $\regularizer_{0}(\theta_{0}+w)$ in \eqref{def:theta_x}, respectively.

\subsection{Basic assumptions} \label{sec:basic_assumption}

We introduce assumptions on the dimensionality $p$ and the loss function $\loss_i$. 
These assumptions are common in high-dimensional statistics and allow for a wide class of estimators.

\begin{assumption}[Proportionally high-dimension]\label{ass::regime}
    There exists $\tau\in(0,1)$ such that $\tau\leq p/n\leq 1/\tau$. 
\end{assumption}
This assumption restricts the high-dimensional setting we consider to the proportionally high-dimensional regime. 

\begin{assumption}[Loss function]\label{ass::loss}
There exist non-negative real numbers $q_{\ell}$ ($\ell = 0, 1, 2, 3$), constants $\bar{\rho} \in (0,1]$, and ${L_{\loss_{i}} \geq 1 : i \in [n]}$, along with two measurable functions $\mathcal{D}_{\loss}(\rho)$ and $\mathcal{M}_{\loss}(\rho): (0, \bar{\rho}) \rightarrow \mathbb{R}_{\geq 0}$, where $\mathcal{M}_{\loss}(\rho) \leq 1 \leq \mathcal{D}_{\loss}(\rho)$, that satisfy the following conditions:
\begin{enumerate}
\item[(i)] The functions $\loss_{i}$ grow at most polynomially, meaning for $i \in [n]$,
\begin{equation*}
\sup_{x \in \mathbb{R}} \frac{|\loss_{i}(x)|}{1 + |x|^{q_{0}}} \leq L_{\loss_{i}}.
\end{equation*}
\item[(ii)] There exist smooth approximations $\{\loss_{i;\rho}: \mathbb{R} \rightarrow \mathbb{R}_{\geq 0}\}_{\rho \in (0, \bar{\rho})}$ of $\loss_{i}$ such that (a) $\loss_{i;\rho} \in C^{3}(\mathbb{R})$, (b) $\max_{i \in [m]}\|\loss_{i;\rho} - \loss_{i}\|_{\infty} \leq \mathcal{M}_{\loss}(\rho)$, and (c) the derivatives of ${\loss_{i;\rho}}$ adhere to the following self-bounding property:
\begin{equation*}
\max_{\ell = 1, 2, 3} \max_{i \in [m]} \sup_{x \in \mathbb{R}} \frac{|\partial \loss_{i;\rho}(x)|}{1 + |\loss_{i;\rho}(x)|^{q_{\ell}}} \leq \mathcal{D}_{\loss}(\rho).
\end{equation*}
\end{enumerate}
In addition, we define
\(
    \mathbf{q}:=\max\{3q_{1},q_{1}+q_{2},q_{3}\}, \bar{\mathbf{q}}_{0}:=2\lceil{q_{0}\mathbf{q}/2\rceil}
\), and $Av(\{L_{\loss_{i}}^{\mathbf{q}}\}):=(1/n)\sum_{i=1}^{n}L_{\loss_{i}}^{\mathbf{q}}$. 
Moreover, $q$ refers to the set $\{q_{0}, q_{1}, q_{2}, q_{3}, \mathbf{q}, \bar{\mathbf{q}}_{0}\}$.
\end{assumption}

Assumption \ref{ass::loss} is exactly the same as an assumption on loss functions in \citet{han2023universality}. 
$ \mathcal{M}_{\loss}(\rho)$ and $\mathcal{D}_{\loss}(\rho)$ prescribe the smooth approximation of a loss function. In particular, when a loss function is in $C^{3}(\mathbb{R})$, $\mathcal{M}_{\loss}(\rho)$ is equal to $0$. 
Assumption \ref{ass::loss}-(i) prevents loss functions from growing too fast as $|x|$ goes to infinity. In many cases, $L_{\loss_{i}}$ is characterized by a noise term of a statistical model. Assumption \ref{ass::loss}-(ii) restricts the smoothness of loss functions. It allows loss functions to have a smooth approximation with respect to the supremum norm that increases at most polynomially as in Assumption \ref{ass::loss}-(ii) . 
We provide several common examples and their corresponding parameters $q$:
\begin{example}[Linear regression with convex losses] \label{ex:losses}
    We consider the linear regression model $Y_i = X_i^\top \theta_0 + \xi_i$ with a true parameter $\theta_0 \in \R^p$ and an independent noise variable $\xi_i$. With this model, we consider the following loss functions:
    \begin{itemize}
  \setlength{\parskip}{0cm}
  \setlength{\itemsep}{0cm}
        \item \textit{Squared loss}: We set $\ell_0(y,z) = (y-z)^2$ and $\ell_i(z) = (z - \xi_i)^2$, then we obtain $q_0 = 2, \mathbf{q} = 3/2$, and $\bar{\mathbf{q}}_0 = 4$. Furthermore, for any $\rho > 0$, we have $\mathcal{M}_{\loss}(\rho) = 0$ and $\mathcal{D}_{\loss}(\rho) \leq C$ with some constant $C > 1$.
        \item \textit{Absolute loss}: We set $\ell_0(y,z) = |y-z|$ and $\ell_i(z) = |z - \xi_i|$, then we obtain $q_0 = 1, \mathbf{q} = 0$, and $\bar{\mathbf{q}}_0 = 0$.
        Furthermore, for any $\rho > 0$, we have $\mathcal{M}_{\loss}(\rho) \leq \rho$ and $\mathcal{D}_{\loss}(\rho) \leq C \rho^{-2}$ with some constant $C > 1$.
        \item \textit{Huber loss}: We set $\ell_0(y,z) = (y-z)^2 \mone\{|y-z| \leq \eta\} +  (\eta |y-z| - \eta^2/2) \mone\{|y-z| > \eta\}$ with $\eta > 0$, then we obtain $q_0 = 1, \mathbf{q} = 0$, and $\bar{\mathbf{q}}_0 = 3$.
        Furthermore, for any $\rho > 0$, we have $\mathcal{M}_{\loss}(\rho) \leq \rho$ and $\mathcal{D}_{\loss}(\rho) \leq C \rho^{-2}$ with some constant $C > 1$.
    \end{itemize}
\end{example}

For the convex regularizer $R$, we introduce the following continuity measure: for $L, \delta > 0$, we define 
\begin{align}
    \mathcal{M}_{\regularizer}(L,\delta):=\sup_{w, w^{\prime}\in[-L, L]^{p}: \|w-w^{\prime}\|_{\infty}\leq\delta}|\regularizer (w)-\regularizer (w^{\prime})|.
\end{align}
This term measures the supremum of discrepancy in the sphere with radius $\delta$ in terms of the uniform norm.

We introduce assumptions about the distribution of the covariate \(X_i\). The assumption is for moments of the distribution of $X_i$ and the largest eigenvalue of the covariance matrix $\Sigma/n$ of $X_i$.
\begin{assumption}[Distribution of covariates]\label{ass::moment}
Let $q^{*}$ equal to $2^{(\bar{\mathbf{q}}_{0}+4)/2}$ and assume there exists $X^{0}$ such that $X_{i}=X^{0}_{i}/\sqrt{n}$ and the following inequality holds:
\begin{align}
    M:=\max&\left\{\max_{1\leq j\leq p}\|X^{0}_{i,j}\|_{q^{*}}, \lambda_{\max}(\Sigma)\right\} < \infty, \quad { \text{uniformly in $p$}}. \label{def:moments_x}
\end{align}
\end{assumption}

The boundedness of the maximum eigenvalue is trivially satisfied if the case with $\Sigma = I$, which has been exploited in the existing universality studies \citep{montanari2022universality,han2023universality}.
In our setup, since we consider a general covariance matrix $\Sigma$, we explicitly assume boundedness of the maximum eigenvalue. A similar assumption is introduced in \cite{lahiry2023universality}.

\subsection{Dependence structure}

We introduce the notion of \textit{block dependence} to describe the dependencies among elements of a $p$-dimensional random vector $X_i$.
In preparation, we define a $k$-\textit{partition} $\{\set_{j} \}_{j=1}^{k}$ of a set $[p]$, which is a sequence of $k$ cells $\set_{1},...,\set_{k} \subset [p]$ that are disjoint to each other and satisfy $\cup_{j=1}^{k} \set_{j}=[p]$.

\begin{definition}[Block dependence]
 A distribution of a $p$-dimensional random vector $X$ is called $d$-\textit{block dependent} with a dependence parameter $d$, if there exists $k \leq p$ and a $k$-partition 
 $\{\set_{j} \}_{j=1}^{k}$ of $[p]$ such that $\max_{j=1,...,k}|\set_{j}|\leq d$ and $X_{\ell}\indep X_{k}$ holds for $\ell\in \set_{i},k \in \set_{j}, i \neq j$.
\end{definition}

The above dependent structure assigns each element of $X$ to a specific cell $\set_{j}$ and restricts the maximum size of $\set_{j}$ to $d$. 
For any two random elements of $X_i$ in the same cell $\set_{j}$, the structure allows them to be dependent while two random variables that come from different cells must be independent. 
Note that the dependent elements need not be adjacent to each other.

In the following sections, we consider the situation where the observations $X_1,...,X_n$ follow the distribution which satisfies the $d$-block dependence.
Note that the observations are independently and identically distributed, hence the partition $\{\set_j\}_{j=1}^k$ is independent of $i=1,...,n$.

\begin{remark}[In related studies]
In several high-dimensional studies, block dependence and related concepts have been utilized. \cite{han2023universality} and \cite{han2024entrywise}, which deal with the universality of estimators, assume the independence of all elements, which can be considered as block dependence with \( d=1 \). The block dependence with \( d \geq 2 \) was introduced in \cite{lahiry2023universality}, demonstrating universality under \( d=o((n/\log n)^{1/5}) \). In the field of high-dimensional central limit theorems, \cite{chang2024central} introduces \( m \)-dependence, a related notion of the block dependence.    
\end{remark}

\section{Main result of universality}

\subsection{Gaussian analogy} \label{sec:gaussian_analogy}
We present the main result, the universality theorem, which demonstrates that the properties of the estimator $\hat{\theta}_{\mathbf{X}}$ are asymptotically equivalent to those in the case of Gaussian covariates. 
As a preparation, we define the corresponding Gaussian covariates to $\bfX$, as well as the risk and estimator for it.

Let $G_1,...,G_n$ be a set of i.i.d. vectors from $N(0_p, \Sigma / n)$, and its design matrix $\bfG = (G_1,...,G_n)^\top \in \R^{n \times p}$. 
Using the Gaussian vectors, we define Gaussian analogies of the empirical risk and the estimator as
\begin{equation}\label{loss::redef_gauss}
    \hat{w}_{\mathbf{G}}\in\argmin_{w\in \R^{p}}H_{\loss}(w,\mathbf{G}):=\argmin_{w\in \R^{p}}\left\{\frac{1}{n}\sum_{i=1}^{n}\loss_{i}(G_{i}^{\top}w)+\regularizer (w)\right\}.
\end{equation}
Note that numerous studies have investigated the asymptotic behavior of the estimator and the risk in the proportionally high-dimensional limit.

\subsection{Universality of empirical risk} \label{sec:universality_risk}

We show the universality of the minimized empirical risk under $d$-block dependence. Specifically, we show that the distributions of \(\min_w H_{\loss}(w,\mathbf{X})\) and \(\min_w H_{\loss}(w,\mathbf{G})\) are asymptotically equivalent. 
The following theorem evaluates the distance between the two distributions under general \(d\), using any $C^{3}(\mathbb{R})$ class function $g$:
\begin{theorem}\label{main_universality}
     Suppose Assumptions  \ref{ass::regime}, \ref{ass::loss} and \ref{ass::moment} hold.
     Also, suppose that the distribution of $X_1,...,X_n$ is $d$-block dependent.
     Then, for sufficiently large $n$, there exists some $C_0 =C_0(\tau,q,M)>0$ such that the following hold: for any $\mathcal{S}_{n}\subset[-L_{n}, L_{n}]^{p}$ with $L_{n}\geq 1$, and any $g\in C^{3}(\mathbb{R})$, we have
     \begin{align}
         \left|E\left[g\left(\min_{w\in \mathcal{S}_{n}}H_{\loss}(w,\mathbf{X})\right)\right]-E\left[g\left(\min_{w\in \mathcal{S}_{n}}H_{\loss}(w,\mathbf{G})\right)\right]\right|\leq C_{0}\cdot K_{g}\cdot r_{R}(L_{n}). \label{ineq:main_exp}
     \end{align}
Here $K_{g}:=1 +\max_{\ell\in[0:3]}\|g^{(\ell)}\|_{\infty}$ and $r_R(L_n)$ is defined by
\begin{align*}
    r_R(L_n):=\inf_{\rho\in(0,\bar{\rho})} \Biggl\{ &\mathcal{M}_{\loss}(\rho) +\mathcal{D}^{3}_{\loss}(\rho) \\
    &  \times \inf_{\delta\in(0,\omega_{n})}\left[\mathcal{M}_{\regularizer}(L_{n},\delta)+ Av(\{L_{\loss_{i}}^{\mathbf{q}}\})^{1/3}\cdot L_{n}^{(\bar{\mathbf{q}}_{0}+3)/{3}}\log^{2/3}_{+}(L_{n}/\delta)\sigma^{1/3}_{n}\right]\Biggr\},
\end{align*}
where we define $\omega_{n}:=n^{-(q_{0}q_{1}+3)/2}\sigma_{n}$ and 
\begin{align}
    \sigma_{n}:=\left(\frac{pd^{3+(\bar{\mathbf{q}}_{0}/2)}}{n^{5/2}}+\frac{pd^{2}}{n^{3/2}}+\frac{pd^{\bar{\mathbf{q}}_{0}+2}}{n^{(\bar{\mathbf{q}}_{0}+3)/2}}  \right). \label{def:sigma}
\end{align}
Consequently, for any $z\in\mathbb{R}$, $\varepsilon>0$, we obtain
\begin{equation*}
    \Pr\left(\min_{w\in \mathcal{S}_{n}}H_{\loss}(w,\mathbf{X})>z+3\varepsilon\right)\leq \Pr\left(\min_{w\in \mathcal{S}_{n}}H_{\loss}(w,\mathbf{G})>z+\varepsilon\right)+C_{1}(\max\{1,\varepsilon^{-3}\})r_{R}(L_{n}).
\end{equation*}
and 
\begin{equation*}
    \Pr\left(\min_{w\in \mathcal{S}_{n}}H_{\loss}(w,\mathbf{G})>z+3\varepsilon\right)\leq \Pr\left(\min_{w\in \mathcal{S}_{n}}H_{\loss}(w,\mathbf{X})>z+\varepsilon\right)+C_{1}(\max\{1,\varepsilon^{-3}\})r_{R}(L_{n}).
\end{equation*}
Here $C_1 > 0$ is a constant multiple of $C_0$.
\end{theorem}

Theorem \ref{main_universality} shows that the difference in expectations of the minimized risk functions is characterized by $r_{R}(L_{n})$, which includes the properties of the loss function $\loss_i$ and the regularizer $R(\cdot)$. 
The asymptotic equivalence of the distributions is guaranteed if $r_{R}(L_{n})$ converges to zero, which largely depends on the convergence of $\sigma_{n}$.

Note that the inequality \eqref{ineq:main_exp} is conditional on the noise variable $\xi_1,...,\xi_n$: the left-hand side is a conditional expectation on the noise and the term $r_{R}(L_{n})$ in the right-hand side is a $\xi_i$-dependent random variable. 
Even in this case, we can give a non-probabilistic upper bound on this term using the weak law of large numbers reflecting the shape of the loss function. 
The specific upper bound is given in Section \ref{sec:application} for the applications.

From the result of Theorem \ref{main_universality}, we can discuss the value of $d$ required for the error to converge to zero.
The following corollary summarizes the result:
\begin{corollary} \label{cor:universality_risk}
    Consider one of the loss functions presented in Example \ref{ex:losses} and the regularizer $R(\cdot) = \|\cdot\|_1$ or $R(\cdot) = \|\cdot\|_2^2$.
    Suppose that the conditions in Theorem \ref{main_universality}, $\bar{\mathbf{q}}_{0}\leq6$, and $Av(\{L_{\loss_{i}}^{\mathbf{q}}\}) = O(1)$ hold. 
    Set $L_n = O( n^{\alpha}\log^\kappa n )$ with $\alpha \in [0,1/(2\bar{\mathbf{q}}_{0} + 6)] $ and $ \kappa \geq 0$ as $n \to \infty$.
    Then, with setting $d = o(n^{1/4 - (\bar{\mathbf{q}}_{0} + 3)\alpha / 2}/ \log^{3/2+1/2(\bar{\mathbf{q}}_{0}+3)\kappa } n)$ as $n \to \infty$, for any $\mathcal{S}_{n}\subset[-L_{n}, L_{n}]^{p}$ with $L_{n}\geq 1$, and any $g\in C^{3}(\mathbb{R})$, we have the following as $n,p \to \infty$:
     \begin{equation*}
         \left|E\left[g\left(\min_{w\in \mathcal{S}_{n}}H_{\loss}(w,\mathbf{X})\right)\right]-E\left[g\left(\min_{w\in \mathcal{S}_{n}}H_{\loss}(w,\mathbf{G})\right)\right]\right| \to 0.
     \end{equation*}
\end{corollary}

We first note the setup of $L_n$.
As seen from Section 3 in \citet{han2023universality} or Section 7 in \citet{lahiry2023universality}, typical estimator examples satisfy $\bar{\mathbf{q}}_{0}\leq 6$, and also $L_n = O(\log^\kappa n)$ with $\kappa \geq 0$ is enough to cover the support of the estimator distribution. 
In this situation, we can set $\alpha = 0$ and obtain an order $d = o(n^{1/4})$ up to logarithmic factors that makes the discrepancy converge to zero in the proportionally high-dimensional limit. 

We discuss the relationship between Theorem \ref{main_universality} and previous results of the existing studies. 
Theorem \ref{main_universality} can be regarded as a generalization of Theorem 2.3 in \cite{han2023universality}. \cite{han2023universality} considers the case where all elements of $X_i$ are independent, i.e., $d$-block dependent with \(d=1\), and the statement of their theorem is equivalent to Theorem \ref{main_universality} with \(\sigma_n\) replaced by \(n^{-1/2}\). Additionally, Theorem 4.1 of \cite{lahiry2023universality} has a similar universality statement under the block dependence, where \(d=o((n/\log n)^{1/5})\) holds. Note that the permissible order of \cite{lahiry2023universality} is more restrictive than that allowed in Theorem \ref{main_universality} and Corollary \ref{cor:universality_risk} with $\alpha = 0$.
Furthermore, \citet{chang2024central} shows the universality of the maximum of high-dimensional random vectors with dependent random variables.
Although there is some technical similarity, the result by \citet{chang2024central} is not comparable with our result, because their setup is different from our setting.
For details, see Remark \ref{remark:connection_gaussian_approx}.

\subsection{Universality of estimator}
We show the universality of the estimator, which is the optimizer of the empirical risk. Specifically, we demonstrate that the  estimator $\hat{w}_{\mathbf{X}}$ defined in \eqref{loss::redef} and the estimator $\hat{w}_{\mathbf{G}}$ under the Gaussian covariate asymptotically have the same distribution.

\begin{theorem}\label{main_universality_structure}
     Suppose Assumptions  \ref{ass::regime}, \ref{ass::loss} and \ref{ass::moment} hold.
    Also, suppose that the distribution of $X_1,...,X_n$ is $d$-block dependent.
     Fix a measurable subset $\mathcal{S}_{n}\subset\mathbb{R}^{p}$. Suppose there exist $z\in\mathbb{R}, \rho_{0}>0$, $L_{n}\geq 1$ and $\varepsilon_n\in[0,1/4)$ such that the following hold: 
     \begin{enumerate}
         \item[(i)] Both $\|\hat{w}_{\mathbf{G}}\|_{\infty}$ and $\|\hat{w}_{\mathbf{X}}\|_{\infty}$ increase moderately in the sense that 
         \begin{equation*}
             \max\left\{\Pr(\|\hat{w}_{\mathbf{G}}\|_{\infty}>L_{n}), \Pr(\|\hat{w}_{\mathbf{X}}\|_{\infty}>L_{n})\right\}\leq \varepsilon_{n},
         \end{equation*}
          \item[(ii)] $\hat{w}_{\mathbf{G}}$ does not satisfy the specific structure $\mathcal{S}_{n}$ in the sense that 
         \begin{equation*}
             \max\left\{\Pr\left(\min_{w\in \mathbb{R}^{p}}H_{\loss}(w,\mathbf{G})\geq z+\rho_{0}\right), \Pr\left(\min_{w\in \mathcal{S}_{n}}H_{\loss}(w,\mathbf{G})\leq z+2\rho_{0}\right) \right\}\leq \varepsilon_{n}.
         \end{equation*}
     \end{enumerate} 
Then, $\hat{w}_{\mathbf{X}}$ also does not satisfy the specific structure $\mathcal{S}_{n}$ with high probability:
\begin{equation*}
    \Pr\left(\hat{w}_{\mathbf{X}}\in\mathcal{S}_{n}\right)\leq 4\varepsilon_{n}+C_{0}\max\{1,\rho_{0}^{-3}\}r_{R}(L_{n}).
\end{equation*}
Here, $r_R(L_n)$ is defined in Theorem \ref{main_universality}, and $C_0>0$ depends on $\tau, q, M$ only.
\end{theorem}

This theorem shows that if \( r_R(L_n) \) converges to zero under the conditions given in Theorem \ref{main_universality}, the distributions of the two estimators become asymptotically equivalent. 
As discussed in Section \ref{sec:universality_risk}, the convergence of \( r_R(L_n) \) is ensured when \( d = o(n^{1/4}) \) holds up to logarithmic factors.
It is summarized in the following statement:
\begin{corollary} \label{cor:corollary_distribution}
    Consider one of the loss functions presented in Example \ref{ex:losses} and the regularizer $R(\cdot) = \|\cdot\|_1$ or $R(\cdot) = \|\cdot\|_2^2$.
    Suppose that the conditions of Theorem \ref{main_universality_structure} hold with a fixed $\mS_n \subset \R^p$ and $\varepsilon_n \in [0,1/4)$.
    $\bar{\mathbf{q}}_{0}\leq6$, and $Av(\{L_{\loss_{i}}^{\mathbf{q}}\}) = O(1)$ hold.
    Set $L_n = O( n^{\alpha}\log^\kappa n )$ with $\alpha \in [0,1/(2\bar{\mathbf{q}}_{0} + 6)] $ and $ \kappa \geq 0$ as $n \to \infty$.
    Then, with setting $d = o(n^{1/4 - (\bar{\mathbf{q}}_{0} + 3)\alpha / 2}/ \log^{3/2+1/2(\bar{\mathbf{q}}_{0}+3)\kappa } n)$ as $n \to \infty$, we have the following as $n,p \to \infty$:
     \begin{equation*}
         |\Pr\left(\hat{w}_{\mathbf{X}}\in\mathcal{S}_{n}\right) - \Pr\left(\hat{w}_{\mathbf{G}}\in\mathcal{S}_{n}\right)| \leq 4\varepsilon_n + o(1).
     \end{equation*}
\end{corollary}
We provide explanations on conditions (i) and (ii) in Theorem \ref{main_universality_structure}. 
Condition (i) requires the boundedness of the estimators $\hat{w}_\mathbf{X}$ and $\hat{w}_\mathbf{G}$, which is demonstrated depending on each example. 
Condition (ii) can be investigated by a limiting value of the minimum of the empirical risk in the proportionally high-dimension regime with a Gaussian covariate, which also depends on specific examples. 
These two conditions are identical to those provided in Theorem 2.4 of \cite{han2023universality}, and more detailed discussions will be given in the subsequent section on applications.

\section{Outline of proof with Lindeberg principle}

In this section, we outline the proof of Theorem \ref{main_universality}. To emphasize the fact that Theorem \ref{main_universality} still holds even though the Gaussianity of $\bfG$ is violated, we replace $\bfG$ with another design matrix $\mathbf{W} = (W_1,...,W_n)^\top \in \R^{n \times p}$ that has the same first and second moments as $\mathbf{X}$, that is, $\Ep[W_i] = \Ep[X_i] = 0 $ and $\Ep[W_i W_i^\top ] = \Ep[X_i X_i^\top ] = \Sigma / n$ hold for $i = 1,...,n$. 

We prepare an additional notation to present the outline.
We define a smoothly approximated version of the empirical risk as $H_{\loss_\rho}(w,\mathbf{X}):=\frac{1}{n}\sum_{i=1}^{n}\loss_{i,\rho}(X_{i}^{\top}w)+\regularizer (w)$.
Let $\bar{H}_{\loss}$ denote the unnormalized function $H_{\loss}$, namely $n\cdot H_{\loss}(w,\mathbf{X})$.
In addition, we define a finite subset $\mathcal{S}_{n,\delta} \subset \mathcal{S}_n$ depending on $\delta > 0$, which will be specified later.
Furthermore, we introduce the soft minimum function that approximates the minimum function: for $\beta>0$ and $N\in\mathbb{N}$, we define the soft minimum function $F_{\beta}$ as
    \begin{equation*}
        F_{\beta}(x):=-\frac{1}{\beta}\log\left(\sum_{j=1}^{N}\exp(-\beta x_{j})\right), \quad x\in\mathbb{R}^{N}.
    \end{equation*}

\subsection{Overview}

We focus on the difference $E[g(\min_{w\in \mathcal{S}_{n}} H_{\loss}(w,\mathbf{X}))]-E[g(\min_{w\in \mathcal{S}_{n}} H_{\loss}(w,\mathbf{W}))]$ with $g \in C^3(\R)$, which is bounded in Theorem \ref{main_universality} with replacing $\mathbf{G}$ by $\mathbf{W}$.
The difference in expectations can be divided into four terms:
\begin{align}
    &\left|E\left[g\left(\min_{w\in \mathcal{S}_{n}} H_{\loss}(w,\mathbf{X})\right)\right]-E\left[g\left(\min_{w\in \mathcal{S}_{n}} H_{\loss}(w,\mathbf{W})\right)\right]\right| \notag \\
    &\leq 2\max_{\mathbf{Z}\in\{\bfX, \mathbf{W}\}} \left|E\left[g\left(\min_{w\in \mathcal{S}_{n}}  H_{\loss}(w,\mathbf{Z})\right)\right]-E\left[g\left(\min_{w\in \mathcal{S}_{n}} H_{\loss_{\rho}}(w,\mathbf{Z})\right)\right]\right|\label{eq::smooth} \\
    &\quad + 2\max_{\mathbf{Z}\in\{\bfX, \mathbf{W}\}} \left|E\left[g\left(\min_{w\in \mathcal{S}_{n,\delta}} H_{\loss_{\rho}}(w,\mathbf{Z})\right)\right]-E\left[g\left(\min_{w\in \mathcal{S}_{n}} H_{\loss_{\rho}}(w,\mathbf{Z})\right)\right]\right| \label{eq::discrete} \\
    &\quad + 2\max_{\mathbf{Z}\in\{\bfX, \mathbf{W}\}} \left|E\left[g\left(\min_{w\in \mathcal{S}_{n,\delta}} H_{\loss_{\rho}}(w,\mathbf{Z})\right)\right]-E\left[g\left(n^{-1}F_{\beta} \left((\bar{H}_{\loss_{\rho}}(w,\mathbf{Z}))_{w\in \mS_{n,\delta}}\right)\right)\right]\right| \label{eq::softmin} \\
    &\quad + \left|E\left[g\left(n^{-1}F_{\beta} \left((\bar{H}_{\loss_{\rho}}(w,\bfX))_{w\in \mS_{n,\delta}}\right)\right)\right]-E\left[g\left(n^{-1}F_{\beta} \left((\bar{H}_{\loss_{\rho}}(w,\mathbf{W}))_{w\in \mS_{n,\delta}}\right)\right)\right]\right| \label{eq::moment} \\
    & =: I + II + III + IV.
\end{align}
where we define $\bar{H}=n\times H$. 
The terms $I,II,III$, and $IV$ each have their own intuition.
$I$ is on a smooth approximation error of the loss function by $\loss_{\rho}$, $II$ is about the discretization on the parameter space from $\mathcal{S}_{n}$ to $\mathcal{S}_{n,\delta}$, $III$ is an approximation error by the soft minimum function $F_{\beta}(\cdot)$, and $IV$ is an effect of approximating the function via moment matching technique.

Our technical novelty lies in studying $IV$ with the block dependency. 
Specifically, we develop the Lindeberg-type approach to derive an upper bound deriving an upper bound for $IV$ in terms of $d$.
The remaining terms $I,II$, and $III$ are evaluated by the proof developed in \cite{han2023universality}, respectively.
We first state the proof outline for studying the term  $IV$ and then explain how to bound the remaining terms $I,II$, and $III$.

\subsection{Bound for $IV$: Lindeberg principle with block dependence}

In this section, we derive an approximation result for $d$-block dependent random variables $X_{i}$ and $W_{i}$. This proof employs two crucial lemmas. The first lemma is a generalization of the Lindeberg principle, applicable to $d$-block dependent random variables. The second lemma provides an upper bound for the sum of $\sum_{j=1}^k|\set_{j}|^m$ for $m \geq 0$. Given that the sum of $|\set_{i}|$ equals the dimension of covariates, we show that, for any $m \in \mathbb{N}$,  the sum of $|\set_{i}|^{m}$ can be bounded above by a term of order $pd^{m-1}$ with two key techniques (i) block replacement, and (ii) construction of the aligned blocks.

\subsubsection{Generalized Lindeberg principle}
We develop a generalized Lindeberg principle, which is a method used in evaluating a discrepancy between random vectors by swapping their elements and evaluating the error using the Taylor approximation. 
Specifically, we employ this principle to evaluate the discrepancy of random matrices with block dependence.
The independence of $\set_{j}$ from any other $\set_{i}$ allows us to upper bound the difference in expectations for any function $f \in C^{3}(\mathbb{R}^{n \times p})$. 

We provide a key result of the Lindeberg principle with the block dependence.
We define $p$-dimensional random vector $X_{i}^{(\set_{k})} := (X_{i,1} \mone\{1 \in \set_{k}\},...,X_{i,p} \mone\{p \in \set_{k}\})^\top \in \R^{p}$.
Then, we obtain the following lemmas.
\begin{lemma}\label{Lipschitz_higherorder}
   Consider two sequences of i.i.d. $\R^p$-valued random vectors $X_1,...,X_n$ and $W_1,...,W_n$. 
   Suppose that each $\{X_{i}\}_{i=1}^{n}$ and $\{W_{i}\}_{i=1}^{n}$ is $d$-block dependent, and both $E[X_{i}]=E[W_{i}]$ and $E[X_{i}X_{i}^{\top}]=E[W_{i}W_{i}^{\top}]$ hold. 
   Then, for any $1\leq j\leq k$ and  $f\in C^{3}(\mathbb{R}^{n\times p})$ satisfying both inequalities $\max_{i \in [n], m \in [p]}|E[\partial_{i,m} f(\mathbf{Z}^{0}_{i,j})]|<\infty$ and $ \max_{i \in [n], m,m^{\prime} \in [p]}|E[\partial_{i,m} \partial_{i,m^{\prime}} f(\mathbf{Z}^{0}_{i,j})]| < \infty$, we have
   \begin{align*}
       &|E[f(\mathbf{X})]-E[f(\mathbf{W})]| \\
       &\leq \frac{1}{2}\sum_{i=1}^{n}\sum_{j=1}^{k}\left|E\left[\int^{1}_{0}(1-t)^2\sum_{\ell \in B_j} \sum_{\ell' \in B_j} \sum_{\ell'' \in B_j} X_{i,\ell} X_{i,\ell'} X_{i,\ell''}\partial_{i, \ell}\partial_{ i, \ell'} \partial_{ i, \ell''}f(\mathbf{Z}_{i,j}^{0}+tX^{(\set_{j})}_{i})dt\right]\right| \\
        &~~ +\frac{1}{2}\sum_{i=1}^{n}\sum_{j=1}^{k}\left|E\left[\int^{1}_{0}(1-t)^2\sum_{\ell \in B_j} \sum_{\ell' \in B_j} \sum_{\ell'' \in B_j} W_{i,\ell} W_{i,\ell'} W_{i,\ell''} \partial_{i, \ell}\partial_{ i, \ell'} \partial_{ i, \ell''}f(\mathbf{Z}_{i,j}^{0}+tW^{(\set_{j})}_{i})dt \right]\right|,
   \end{align*}
   where 
   \begin{align*}
       \mathbf{Z}_{i,j}^{0}&:=\left(X_{1},\cdots,X_{i-1},(W_{i}-X_{i}^{(\set_{j})}+\sum_{\ell=1}^{j}(X_{i}^{(\set_{\ell})}-W_{i}^{(\set_{\ell})})),W_{i+1},\cdots,W_{n}\right)^{\top}.
   \end{align*}
\end{lemma}
This lemma is technically important for our result. Specifically, we bound the difference $|E[f(\mathbf{X})]-E[f(\mathbf{W})]|$ by the Taylor expansion, then we can cancel out interactions between variables belonging to different cells $B_j$ using the block dependence. Thus, if the cells $B_j$ are not excessively large, we can properly control for the effects of the dependencies.

 This result indicates that the difference in the distributions of $\mathbf{X}$ and $\mathbf{W}$ is evaluated based on the size of \( B_j \) of the block dependencies. Since there is a triple summation over \( B_j \) within the upper bound of the lemma, the cubic value \( |B_j|^3 \) becomes critical. If we have \( |B_j| = 1 \), which  corresponds to the case where all elements of \( X_i \) and \( W_i \) are independent, the expectation terms of the upper bound remain bounded. 
In the following sections, it becomes crucial to assess  \( \sum_j |B_j|^m \) with $m \geq 0$ so that this upper bound does not diverge.

\subsubsection{Bounding sum of block size}
In this section, we provide an upper bound for the sum of the powers of the size of cells $\sum_{j=1}^{k}|\set_{j}|^{m}$. This bound is useful when evaluating the sum of the cubic values $|B_j|^3$ introduced in the previous section. The result and the proof are presented below:
\begin{lemma}\label{dependent_group_sum}
    Consider the $k$-partition $\{\set_{j} \}_{j=1}^{k}$ of $[p]$, which is a sequence of $k$ cells $\set_{1},...,\set_{k} \subset [p]$ that are disjoint to each other and satisfy $\cup_{j=1}^{k} \set_{j}=[p]$.
    Also, suppose that $\max_{j=1,...,k}|\set_{j}|\leq d$ holds with $1 \leq d \leq p$.
    Then, for any $m\in\mathbb{N}$, we have
\begin{equation}\label{\set_{j}-inequality}
    \sum_{j=1}^{k}|\set_{j}|^{m}\leq  pd^{m-1}.
\end{equation}
\end{lemma}

\begin{proof}
A straightforward calculation gives
\begin{align*}
    \sum_{j=1}^{k}|\set_{j}|^{m}&= \sum_{j=1}^{k}|\set_{j}||\set_{j}|^{m-1} \leq \sum_{j=1}^{k}|\set_{j}|d^{m-1}=pd^{m-1}.
\end{align*}
where the inequality holds from  the assumption $\max_{j=1,...,k}|\set_{j}|\leq d$ and the last equality follows from the property of cells. 
\end{proof}

The important consequence in Lemma \ref{dependent_group_sum} is that the upper bound holds independently of \( k \). In other words, as long as we determine the upper bound \( d \) of the block size, we do not need to care how many blocks are used to divide the \( p \) variables.
The equality holds  when $k=p/d$ with $|B_j| = d$ for all $j = 1,..., k$.

\subsubsection{Unifying these results}
Using Lemma \ref{Lipschitz_higherorder} and Lemma \ref{dependent_group_sum} above, we bound the term $IV$, which describes the difference between the distributions of the minimized risk with the datasets.
The statement holds for $g \in C^3(\R^{n \times p})$.
\begin{proposition}\label{step4}
 Suppose Assumptions \ref{ass::regime},  \ref{ass::loss}, and \ref{ass::moment} hold. 
 Consider two sequences of i.i.d. $\R^p$-valued random vectors $X_1,...,X_n$ and $W_1,...,W_n$. 
   Suppose that each $\{X_{i}\}_{i=1}^{n}$ and $\{W_{i}\}_{i=1}^{n}$ is $d$-block dependent, and $E[X_{i}]=E[W_{i}]$ and $E[X_{i}X_{i}^{\top}]=E[W_{i}W_{i}^{\top}]$ hold. 
 Then, there exists some  $C_{0}=C_{0}(\tau,q,M)>0$ such that for any finite set $\mathcal{W
}\subset[-L_{n},L_{n}]^{p}$ with $L_{n}\geq 1$, any $\rho\in (0,1)$, and any $g\in C^{3}(\mathbb{R}^{n\times p})$ satisfying $\max_{i \in [n], j \in [p]} |E[\partial_{i,j} g(\mathbf{Z}^{0}_{i})]|<\infty$ and $ \max_{i \in [n], j,k \in [p]} |E[\partial_{i,j} \partial_{i,k} g(\mathbf{Z}^{0}_{i})]| < \infty$, we have
   \begin{align*}
    &|E[g(n^{-1}F_{\beta}((\bar{H}_{\loss_{\rho}}(w,\mathbf{X})))_{w\in\mathcal{W
}}]-E[g(n^{-1}F_{\beta}((\bar{H}_{\loss_{\rho}}(w,\mathbf{W})))_{w\in\mathcal{W
}}]| \\
       &\leq C_{0}(\tau,q,M)  K_{g}\max\{1,\beta^{2}\} L_{n}^{\bar{\mathbf{q}}_{0}+3}D^{3}_{\loss}(\rho)Av(\{L^{\mathbf{q}}_{\loss_{i}}\})\sigma_{n}.
\end{align*}
\end{proposition}

By setting the arbitrarily set $\mathcal{W}$ here as $\mS_{n,\delta}$, we develop an upper bound for the term $IV$. 
Note that the definition of $\sigma_n$ is given in the statement of Theorem \ref{main_universality}. 

\subsection{Bound for other terms}

We examine the other remaining items. Although the results here are the same as in \cite{han2023universality}, we explain their outline for the sake of completeness.

\subsubsection{Bound for $I$: smooth approximation}

We bound the term $I$ on the smooth approximation for a loss function $\loss_{i}$. 
From Assumption \ref{ass::loss}, $\|\loss_{i}-\loss_{i;\rho}\|_{\infty}$ is upper bounded by $\mathcal{M}_{\loss}(\rho)$. Hence, for any $\mathbf{Z}\in\{\bfX,\mathbf{W}\}$, we naturally obtain
\begin{align*}
    &\left|E\left[g\left(\min_{w\in \mathcal{S}_{n}}  H_{\loss}(w,\mathbf{Z})\right)\right]-E\left[g\left(\min_{w\in \mathcal{S}_{n}} H_{\loss_{\rho}}(w,\mathbf{Z})\right)\right]\right| \\
    &\leq \|g^{\prime}\|_{\infty}E\left[\left|\min_{w\in \mathcal{S}_{n}}  H_{\loss}(w,\mathbf{Z})-\min_{w\in \mathcal{S}_{n}} H_{\loss_{\rho}}(w,\mathbf{Z})\right|\right] \\
    &\leq K_{g}\mathcal{M}_{\loss}(\rho).
\end{align*}

\subsubsection{Bound for $II$: discretization of parameter space}

We study the term $II$, which explains an effect of approximating the minimization problem $\min_{w\in \mS_{n}} H_{\loss{\rho}}(w,\mathbf{X})$ on the parameter space $\mS_n$ to the minimization problem  $\min_{w\in \mS_{n,\delta}} H_{\loss{\rho}}(w,\mathbf{X})$ on a discretized parameter space $\mS_{n,\delta}$.
To this aim, we obtain the following statement.
\begin{proposition}\label{step2}
    Suppose Assumption \ref{ass::loss} holds. Fix $L_{n}\geq 1$. There exists some $C_{0}=C_{0}(q)>0$, such that for any measurable $\mathcal{S}_{n}\subset[-L_{n},L_{n}]^{p}$ and any $\delta\in(0,1)$, 
 there exists a deterministic finite set $\mS_{n,\delta}$ with $|\mS_{n,\delta}|\leq\lceil2L_{n}/\delta\rceil^{p}$ and the following
hold:
for any $\rho\in(0,\bar{\rho})$, $\mathbf{Z}\in\{\bfX,\mathbf{W}\}$ and $g\in C^{1}(\mathbb{R})$, we have
\begin{align*}
&\left|E\left[g\left(\min_{w\in \mS_{n,\delta}} H_{\loss_{\rho}}(w,\mathbf{Z})\right)\right]-E\left[g\left(\min_{w\in \mathcal{S}_{n}} H_{\loss_{\rho}}(w,\mathbf{Z})\right)\right]\right| \\
&\leq C_{0} \|g'\|_{\infty}\left(\mathcal{D}_{\loss}(\rho)\delta\cdot L_{n}^{q_{0}q_{1}}Av({L_{\loss_{i}}^{q_{1}}})\cdot \mM_X+\mathcal{M}_{\regularizer}(L_{n},\delta)\right),
\end{align*}
where $\mM_X := E[\max_{1\leq i\leq n}\sum_{j=1}^{p}|X_{i,j}|]+E[(\max_{1\leq i\leq n}\sum_{j=1}^{p}|X_{i,j}|)^{q_{0}q_{1}+1}]$.
\end{proposition}
This result controls the error that arises when optimization is performed over a discretized finite set \( \mS_{n,\delta} \) instead of the original parameter set \( \mS_n \). 
Although this result is similar to Equation (4.11) in \cite{han2023universality}, there are several differences.
By taking \(\delta\) sufficiently small, the upper bound of this error can be made arbitrarily small.
The smoothness of a loss function and a regularizer allows us to achieve this result. 
By assigning each value in \(\mathcal{S}_{n; \delta}\) to a loss function, we can represent them as a high-dimensional vector. Let \((H(w, \mathbf{X}))_{w \in \mathcal{S}_{n; \delta}}\) denote this vector with covariate \(\mathbf{X}\). This step plays an important role as reducing the approximation problem to approximating two high-dimensional vectors: \((H(w, \mathbf{X}))_{w \in \mathcal{S}_{n; \delta}}\) and \((H(w, \mathbf{W}))_{w \in \mathcal{S}_{n; \delta}}\).

\subsubsection{Bound for $III$: smooth minimum approximation}
We study the term $III$ to evaluate the approximation error by the smooth minimum function $F_\beta$. 
A straightforward calculation yields an approximation to the minimum function as
\begin{equation*}
    \min_{1\leq j\leq N}x_{j}-\frac{\log N}{\beta}\leq F_{\beta}(x)\leq \min_{1\leq j\leq N}x_{j}, ~ x \in \R^N.
\end{equation*}
By utilizing this property, we obtain the following proposition.
\begin{proposition}\label{step3}
 For any finite set $\mW \subset \R^p$ and $g\in C^{1}(\mathbb{R})$, we have
\begin{equation*}
    \left|E\left[g\left(\min_{w\in \mW}H_{\loss_{\rho}}(w,\mathbf{X})\right)\right]-E\left[g\left[n^{-1}F_{\beta}\left((\bar{H}_{\loss_{\rho}}(w,\mathbf{X}))_{w\in \mW}\right)\right]\right]\right|\leq\frac{\|g^{'}\|_{\infty}\log|W|}{n\beta}.
\end{equation*}
\end{proposition}
Evaluation by the soft minimum function is a useful tool used in high-dimensional analysis, as demonstrated by \citet{chernozhukov2013gaussian}. By exploiting the smoothness of the soft minimum function, we can apply various approximation methods, such as Stein's method and the generalized Lindeberg principle.

\begin{remark}[Connection to high-dimensional Gaussian approximation studies] \label{remark:connection_gaussian_approx}
In this proof, we employ several techniques used in the high-dimensional Gaussian approximation studies \citep{chernozhukov2013gaussian,chernozhukov2014gaussian,chernozhuokov2022improved,lopes2022central,chang2024central}. However, not all techniques from the approximation can be applied directly. There are two main reasons for this.
First, there is a difference in the objects of differentiation when considering neighborhood approximations.
In the approximation studies, one needs to 
differentiate the smoothly approximated maximum of a high-dimensional data vector with respect to the high-dimensional data itself. 
In contrast, in our setting, the loss function minimized with respect to the parameters is differentiated by the values of the high-dimensional data. Since the values being maximized and differentiated differ in this way, we cannot apply the approximation studies in our setup.
Second, the use of Nazarov's inequality plays a crucial role.
While some approximation studies, such as \cite{chang2024central}, need Nazarov's inequality, Nazarov's inequality also does not hold for the design matrix $\mathbf{X}$ or $\mathbf{W}$, since we consider the explicit form of the minimum of the empirical risk function.
Due to these differences, the results from the studies on high-dimensional Gaussian approximations cannot be directly applied.
\end{remark}

\section{Application} \label{sec:application}

In this section, we demonstrate the universality of the distribution of estimators with block dependence. First, we examine the asymptotic behavior of M-estimators under Gaussian
covariates with a general covariance matrix. We prove a convergence result for M-estimators by generalizing the main theorem of \citet{thrampoulidis2018precise}. Then, we apply this extension to estimators with block dependence under specific loss functions.

\subsection{Precise error of regularized M-estimators with general covariance}

 Consider the following regularization problem 
\begin{equation}\label{eq::theta_x_app}
    \hat{\theta}\in\argmin_{\theta\in\R^p}\left\{\mathcal{L}(\mathbf{y}-\mathbf{G}\theta)+\lambda f(\theta)\right\} \quad \text{where}\quad Y_{i}=G_{i}^{\top}\theta_0+\xi_i 
\end{equation}
where we define $\mathcal{L}(\mathbf{y}-\mathbf{G}\theta)=\sum_{i=1}^{n}\loss_{0}(Y_{i}-G_{i}^{\top}\theta)$ where $G_{i}\sim \mathcal{N}(0,\Sigma)$ and  $\xi_i$ is an independent noise variable.
We consider the asymptotic behavior of M-estimators derived from equation \eqref{eq::theta_x_app}. To establish the result, we assume the following:
\begin{assumption}\label{ass::eigen}
There exists $a > 0$ such that for any $n$, we have $a\leq \lambda_{min}(\Sigma)$.

\end{assumption}
\begin{assumption}\label{ass::TAH_1}{(Summary functionals $L(\cdot)$ and $F(\cdot)$)}
The following conditions hold:
\begin{enumerate}
    \item[(a)] \label{ass::TAH_1_L} For any $c\in\mathbb{R}$ and $\tau>0$, we have continuous functions $L:\mathbb{R}\times\mathbb{R}_{+}\rightarrow \mathbb{R}$ such that
    \begin{equation*}
        \frac{1}{n}\{e_{\mathcal{L}}(c\mathbf{g}+\xi;\tau)-\mathcal{L}(\mathbf{\xi})\}\xrightarrow[]{p}L(c,\tau). 
    \end{equation*}
    \item[(b)] \label{ass::TAH_1_F}
    For any $c\in\mathbb{R}$ and $\tau>0$, we have continuous functions $F:\mathbb{R}\times\mathbb{R}_{+}\rightarrow \mathbb{R}$ such that we have
    \begin{equation*}
\frac{1}{p}\{e_{g}(c\mathbf{h}+\Sigma^{1/2}\theta_0;\tau)-g(\Sigma^{1/2}\theta_{0})\}\xrightarrow[]{p}F(c,\tau),
    \end{equation*}
    where we define $g(\mathbf{x})$ as $f(\Sigma^{-1/2}\mathbf{x})$ and $\mathbf{g}$ and $\mathbf{h}$ follow normal distributions, $\mathcal{N}(0,I_{n})$ and $\mathcal{N}(0,I_{p})$, respectively. Note that $g(\Sigma^{1/2}\theta_0)$ is equal to $f(\theta_0)$.
     \item[(c)] \label{ass::TAH_1_complement} At least one of the following holds. There exists constant $C>0$ such that $\|\xi\|_{2}/\sqrt{p}\leq C$ with probability approaching 1 (w.p.a.1), or, $\sup_{\mathbf{v}\in\mathbb{R}^{p}}\sup_{\mathbf{s}\in\partial\mathcal{L}(\mathbf{v})}\|\mathbf{s}\|_{2}<\infty$ for any $p\in\mathbb{N}$.
\end{enumerate}
\end{assumption}
\begin{assumption}\label{ass::TAH_2}{(Properties of  $L(\cdot)$ and $F(\cdot)$, Assumption 2 of \citet{thrampoulidis2018precise})}      We say that Assumption \ref{ass::TAH_2} holds if all the following are true.
\begin{enumerate}
    \item[(a)] $\lim_{\tau\rightarrow0^{+}}F(\tau,\tau)=0$ and $\lim_{c\rightarrow+\infty}\left\{\frac{c^{2}}{2\tau}-F(c,\tau)\right\}=+\infty$ for all $\tau>0$. 
     \item[(b)] $\lim_{\tau\rightarrow0^{+}}L(\alpha,\tau)<+\infty,$ $\lim_{\tau\rightarrow 0^{+}}L(0,\tau)=0$ and $-\infty<L_{2,+}(0,0):=\lim_{\tau\rightarrow 0^{+}}L_{2,+}(0,\tau)\leq0$. 
      \item[(c)] $\frac{1}{n}\mathcal{L}(\xi)\xrightarrow[]{p}L_0\in[0,\infty]$. Also, $L_{0}=-\lim_{\tau\rightarrow+\infty}L(c,\tau)\geq-L(c,\tau^{\prime})$ for all $c\in\mathbb{R},\tau^{\prime}>0$. 
       \item[(d)] If $L_{0}=+\infty$, then $\lim_{\tau\rightarrow+\infty}L(c,\tau)/\tau=0$ for all $c\in\mathbb{R}$. 
\end{enumerate}
\end{assumption}

These assumptions are extensions of those of \citet{thrampoulidis2018precise} under a general covariance $\Sigma$.
Specifically, Assumptions \ref{ass::TAH_1} and \ref{ass::TAH_2} are identical to Assumption 1 and Assumption 2 in \citet{thrampoulidis2018precise} respectively,  except for the part on $F$. Because of the existence of $\Sigma$, we need to rotate the true parameter $\theta_0$ and a regularized function $f$. The rotation alters the structure of $F$, resulting in a different asymptotic behavior of the error $\frac{1}{p}\|\hat{\theta}-\theta_{0}\|^{2}_{\Sigma}$ relative to that in \citet{thrampoulidis2018precise}.

Under Assumptions \ref{ass::moment} to \ref{ass::TAH_2}, we obtain the generalization of 
Theorem 3.1 in \citet{thrampoulidis2018precise}.
\begin{theorem}\label{Theorem_3.1}
    Let Assumptions \ref{ass::moment} to \ref{ass::TAH_2} hold and define $\tau_{0}=\lim_{n\rightarrow\infty}(n/p)$. If the following convex-concave minimax scalar optimization 
    \begin{equation}\label{eq::TAH_main}
        \inf_{\substack{\alpha\geq0 \\ \tau_g>0}}\sup_{\substack{\beta\geq0\\ \tau_h>0}} \mathcal{\mathcal{D}}(\alpha,\tau_g,\beta,\tau_h) \\
        := \frac{\beta \tau_g}{2}+\tau_0 L\left(\alpha,\frac{\tau_g}{\beta}\right)- \frac{\alpha\tau_h}{2}-\frac{\beta^{2}\alpha}{2\tau_h}+\lambda F\left(\frac{\alpha\beta}{\tau_h},\frac{\alpha\lambda}{\tau_h}\right)
    \end{equation}
    has a unique minimizer $\alpha_*$, then we have
    \begin{align}
        \frac{1}{p}\|\hat{\theta}-\theta_{0}\|^{2}_{\Sigma}\xrightarrow[]{p}\alpha^{2}_{*}. \label{eq:limit_tah}
    \end{align}
\end{theorem}
This theorem is a generalization of Theorem 3.1 in \cite{thrampoulidis2018precise}, extending the covariance matrix from the identity matrix to a general matrix $\Sigma$. Especially,  the result of Theorem \ref{Theorem_3.1} coincides with that of Theorem 3.1 in \citet{thrampoulidis2018precise} when $\Sigma=I$ holds.
Specifically, a difference is that the norm $\|\cdot\|_\Sigma$ used in \eqref{eq:limit_tah} is weighted by $\Sigma$, and the function $F(\cdot)$ in \eqref{eq::TAH_main} corresponds to the general $\Sigma$, as introduced in Assumption \ref{ass::TAH_1}.

\subsection{Application to estimators}

In this section, we demonstrate the universality of the distribution of estimators with block dependence under specific loss functions. 

First, we consider estimators under the absolute loss and the Huber loss, as presented in Example \ref{ex:losses}.
We define a robust cost function as
\begin{equation*}
   H^{R}(w,\bfX,\xi):=\frac{1}{n}\sum_{i=1}^{n}\loss_{0}((X_{i}^{\top}w)-\xi_{i})+\frac{\lambda}{2}\|w+\theta_{0}\|^{2}.
\end{equation*}
where $\lambda > 0$ is a coefficient for the regularizer. Let $\hat{w}^{R}_{\bfX}$ denote the minimizer of the cost function with covariates $\mathbf{X}$.
Recall that $\theta_0 \in \R^p$ is the true parameter of the regression model \eqref{eq:linear_model}, and the estimator of the regression model as the minimizer \eqref{def:theta_x} with the setup is given as $ \hat{\theta}_{\mathbf{X}}^R = \hat{w}_{\mathbf{X}}^R + \theta_{0}$.
By substituting a Gaussian design matrix $\bfG$ instead of $\bfX$, we can consider the Gaussian analogue $H^{R}(w,\mathbf{G},\xi)$ as in Section \ref{sec:gaussian_analogy}.

To show the universality, we introduce the following assumptions:

\begin{assumption}[Assumptions for estimators]\label{ass::robust}
The following hold:
\begin{enumerate}
    \item[(a)] $\ell_{0}:\mathbb{R}\rightarrow\mathbb{R}$ is convex with a weak derivative $\ell^{\prime}_{0}$ satisfying $\max\{\ell_{0}(0),\esssup|\ell^{\prime}_{0}|\}\leq L_{0}$ for some $L_{0}>0$. \label{ass::robust::loss}

    \item[(b)] There exists $0<\varepsilon\leq (1/4)$ such that $d=o(n^{{1} / {4}-\varepsilon})$ \label{ass::robust::epsilon}

    \item[(c)] Uniformly in $p$, a rescaled covariate $X^{0}_{i} = \sqrt{n} X_{i}$ satisfies $M_{k^{*}:\mathbf{X}}:=\max_{j\in[p]}E[|X^{0}_{i, j}|^{2\lceil k^{*}/2\rceil}]$ $<\infty$ where $k^{*}>{3} / ({2\varepsilon})$. Also, $\lambda_{\max}(\Sigma)<\infty$ holds. \label{ass::robust::moment}
 \item[(d)] A distribution of $\xi$ is absolutely continuous with respect to the Lebesgue measure.

 \item[(e)] $\theta_{0}$ contains i.i.d. components whose law $\Pi_{0}$ guarantees the existence of moments of any order.
    \end{enumerate}
\end{assumption}
In Assumption \ref{ass::robust}, the condition (a) generalizes the condition previously mentioned for the robustness of the loss. Both the absolute loss and the Huber loss satisfy this condition. By combining the conditions (b) and (c), 
we can explain how the moment condition of the scaled covariates $X_i^0$ is affected by the order of $d$ for the block dependence.

To apply the result of Theorem \ref{main_universality_structure}, we establish the upper bound for the moment of robust estimators. 
\begin{lemma}\label{robus_linf}
    Suppose Assumptions \ref{ass::regime} and \ref{ass::robust} hold.
    Also, suppose that the distribution of $X_1,...,X_n$ is $d$-block dependent. 
     Then, for any $k\geq 2$, we have $K:=K(k)>0$ satisfying 
    \begin{equation*}
        \max_{j\in[p]}E\left[\left|\hat{w}^{R}_{\mathbf{X},j}\right|\right]\leq K\cdot \left\{(L_{0}/\lambda)^{k}(M_{k:\mathbf{X}})^{k/(2\lceil k/2\rceil )}+\|\theta_{0}\|^{k}_{\infty} \right\}.
    \end{equation*}
\end{lemma}

By combining the result with Theorem \ref{main_universality_structure}, we obtain the following universality result on the distribution of the estimator $\hat{w}^{R}_{\bfX}$:
\begin{theorem}\label{universality_robust}
    Suppose Assumptions \ref{ass::regime} and \ref{ass::robust} hold.
    Also, suppose that the distribution of $X_1,...,X_n$ is $d$-block dependent. Fix a measurable subset $\mathcal{S}_{n}\subset\mathbb{R}^{p}$. Suppose that there exist $z\in\mathbb{R}, \rho_{0}>0$,  and $\varepsilon_n\in[0,1/4)$ such that 
         $\hat{w}_{\mathbf{G}}$ does not satisfy the specific structure $\mathcal{S}_{n}$ in the sense that 
         \begin{equation*}
             \max\left\{\Pr\left(\min_{w\in \mathbb{R}^{p}} H^{R}(w,\mathbf{G},\xi)\geq z+\rho_{0}\right), \Pr\left(\min_{w\in \mathcal{S}_{n}}H^{R}(w,\mathbf{G},\xi)\leq z+2\rho_{0}\right) \right\}\leq \varepsilon_{n}.
         \end{equation*}
Then, $\hat{w}_{\mathbf{X}}$ also does not satisfy the specific structure $\mathcal{S}_{n}$ with high probability, i.e. there exists some  $K=K(\lambda,\tau,M_{k^{*}:\mathbf{X}},k^{*}, \varepsilon,L_{0})>0$ such that it holds that 
\begin{equation*}
    \Pr(\hat{w}^{R}_{\mathbf{X}}\in\mathcal{S}_{n})\leq 4\varepsilon_{n}+K(1+\|\theta_{0}\|^{k^{*}}_{\infty}+\rho^{-3}_{0})n^{-{1}/  {45}}.
\end{equation*}
\end{theorem}

With Theorem \ref{main_universality_structure}  for the universality, Theorem \ref{Theorem_3.1} and Theorem \ref{universality_robust} derive the convergence result of estimation error with non-Gaussian covariates as follows:
\begin{theorem}[Robust estimators]\label{universality_robust_CGMT}
   Suppose that Assumptions \ref{ass::regime},  \ref{ass::eigen}, \ref{ass::TAH_1}-(b) \ref{ass::TAH_2}-(a), and \ref{ass::robust} hold. Define $\tau_{0}=\lim_{n\rightarrow\infty}(n/p)$. 
   Moreover,  we assume either (a) $\ell_{0}$ does not have a derivative at a certain point, or (b) $\ell_{0}$
 includes an interval where $\psi_{0}$ has a strictly increasing derivative.
   Then, if convex-concave minimax scalar optimization \eqref{eq::TAH_main} has
     a unique minimizer $\alpha^{*}$, we have 
   \begin{equation*}
p^{-1} {\|\hat{\theta}^{R}_{\mathbf{X}}-\theta_{0}\|_{\Sigma}^2} \overset{p}{\rightarrow} \alpha_{*}^{2}.
   \end{equation*}
\end{theorem}
This result demonstrates that the error of the estimator $\widehat{\theta}^{R}_{\mathbf{X}}$ from non-Gaussian covariates converges to the known value $\alpha^{2}_{*}$, which is given as the solution of the convex-concave minimax scalar optimization defined in \eqref{eq::TAH_main}.
Our result shows that this universality of the estimation error also holds when the independence assumption of covariates $X_i$ is relaxed.

Second, we apply our main result to the Lasso and Ridge regressions under the squared loss, as examined in \cite{han2023universality}. 
Let $\widehat{w}^{Ri}$ and $\widehat{w}^{L}$ denote Ridge and Lasso estimators, respectively. Moreover, define $f^{Ri}$ and $f^{L}$ denote $\|\cdot\|_2$ and $\|\cdot\|_1$, respectively.  
\begin{theorem}[Ridge/ Lasso estimator]\label{universality_Ridge_Lasso_CGMT}
   Suppose that Assumptions \ref{ass::regime},  \ref{ass::moment}, \ref{ass::eigen}, \ref{ass::TAH_1}-(b), \ref{ass::TAH_2}-(a), and \ref{ass::robust}-(d)(e) hold  and  define $\tau_{0}=\lim_{n\rightarrow\infty}(n/p)$. 
   Moreover, we assume the second moment of $\xi$ is finite, $d$ increases so that $L_n^{7}(\log n)^{2}\sigma_{n}=o(1)$ holds as $n \to \infty$, and 
 there exist $L_{n}\geq 1$ and $\varepsilon_n\in[0,1/4)$ such that both $\|\hat{w}^{A}_{\mathbf{G}}\|_{\infty}$ and $\|\hat{w}^{A}_{\mathbf{X}}\|_{\infty}$ increase moderately in the sense that for each $A\in\{Ri, L\}$, we have
         \begin{equation*}
             \max\left\{\Pr(\|\hat{w}^{A}_{\mathbf{G}}\|_{\infty}>L_{n}), \Pr(\|\hat{w}^{A}_{\mathbf{X}}\|_{\infty}>L_{n})\right\}\leq \varepsilon_{n}.
         \end{equation*}
   Let $\alpha^{Ri}_{*}$ and $\alpha^{L}_{*}$ denote the unique minimizers of convex-concave minimax scalar optimization \eqref{eq::TAH_main} with $F^{Ri}$ and $F^{L}$, respectively. Then, for each $A\in\{Ri, L\}$, we have 
   \begin{equation*}
p^{-1} {\|\hat{\theta}^{A}_{\mathbf{X}}-\theta_{0}\|_{\Sigma}^2} \overset{p}{\rightarrow} (\alpha_{*}^A)^{2}.
   \end{equation*}
\end{theorem}
A key point here is whether the upper bound for the estimator holds, as in Lemma \ref{robus_linf} in the setup in this section. 
In \cite{han2023universality}, the upper bound is derived under the assumption that each element of $X_i$ is independent, so the bound does not apply directly to our case. 
Instead, we can utilize the upper bound under dependent covariates, as in Lemma 7.3 of \cite{lahiry2023universality}, which allows us to show the universality of the Lasso and Ridge regressions.

\section{Numerical experiments}

We numerically investigate the universality of the distribution of the estimation error under dependent data. 
In particular, the cell size $d$ has the same order as the parameter dimension $p$.

We generate synthetic data from the linear regression model \eqref{eq:linear_model}.
To introduce dependence between the features, we partition the \(p\) elements of inputs into cells of size \( d \). Each block is generated independently, but features within each cell have a correlation \(\rho = 0.5\). Formally, we partition the $p$ elements of the full vector $X_i \in \R^p$ into blocks of size $d$: within the $b$-th block (\(b=1,\dots,p/d\)), we have $X_i^{(b)} \in \R^d$.
Then, $X_i^{(b)}$ is generated from a multivariate distribution with a covariance matrix \(\Sigma_{\text{block}}\), defined by
\(
\Sigma_{\text{block}} = \rho \mathbf{1}\mathbf{1}^{\top} + (1-\rho)I_d.
\)
We consider two types of distributions for generating each block of features:
(i) A Gaussian distribution (multivariate normal) with mean zero and covariance matrix \(\Sigma_{\text{block}}\), and (ii) A multivariate \(t\)-distribution with degrees of freedom \(\nu = 5\), standardized to have covariance matrix \(\Sigma_{\text{block}}\).
We set $(n,p) \in \{(100,50),(200,100)\},$ and $ d \in \{5,10,25,50\}$ and investigate cases where the cell size $d$ is as large as the dimension $p$.

We evaluate and compare the following four estimators for $\theta_0$: (i) {the ordinary least squares estimator (OLS)}, which minimizes the squared loss function without regularization, (ii) the {Ridge estimator}, which minimizes the squared loss function with an $R_0(w) = \|\cdot\|_2^2$, (iii) the {Lasso estimator} which minimizes the squared loss function with an $R_0(w) = 0.1\|\cdot\|_1$, and (iv) the {Huber estimator}, with its parameter \(\eta = 1.35\).
For each scenario, we generate synthetic datasets repeatedly (\(100\) trials) from both Gaussian and \(t\)-distributed designs and compute the norms of the estimation errors $\|\hat{\theta} - \theta_0\|_\Sigma^2$ for each estimator.

Figure~\ref{fig:qqplots_n100} and \ref{fig:qqplots_n200} show 
quantile-quantile plots, which plot quantiles of estimation error norms obtained under Gaussian-generated data (horizontal axis) against those obtained under \(t\)-distributed data (vertical axis). In all tested scenarios, points closely align with the 45-degree line, visually supporting the universality even if $d$ has the same order as $p$.

\begin{figure}[h!]
    \centering
    \includegraphics[width=0.83\linewidth]{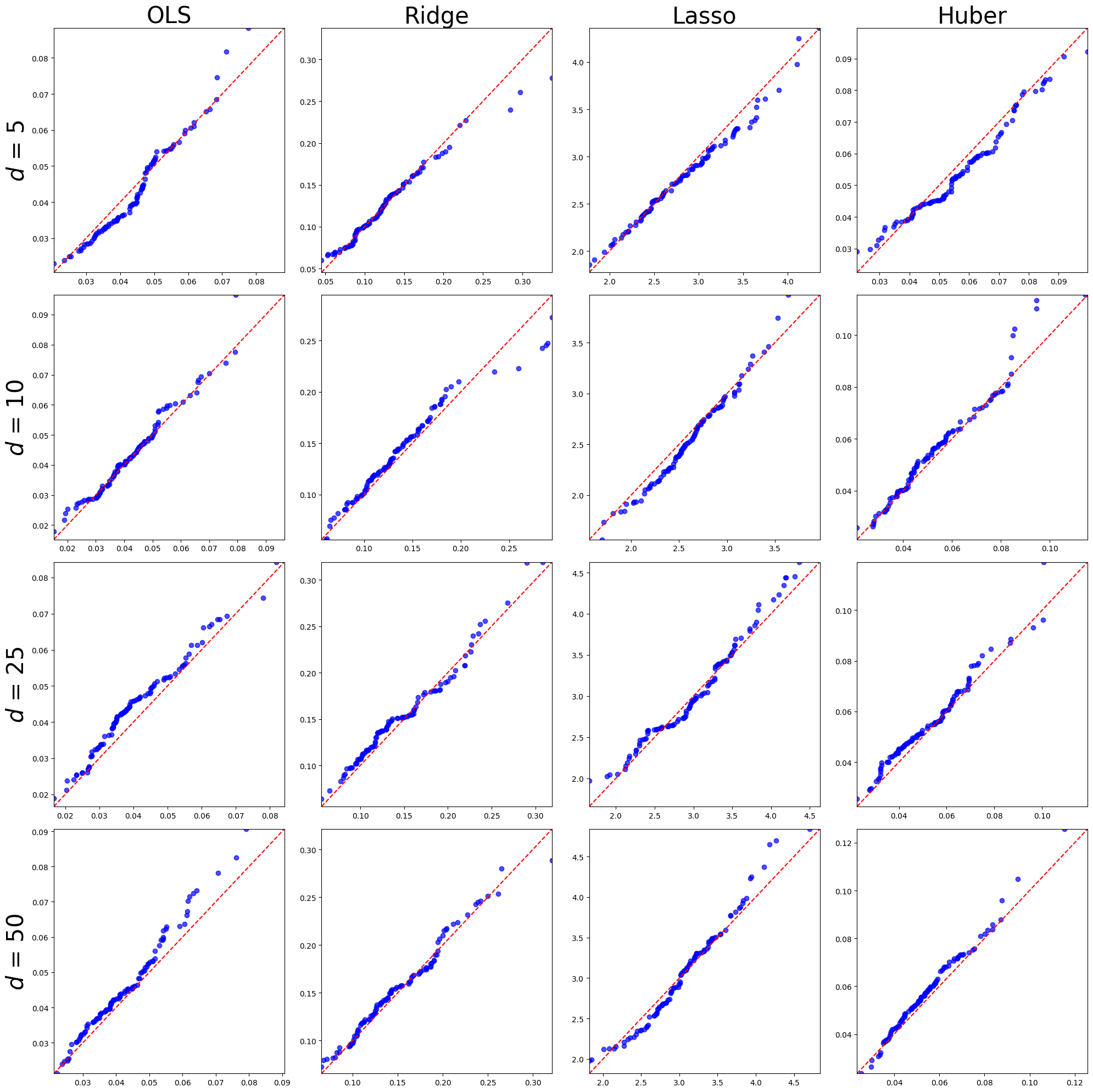}
    \caption{Q-Q plots of the estimation errors. 
    The horizontal axis is for the Gaussian design, and the vertical axis is for the $t$‑distributed design.
    We set $n=100,p=50$.}
    \label{fig:qqplots_n100}
\end{figure}

\begin{figure}[h!]
    \centering
    \includegraphics[width=0.83\linewidth]{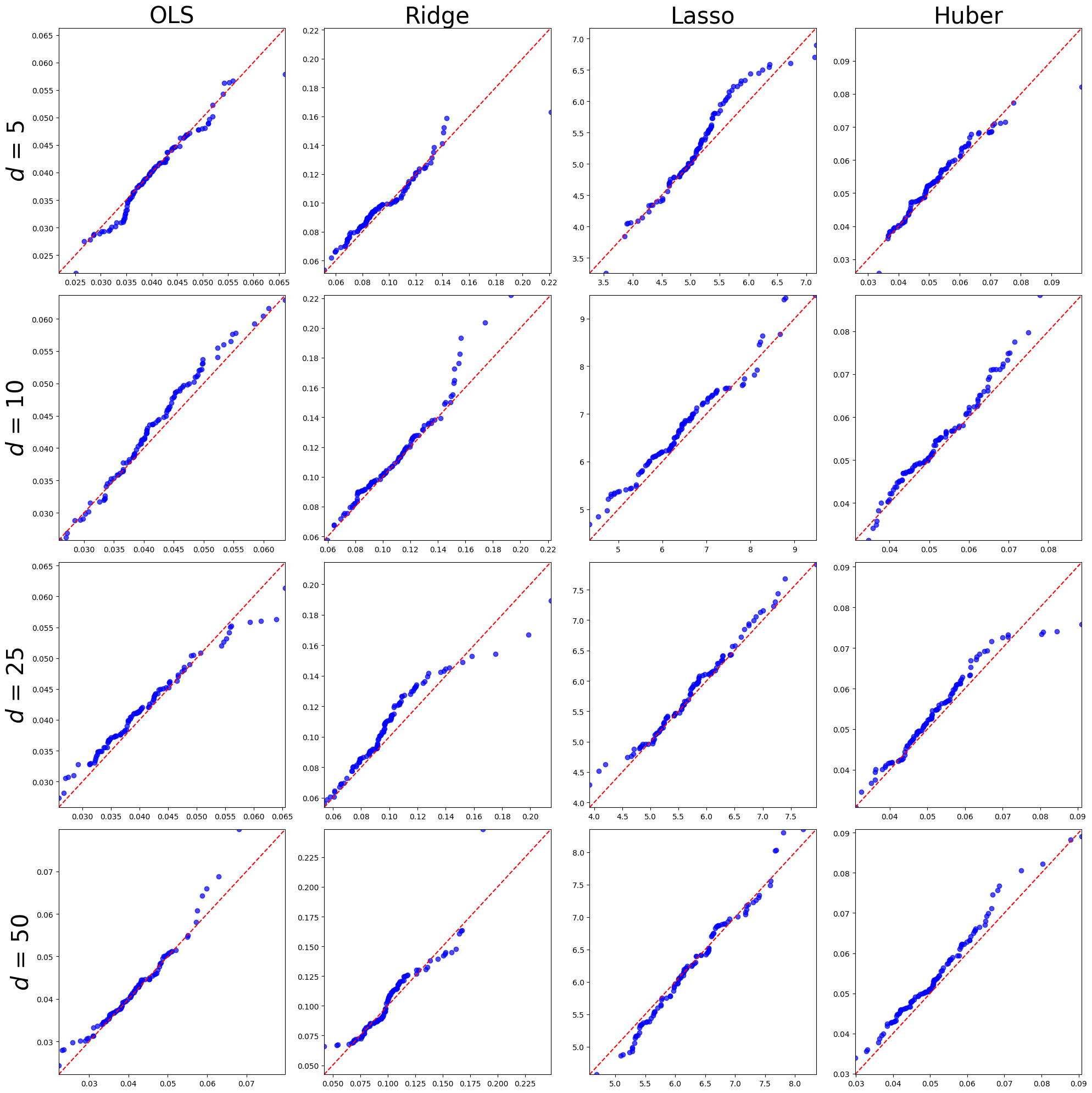}
    \caption{Q-Q plots of the estimation errors. The horizontal axis is for the Gaussian design, and the vertical axis is for the $t$‑distributed design. We set $n=200,p=100$.} 
    \label{fig:qqplots_n200}
\end{figure}

\section{Conclusion}

We derive the universality of the estimator for linear models under the assumption that the distribution of covariates exhibits the block dependence. While the properties of estimators in many high-dimensional models are typically identified under the assumption of Gaussian covariates, extending these results to settings with non-Gaussian covariates remains an important challenge. While the existing studies have required independence among each element of the covariates, this study relaxes that assumption and demonstrates that universality holds even under the block dependence of covariates of a certain order. This allows us to precisely identify the robust estimator for non-Gaussian covariates with dependence among elements.

\section*{Acknowledgments}

The authors would like to thank the anonymous referees, an Associate Editor and the Editor for their constructive comments that improved the quality of this paper.

\section*{Funding}

Toshiki Tsuda was supported by Grant-in-Aid for JSPS Fellows (23KJ0713). 
Masaaki Imaizumi was supported by JSPS KAKENHI (24K02904), JST CREST (JPMJCR21D2),  JST FOREST (JPMJFR216I), and JST BOOST (JPMJBY24A9).

\appendix

\section{Proof of main results}

\subsection{Proof of Theorem \ref{main_universality}}

\begin{proof}

From Proposition \ref{step4}, \ref{step2}, and \ref{step3}, for any fixed $\delta\in (0,1)$,  we obtain
\begin{flalign*}
& \left|E\left[g\left(\min_{w\in \mathcal{S}_{n}} H_{\loss}(w,\mathbf{X})\right)\right]-E\left[g\left(\min_{w\in \mathcal{S}_{n}} H_{\loss}(w,\mathbf{W})\right)\right]\right|&
\end{flalign*}
\begin{align*}
    &\lesssim_{\tau,q,M} K_{g}&&\Biggl[\mathcal{M}_{\loss}(\rho)+\mathcal{D}^{3}_{\loss}(\rho)  \\
    & &&\times \left\{L_{n}^{q_{0}q_{1}}Av({L_{\loss_{i}}^{q_{1}}})\cdot\max_{Z_{ij}\in\{X_{ij},W_{ij}\}}\left(E\left[\max_{1\leq i\leq n}\sum_{j=1}^{p}|Z_{i,j}|\right]+E\left[\left(\max_{1\leq i\leq n}\sum_{j=1}^{p}|Z_{i,j}|\right)^{q_{0}q_{1}+1}\right]\right)\delta\right. \\
    & && \quad \left.+\mathcal{M}_{\regularizer}(L_{n},\delta)+\beta^{-1}\log(L_{n}/\delta)+\max\{1,\beta^{2}\} L_{n}^{\bar{\mathbf{q}}_{0}+3}Av(\{L^{\mathbf{q}}_{\loss_{i}}\})\sigma_{n}\right\} \biggr].
\end{align*}
For each $\mathbf{Z}\in\{\mathbf{X}, \mathbf{W}\}$, we have
\begin{align*}
    E\left[\left(\max_{1\leq i\leq n}\sum_{j=1}^{p}|Z_{i,j}|\right)^{q_{0}q_{1}+1}\right]=&\left(\left\|\max_{1\leq i\leq n}\sum_{j=1}^{p}|Z_{i,j}|\right\|_{q_{0}q_{1}+1}\right)^{q_{0}q_{1}+1} \\
    \leq&n\left(\max_{1\leq i\leq n}\left\|\sum_{j=1}^{p}|Z_{i,j}|\right\|_{q_{0}q_{1}+1}\right)^{q_{0}q_{1}+1} \\
    \leq &n\left(\left\|\sum_{j=1}^{p}|Z_{i,j}|\right\|_{q_{0}q_{1}+1}\right)^{q_{0}q_{1}+1} \\
    \leq&n\left(p\max_{1\leq j\leq p}\left\|Z_{i,j}\right\|_{q_{0}q_{1}+1}\right)^{q_{0}q_{1}+1}, 
\end{align*}
where the first inequality follows from Lemma 2.2.2 in \citet{vaart2023empirical}. For the term $\max_{1\leq j\leq p}\|Z_{i,j}\|_{q_{0}q_{1}+1}$, from Assumption \ref{ass::moment}, we have
\begin{equation*}
    \max_{1\leq j\leq p}\left\|Z_{i,j}\right\|_{q_{0}q_{1}+1}\lesssim_{q,M} \frac{1}{\sqrt{n}} 
\end{equation*}
Thus, we have
\begin{equation*}
    E\left[\left(\max_{1\leq i\leq n}\sum_{j=1}^{p}|Z_{i,j}|\right)^{q_{0}q_{1}+1}\right]\lesssim_{\tau,q,M}n^{\frac{q_{1}q_{0}+3}{2}}.
\end{equation*}
Hence, by $q_0 q_1 >0$, we also obtain
\begin{align*}
    &E\left[\max_{1\leq i\leq n}\sum_{j=1}^{p}|Z_{i,j}|\right]+E\left[\left(\max_{1\leq i\leq n}\sum_{j=1}^{p}|Z_{i,j}|\right)^{q_{0}q_{1}+1}\right] 
    \lesssim_{\tau,q,M} n^{(q_{0}q_{1}+3)/2}.
\end{align*}
Therefore, by setting $0<\delta\leq n^{-(q_{0}q_{1}+3)/2}\sigma_{n}$, the upper bound becomes
\begin{equation*}
    \mathcal{M}_{\loss}(\rho)+\mathcal{D}^{3}_{\loss}(\rho)\left\{\mathcal{M}_{\regularizer}(L_{n},\delta)+\beta^{-1}\log(L_{n}/\delta)+\max\{1,\beta^{2}\} L_{n}^{\bar{\mathbf{q}}_{0}+3}Av(\{L^{\mathbf{q}}_{\loss_{i}}\})\sigma_{n}\right\}.
\end{equation*}
By minimizing the above term in terms of $\beta$, we obtain the required result.

For the second claim, take any function $g\in C^{3}(\mathbb{R})$ such that 
\begin{equation*}
    g(x):=\begin{cases}
    0 & \text{if $x\leq 1$,} \\
    2 & \text{if $x>3$.}
    \end{cases}
\end{equation*} 
With $\varphi(\mathbf{X}) :=\min_{w\in \mS_{n}} H_{\loss}(w, \mathbf{X})$, for any $z > 0, \varepsilon>0$, we have
\begin{align*}
    \Pr(\varphi(\mathbf{X})-z>3\varepsilon)\leq &\Pr[g((\varphi(\mathbf{X})-z)/\varepsilon)\geq2] \\
    \leq &\frac{1}{2}E[g((\varphi(\mathbf{X})-z)/\varepsilon)] \\
    \leq &\frac{1}{2}E[g((\varphi(\mathbf{W})-z)/\varepsilon)]+C\cdot\max\{1,\varepsilon^{-3}\}r_{R}(L_{n}) \\
    \leq &\Pr(\varphi(\mathbf{W})-z)>\varepsilon)+C\cdot\max\{1,\varepsilon^{-3}\}r_{R}(L_{n}).
\end{align*}

\end{proof}

\subsection{Proof of Theorem \ref{main_universality_structure}}

\begin{proof}
    From the condition (1), for the event $A:=\{\hat{w}_{\mathbf{X}}\in[-L_{n},L_{n}]^{p}\}$, we have
    $\Pr (A^{c})\leq \varepsilon_{n}$. Because the event
    \begin{equation*}
        \left\{\min_{\|w\|_{\infty}\leq L_{n}}H_{\loss}(w,\mathbf{X})<z+3\rho_{0}\right\}\bigcap\left\{\min_{w\in\mathcal{S}_{n}\cap [-L_{n}, L_{n}]^{p}}H_{\loss}(w,\mathbf{X})>z+6\rho_{0}\right\}
    \end{equation*}
    is included in the event $\{\hat{w}_{\mathbf{X}}\notin \mathcal{S}_{n}\cap [-L_{n}, L_{n}]^{p}\}$, we have
    \begin{align*}
        \Pr(\hat{w}_{\mathbf{X}}\in\mathcal{S}_{n})&\leq \Pr(\hat{w}_{\mathbf{X}}\in\mathcal{S}_{n}\cap [-L_{n}, L_{n}]^{p})+\Pr(A^{c}) \\
        &\leq \Pr\left(\min_{\|w\|_{\infty}\leq L_{n}}H_{\loss}(w,\mathbf{X})\geq z+3\rho_{0}\right)+\Pr\left(\min_{w\in\mathcal{S}_{n}\cap [-L_{n}, L_{n}]^{p}}H_{\loss}(w,\mathbf{X})\leq z+6\rho_{0}\right)+\varepsilon_{n}.
    \end{align*}
For the first term, by Theorem \ref{main_universality}, we obtain
\begin{align*}
    &\Pr\left(\min_{\|w\|_{\infty}\leq L_{n}}H_{\loss}(w,\mathbf{X})\geq z+3\rho_{0}\right) \\
    &\leq\Pr\left(\min_{\|w\|_{\infty}\leq L_{n}}H_{\loss}(w,\mathbf{W})\geq z+3\rho_{0}\right)+C_{1}(\max\{1,\rho_{0}^{-3}\})r_{R}(L_{n}) \\
    &\leq \Pr\left(\min_{w\in \mathbb{R}^{p}}H_{\loss}(w,\mathbf{W})\geq z+3\rho_{0}\right)+\varepsilon_{n}+C_{1}(\max\{1,\rho_{0}^{-3}\})r_{R}(L_{n}) \\
    &\leq 2\varepsilon_{n}+C_{1}(\max\{1,\rho_{0}^{-3}\})r_{R}(L_{n}),
\end{align*}
where the last two inequalities hold by assumption. For the second term, we have
\begin{align*}
    \Pr\left(\min_{w\in\mathcal{S}_{n}\cap [-L_{n}, L_{n}]^{p}}H_{\loss}(w,\mathbf{X})\leq z+6\rho_{0}\right)\leq &\Pr\left(\min_{w\in\mathcal{S}_{n}}H_{\loss}(w,\mathbf{X})\leq z+6\rho_{0}\right) \\
    \leq&\Pr\left(\min_{w\in\mathcal{S}_{n}}H_{\loss}(w,\mathbf{W})\leq z+18\rho_{0}\right)+C_{1}(\max\{1,\rho_{0}^{-3}\})r_{R}(L_{n}) \\
    \leq&\varepsilon_{n} +C_{1}(\max\{1,\rho_{0}^{-3}\})r_{R}(L_{n}),
\end{align*}
where the second last inequality follows from Theorem \ref{main_universality}. The last inequality holds from the assumption. Therefore, the required result follows.
\end{proof}

\subsection{Proof of Corollary \ref{cor:universality_risk} and Corollary \ref{cor:corollary_distribution}}

\begin{proof}
Whichever corollary we prove, we verify that $r_R(L_n)$ converges to zero with the setup.
For the minimization in terms of $\delta \in (0,\omega_n)$, we take $\delta = \delta_n \asymp n^{-c_1}$ as $n \to \infty$ with some $c_1 > 1$.
Since $L_n$ has an order $O(n \log^\kappa n)$, we have $\log^{2/3}_{+}(L_{n}/\delta_n) =O(\log^{2/3} n) = o(\log n)$.
In addition, we have $\mathcal{M}_{\regularizer}(L_{n},\delta_n) = o(1)$ with $\delta_n \searrow 0$.
Here, we apply the locally Lipschitz or pseudo-Lipschitz property of the specified regularizer to state this result.
We set $d=o(n^\mu/(\log n)^{d^{*}})$ with some $\mu, d^* \in \R$ and apply the definition of $\sigma_n$ in \eqref{def:sigma}, then obtain 
\begin{align*}
    &\log^{2/3}_{+}(L_{n}/\delta_n) L_{n}^{(\bar{\mathbf{q}}_{0}+3)/{3}}\sigma^{1/3}_{n}\\
    &=o(\log n)O(n^{ (\bar{\mathbf{q}}_{0}+3)\alpha/{3}}(\log n)^{(\bar{\mathbf{q}}_{0}+3)\kappa/{3}})o(n^{2\mu/3}/(\log n)^{2d^{*}/3})n^{-1/6} \\
    &=o(n^{(\bar{\mathbf{q}}_{0}+3)\alpha/{3} +  2\mu/3 -1/6 }(\log n)^{1-2d^{*}/3+(\bar{\mathbf{q}}_{0}/3+1)\kappa}).
\end{align*}
Therefore, we obtain $\log^{2/3}_{+}(L_{n}/\delta_n) L_{n}^{(\bar{\mathbf{q}}_{0}+3)/{3}}\sigma^{1/3}_{n} = o(1)$, when $\mu$ is no more than $1/4 - (\bar{\mathbf{q}}_{0} + 3) \alpha / 2$ and $d^{*}$ is greater than or equal to $3/2+1/2(\bar{\mathbf{q}}_{0}+3)\kappa$.
Using the result, we obtain
\begin{equation*}
    r_R(L_n) \lesssim  \inf_{\rho\in(0,\bar{\rho})}\left\{\mathcal{M}_{\loss}(\rho)+\mathcal{D}^{3}_{\loss}(\rho)\cdot \zeta_n\right\},
\end{equation*}
with some sequence $\zeta_n = o(1)$.

In the squared loss case, we obtain $\mathcal{M}_{\loss}(\rho) = 0$ and $\mathcal{D}_{\loss}(\rho) \leq C$ with some $C > 0$, hence we obtain the following as $n \to \infty$:
\begin{align}
    r_R(L_n) = \zeta_n =o(1).
\end{align}

In the absolute loss and the Huber loss cases, we have $\mathcal{M}_{\loss}(\rho) \lesssim \rho$ and $\mathcal{D}_{\loss}(\rho) \lesssim \rho^{-2}$, which implies
\begin{align}
    r_R(L_n) \lesssim  \inf_{\rho\in(0,\bar{\rho})}\left\{\rho+\rho^{-6} \cdot \zeta_n\right\}.
\end{align}
Then, at the minimization in $\rho$, we take $\rho = \rho_n =\zeta_n^{1/7}$, then we obtain
\begin{align}
    r_R(L_n) \lesssim \zeta_n^{1/7} = o(1).
\end{align}
Finally, we obtain the statement.
\end{proof}

\section{Proof of lemmas and propositions}

In this section, we provide several proofs.
We remind some notation. 
For a matrix $A \in \R^{d_1 \times d_2}$ and a function $h: \R^{d_1 \times d_2} \to \R$, $\partial_{j_1, j_2} h(A) := \partial h(A)  / \partial {A_{j_1,j_2}}$ denotes a partial derivative in terms of a $(j_1,j_2)$-element $A_{j_1,j_2}$ of $A$.

\subsection{Preliminaries and some auxiliary results}\label{appen:preliminaries}

For later use, we summarize the notations and the derivatives of $g(n^{-1}F_{\beta}(\cdot))$. Set $F_{\beta}(A)$ and $\bar{H}(w,A)$ as follows:
\begin{align*}
        F_{\beta}(A)&:=F_{\beta}((\bar{H}(w,A))_{w\in \mW}) =-\frac{1}{\beta}\log(\sum_{w\in \mW}\exp(-\beta \bar{H}(w,A))),
\end{align*}
and
\begin{align*}
    \bar{H}(w,A)&:=n\times H(w,A) =\sum_{i=1}^{n}\loss_{i}((Aw)_{i})+n\cdot \regularizer (w).
\end{align*}
We introduce the notation $\langle\cdot\rangle_{H_{0}}$ as follows:
\begin{equation*}
    \langle f\rangle_{H_{0}}:=\frac{\sum_{w\in \mW} f(w,A)e^{-H_{0}(w,A)}}{\sum_{w\in \mW}e^{-H_{0}(w,A)}}.
\end{equation*}
Then, for any $g\in C^{3}(\mathbb{R})$, we have
\begin{equation*}
    \partial_{i,j}g(n^{-1}F_{\beta}(A))=n^{-1}\langle\partial_{i,j} \bar{H}\rangle_{\bar{H}_{\beta}}\cdot g^{\prime}(n^{-1}F_{\beta}(A)),
\end{equation*}
and
\begin{align*}
    &\partial_{i,j}\partial_{i,k}g(n^{-1}F_{\beta}(A))\\
    &=n^{-2}\langle\partial_{i,j} \bar{H}\rangle_{\bar{H}_{\beta}}\langle\partial_{i,k} \bar{H}\rangle_{\bar{H}_{\beta}}\cdot g^{\prime\prime}(n^{-1}F_{\beta}(A)) \\
    &\quad +n^{-1}g^{\prime}(n^{-1}F_{\beta}(A))\left(\langle\partial_{i,j}\partial_{i,k} \bar{H}\rangle_{\bar{H}_{\beta}}-\beta\langle\partial_{i,j} \bar{H}\partial_{i,k} \bar{H}\rangle_{\bar{H}_{\beta}}+\beta\langle\partial_{i,j} \bar{H}\rangle_{\bar{H}_{\beta}}\langle\partial_{i,k} \bar{H}\rangle_{\bar{H}_{\beta}} \right).
\end{align*}
where we define  \(\bar{H}_{\beta} := \beta \cdot \bar{H}\) .
Also, we obtain the higher-order derivative as
\begin{align*}
    &\partial_{i,j} \partial_{i,k} \partial_{i,\ell}g(n^{-1}F_{\beta}(A))\\
    &=n^{-3} \langle\partial_{i,j} \bar{H}\rangle_{\bar{H}_{\beta}}\langle\partial_{i,k} \bar{H}\rangle_{\bar{H}_{\beta}}\langle\partial_{i,\ell} \bar{H}\rangle_{\bar{H}_{\beta}}\cdot g^{\prime\prime\prime}(n^{-1}F_{\beta}(A)) \\
    &~+n^{-2}\langle\partial_{i,k} \bar{H}\rangle_{\bar{H}_{\beta}}\cdot g^{\prime\prime}(n^{-1}F_{\beta}(A))\left(\langle\partial_{i,j}\partial_{i,\ell} \bar{H}\rangle_{\bar{H}_{\beta}}-\beta\langle\partial_{i,j} \bar{H}\partial_{i,\ell} \bar{H}\rangle_{\bar{H}_{\beta}}+\beta \langle\partial_{i,j} \bar{H}\rangle_{\bar{H}_{\beta}}\langle\partial_{i,\ell} \bar{H}\rangle_{\bar{H}_{\beta}}\right) \\
    &~+n^{-2}\langle\partial_{i,j} \bar{H}\rangle_{\bar{H}_{\beta}}\cdot g^{\prime\prime}(n^{-1}F_{\beta}(A))\left(\langle\partial_{i,k}\partial_{i,\ell} \bar{H}\rangle_{\bar{H}_{\beta}}-\beta\langle\partial_{i,k} \bar{H}\partial_{i,\ell} \bar{H}\rangle_{\bar{H}_{\beta}}+\beta \langle\partial_{i,k} \bar{H}\rangle_{\bar{H}_{\beta}}\langle\partial_{i,\ell} \bar{H}\rangle_{\bar{H}_{\beta}}\right) \\
    &~+n^{-2}\langle\partial_{i,\ell} \bar{H}\rangle_{\bar{H}_{\beta}}\cdot g^{\prime\prime}(n^{-1}F_{\beta}(A))\left(\langle\partial_{i,j}\partial_{i,k} \bar{H}\rangle_{\bar{H}_{\beta}}-\beta\langle\partial_{i,j} \bar{H}\partial_{i,k} \bar{H}\rangle_{\bar{H}_{\beta}}+\beta\langle\partial_{i,j} \bar{H}\rangle_{\bar{H}_{\beta}}\langle\partial_{i,k} \bar{H}\rangle_{\bar{H}_{\beta}} \right) \\
    &~+n^{-1}g^{\prime}(n^{-1}F_{\beta}(A)) \\
    &~\times\left\{\langle\partial_{i,j}\partial_{i,k}\partial_{i,\ell} \bar{H}\rangle_{\bar{H}_{\beta}}-\beta\langle\partial_{i,j}\partial_{i,k} \bar{H}\partial_{i,\ell}  \bar{H}\rangle_{\bar{H}_{\beta}}+\beta\langle\partial_{i,j}\partial_{i,k} \bar{H}\rangle_{\bar{H}_{\beta}}\langle\partial_{i,\ell} \bar{H}\rangle_{\bar{H}_{\beta}}\right.
     \\
    &~-\beta\left(\langle\partial_{i,j}\partial_{i,\ell} \bar{H}\partial_{i,k} \bar{H}\rangle_{\bar{H}_{\beta}}+\langle\partial_{i,j} \bar{H}\partial_{i,k}\partial_{i,\ell} \bar{H}\rangle_{\bar{H}_{\beta}}-\beta\langle\partial_{i,j} \bar{H}\partial_{i,k} \bar{H}\partial_{i,\ell} \bar{H}\rangle_{\bar{H}_{\beta}}\right. \\
    &~\left.+\beta\langle\partial_{i,j} \bar{H}\partial_{i,k} \bar{H}\rangle_{\bar{H}_{\beta}}\langle\partial_{i,\ell} \bar{H}\rangle_{\bar{H}_{\beta}}\right) \\
    &+\beta\langle\partial_{i,j} \bar{H}\rangle_{\bar{H}_{\beta}}\left(\langle\partial_{i,k}\partial_{i,\ell} \bar{H}\rangle_{\bar{H}_{\beta}}-\beta\langle\partial_{i,k} \bar{H}\partial_{i,\ell} \bar{H}\rangle_{\bar{H}_{\beta}}+\beta\langle\partial_{i,k} \bar{H}\rangle_{\bar{H}_{\beta}}\langle\partial_{i,\ell} \bar{H}\rangle_{\bar{H}_{\beta}}\right) \\
    &~+\left.\beta\langle\partial_{i,k} \bar{H}\rangle_{\bar{H}_{\beta}}\left(\langle\partial_{i,j}\partial_{i,\ell} \bar{H}\rangle_{\bar{H}_{\beta}}-\beta\langle\partial_{i,j} \bar{H}\partial_{i,\ell} \bar{H}\rangle_{\bar{H}_{\beta}}+\beta \langle\partial_{i,j} \bar{H}\rangle_{\bar{H}_{\beta}}\langle\partial_{i,\ell} \bar{H}\rangle_{\bar{H}_{\beta}}\right)\right\}.
\end{align*}

\subsection{Proof of Lemma \ref{Lipschitz_higherorder}}

\begin{proof}
$W_{i}^{(\set_{k})}$ is defined in the similar manner of $X_{i}^{(\set_{k})}$. We  define 
\begin{align}
    &\mathbf{Z}_{i,j}:=(X_{1},\cdots,X_{i-1},(W_{i}+\sum_{\ell=1}^{j}(X_{i}^{(\set_{\ell})}-W_{i}^{(\set_{\ell})})),W_{i+1},\cdots,W_{n})^{\top},
\end{align}
and
\begin{align}
    &\mathbf{Z}_{i,j}^{0}:=\left(X_{1},\cdots,X_{i-1},(W_{i}-X_{i}^{(\set_{j})}+\sum_{\ell=1}^{j}(X_{i}^{(\set_{\ell})}-W_{i}^{(\set_{\ell})})),W_{i+1},\cdots,W_{n}\right)^{\top}.
\end{align}
Especially, when $j=0$ holds, $\mathbf{Z}_{i,0}$ becomes $(X_{1},\cdots,X_{i-1},W_{i},W_{i+1},\cdots,W_{n})^{\top}$.
       
By definition, we have 
    \begin{align*}
        E[f(\mathbf{X})]-E[f(\mathbf{W})]=\sum_{i=1}^{n}&\left(\sum_{j=1}^{k}E[(f(\mathbf{Z}_{i,j})-f(\mathbf{Z}^{0}_{i,j}))-(f(\mathbf{Z}_{i,j-1})-f(\mathbf{Z}^{0}_{i,j
        }))]\right) .
    \end{align*}
    By third-order Taylor approximation, we have
    \begin{align*}
f(\mathbf{Z}_{i,j})&=f(\mathbf{Z}^{0}_{i,j})+(X^{(\set_{j})}_{i})^{\top}\nabla_{i} f(\mathbf{Z}_{i,j}^{0})+\frac{1}{2}(X^{(\set_{j})}_{i})^{\top}\nabla^{2}_{i} f(\mathbf{Z}_{i,j}^{0})X^{(\set_{j})}_{i} \\
&~+\frac{1}{2}\int^{1}_{0}(1-t)^2 \sum_{\ell \in B_j} \sum_{\ell' \in B_j} \sum_{\ell'' \in B_j} X_{i,\ell} X_{i,\ell'} X_{i,\ell''}\partial_{i, \ell}\partial_{ i, \ell'} \partial_{ i, \ell''}f(\mathbf{Z}_{i,j}^{0}+tX^{(\set_{j})}_{i})dt,
    \end{align*}
    and similarly,
     \begin{align*}
f(\mathbf{Z}_{i,j-1})&=f(\mathbf{Z}^{0}_{i,j})+(W^{(\set_{j})}_{i})^{\top}\nabla_{i} f(\mathbf{Z}_{i,j}^{0})+\frac{1}{2}(W^{(\set_{j})}_{i})^{\top}\nabla^{2}_{i} f(\mathbf{Z}_{i,j}^{0})W^{(\set_{j})}_{i} \\
&~+\frac{1}{2}\int^{1}_{0}(1-t)^2\sum_{\ell \in B_j} \sum_{\ell' \in B_j} \sum_{\ell'' \in B_j} W_{i,\ell} W_{i,\ell'} W_{i,\ell''} \partial_{i, \ell}\partial_{ i, \ell'} \partial_{ i, \ell''}f(\mathbf{Z}_{i,j}^{0}+tW^{(\set_{j})}_{i})dt,
\end{align*}
    where we define
    \begin{equation*}
        \nabla_{i} f(\mathbf{Z}^{0}_{i,j}):=\begin{bmatrix}
\partial_{i,1}f(\mathbf{Z}^{0}_{i,j})  \\
\vdots  \\
\partial_{i,p}f(\mathbf{Z}^{0}_{i,j})  \\
\end{bmatrix},
\quad
 \nabla^{2}_{i} f(\mathbf{Z}^{0}_{i,j}):=\begin{bmatrix}
\partial^{2}_{i,1}f(\mathbf{Z}^{0}_{i,j}) & \cdots & \partial_{i,1}\partial_{i,p}f(\mathbf{Z}^{0}_{i,j})\\
\vdots &  \ddots & \vdots \\
\partial_{i,1}\partial_{i,p}f(\mathbf{Z}^{0}_{i,j}) & \cdots & \partial^{2}_{i,p}f(\mathbf{Z}^{0}_{i,j}) \\
\end{bmatrix}.
    \end{equation*}
 Therefore, we have
   \begin{align*}
        &E[f(\mathbf{X})]-E[f(\mathbf{W})] \\
       &=\sum_{i=1}^{n}\left(\sum_{j=1}^{k}E[(f(\mathbf{Z}_{i,j})-f(\mathbf{Z}^{0}_{i,j}))-(f(\mathbf{Z}_{i,j-1})-f(\mathbf{Z}^{0}_{i,j
        }))]\right) \\
&=\sum_{i=1}^{n}\sum_{j=1}^{k}\sum_{\ell \in \set_{j}}E[(X_{i,\ell}-W_{i,\ell})\partial_{i,\ell} f(\mathbf{Z}^{0}_{i,j})]  \\
&~+\frac{1}{2}\sum_{i=1}^{n}\sum_{j=1}^{k}\sum_{\ell \in \set_{j}}\sum_{\ell^{\prime} \in \set_{j}}E[(X_{i,\ell}X_{i,\ell^{\prime}}-W_{i,\ell}W_{i, \ell^{\prime}})\partial_{i,\ell}\partial_{i,\ell^{\prime}} f(\mathbf{Z}^{0}_{i,j})]  \\
        &~+\frac{1}{2}\sum_{i=1}^{n}\sum_{j=1}^{k}\int^{1}_{0}(1-t)^2\sum_{\ell \in B_j} \sum_{\ell' \in B_j} \sum_{\ell'' \in B_j} X_{i,\ell} X_{i,\ell'} X_{i,\ell''}\partial_{i, \ell}\partial_{ i, \ell'} \partial_{ i, \ell''}f(\mathbf{Z}_{i,j}^{0}+tX^{(\set_{j})}_{i})dt \\
        &~-\frac{1}{2}\sum_{i=1}^{n}\sum_{j=1}^{k}\int^{1}_{0}(1-t)^2\sum_{\ell \in B_j} \sum_{\ell' \in B_j} \sum_{\ell'' \in B_j} W_{i,\ell} W_{i,\ell'} W_{i,\ell''} \partial_{i, \ell}\partial_{ i, \ell'} \partial_{ i, \ell''}f(\mathbf{Z}_{i,j}^{0}+tW^{(\set_{j})}_{i})dt.
    \end{align*}
    By the assumption on block dependence, for any $j$ and $\ell\in\set_{j}$, we have
    \begin{align*}
        E[(X_{i,\ell}-W_{i,\ell})\partial_{i, \ell} f(\mathbf{Z}^{0}_{i,j})] &=E[E[(X_{i,\ell}-W_{i,\ell})|\mathbf{Z}^{0}_{i,j}]\partial_{i,\ell} f(\mathbf{Z}^{0}_{i,j})] \\
        &=E[E[(X_{i,\ell}-W_{i,\ell})]\partial_{i,\ell} f(\mathbf{Z}^{0}_{i,j})]\\
        &=0.
    \end{align*}
    In the same way, for any $j$ and $\ell, \ell^{\prime}\in\set_{j}$, we have
    \begin{equation*}
      E[(X_{i,\ell}X_{i, \ell^{\prime}}-W_{i,\ell}W_{i,\ell^{\prime}})\partial_{i,\ell}\partial_{i,\ell^{\prime}} f(\mathbf{Z}^{0}_{i,j})] =0.
    \end{equation*}
    Therefore, we have the required result.
\end{proof}

\subsection{Proof of Proposition \ref{step4}}

To derive the result , we introduce two lemmas. First, we establish that the difference in the expectation of certain functions can be upper bounded by moments of their inputs.
\begin{lemma}\label{Lipschitz_higherorder_d-dependent}
   Consider two sequences of i.i.d. $\R^p$-valued random vectors $X_1,...,X_n$ and $W_1,...,W_n$. 
   Suppose that each $\{X_{i}\}_{i=1}^{n}$ and $\{W_{i}\}_{i=1}^{n}$ is $d$-block dependent, and $E[X_{i}]=E[W_{i}]$ and $E[X_{i}X_{i}^{\top}]=E[W_{i}W_{i}^{\top}]$ hold. 
   Let $\mathbf{X}$ and $\mathbf{W}$ denote matrices $(X^{\top}_{1},\cdots,X^{\top}_{n})^{\top}$ and $(W^{\top}_{1},\cdots,W^{\top}_{n})^{\top}$, respectively. 
   Then, there exists some $C_{0}=C_{0}(\tau,q)>0$ such that for any finite set $\mW\subset[-L_{n},L_{n}]^{p}$ with $L_{n}\geq 1$, any $\rho\in (0,1)$, and any $g\in C^{3}(\mathbb{R}^{n\times d})$ satisfying $\max_{i \in [n], m \in [p]}|E[\partial_{i,m} g(\mathbf{Z}^{0}_{i,j})]|<\infty$ and $ \max_{i \in [n], m,m^{\prime} \in [p]}|E[\partial_{i,m} \partial_{i,m^{\prime}} g(\mathbf{Z}^{0}_{i,j})]| < \infty$, we have
   \begin{align*}
          &|E[g(n^{-1}F_{\beta}((\bar{H}_{\loss_{\rho}}(w,\mathbf{X})))_{w\in\mathcal{W
}}]-E[g(n^{-1}F_{\beta}((\bar{H}_{\loss_{\rho}}(w,\mathbf{W})))_{w\in\mathcal{W
}}]| \\
       &\leq C_{0}(\tau,q)K_{g}\max\{1,\beta^{2}\}  L_{n}^{\bar{\mathbf{q}}_{0}+3}D^{3}_{\loss}(\rho) \\
   & \times \left(\frac{1}{n}\sum_{i=1}^{n}\sum_{j=1}^{k}\max_{Z_{i}^{(\set_{j})}\in (X_{i}^{(\set_{j})}, W_{i}^{(\set_{j})})}E\left[\left(\sum_{\ell \in \set_{j}}\left|Z_{i,\ell}\right|\right)^{3}\right]\right. \\
&+\frac{1}{n}\sum_{i=1}^{n}L^{\mathbf{q}}_{\loss_{i}}\sum_{j=1}^{k}\max_{Z_{i}^{(\set_{j})}\in (X_{i}^{(\set_{j})}, W_{i}^{(\set_{j})})}
E\left[\left(\sum_{\ell \in \set_{j}}\left|Z_{i,\ell}\right|\right)^{3}\right] \\
&+\frac{1}{n}\sum_{i=1}^{n}L^{\mathbf{q}}_{\loss_{i}}\sum_{j=1}^{k}\max_{Z_{i}^{(\set_{j})}\in (X_{i}^{(\set_{j})}, W_{i}^{(\set_{j})})}
E\left[\left(\sum_{\ell \in \set_{j}}\left|Z_{i,\ell}\right|\right)^{\bar{\mathbf{q}}_{0}+3}\right] \\
&\left.+\frac{1}{n}\sum_{i=1}^{n}L^{\mathbf{q}}_{\loss_{i}}\sum_{j=1}^{k}\max_{Z_{i}^{(\set_{j})}\in (X_{i}^{(\set_{j})}, W_{i}^{(\set_{j})})}E\left[\left(\sum_{\ell \in \set_{j}}\left|Z_{i,\ell}\right|\right)^{3}\right]\sum_{k_{1},\cdots,k_{\bar{\mathbf{q}}_{0}}\notin \set_{j}} E\left[C^{j}_{ik_{1}}\cdots C^{j}_{ik_{\bar{\mathbf{q}}_{0}}}\right]\right).
\end{align*}
where we set $\tau\leq \frac{n}{p}\leq \frac{1}{\tau}$ and  denote $C^{j}_{i\ell}$ as the $(i,\ell)$ element of $\mathbf{Z}_{i,j}^{0}$.  The summation of the last term expresses the sum of all the combinations of $\mathbf{{\bar{q}}}_{0}$ random variables exerted from the $i$-th row of $\mathbf{Z}_{i,j}^{0}$ except those in $\set_{j}$.
\end{lemma}
Second, we analyze and derive the upper bound of  the complicated summation of the moment  in Lemma \ref{Lipschitz_higherorder_d-dependent}.
\begin{lemma}\label{general_inequality_dependent}

   Suppose Assumptions \ref{ass::regime} and \ref{ass::loss} hold. Let $C_{i}^{j}$ denote the $i$-th row of $\mathbf{Z}^{0}_{ij}$.
   Consider two sequences of i.i.d. $\R^p$-valued random vectors $X_1,...,X_n$ and $W_1,...,W_n$. 
   Suppose that each $\{X_{i}\}_{i=1}^{n}$ and $\{W_{i}\}_{i=1}^{n}$ is $d$-block dependent. Suppose $\bar{\mathbf{q}}_{0}\leq p$ holds.  Then, for any $\set_{j}$, we have
    \begin{align*}
     &\sum_{k_{1},\cdots,k_{\bar{\mathbf{q}}_{0}}\notin \set_{j}} E\left[C^{j}_{ik_{1}}\cdots C^{j}_{ik_{\bar{\mathbf{q}}_{0}}}\right] \\
     &\lesssim_{\tau,q}\frac{d^{(\bar{\mathbf{q}}_{0}/2)+1}}{n}\max\left\{\max_{1\leq k\leq p}(\|X_{i,k}\|_{2^{(\bar{\mathbf{q}}_{0}/2)}})^{(\bar{\mathbf{q}}_{0})}, \max_{1\leq k\leq p}(\|W_{i,k}\|_{2^{(\bar{\mathbf{q}}_{0}/2)}})^{(\bar{\mathbf{q}}_{0})}\right\}\\
     & \quad +\left(\max_{1\leq i\leq k}\lambda_{\max}(\Sigma_{|\set_{i}|})\right)^{(\bar{\mathbf{q}}_{0}/2)},
\end{align*}
where the summation expresses the sum of all the combinations of $\mathbf{{\bar{q}}}_{0}$ random variables with repetition, exerted from the $i$-th row of $\mathbf{Z}_{i,j}^{0}$ except those in $\set_{j}$. 
\end{lemma}

Now, we are ready to prove Proposition \ref{step4}.
\begin{proof}[Proof of Proposition \ref{step4}]
From Lemma \ref{Lipschitz_higherorder_d-dependent},  we have
\begin{align*}
&|E[g(n^{-1}F_{\beta}((\bar{H}_{\loss_{\rho}}(w,\mathbf{X})))_{w\in\mathcal{W
}}]-E[g(n^{-1}F_{\beta}((\bar{H}_{\loss_{\rho}}(w,\mathbf{W})))_{w\in\mathcal{W
}}]| \\
 &\lesssim_{\tau,q,M} K_{g}\max\{1,\beta^{2}\}L_{n}^{\bar{\mathbf{q}}_{0}+3}D^{3}_{\loss}(\rho) \\
    &\times\left\{\frac{1}{n}\sum_{i=1}^{n}\sum_{j=1}^{k}\max_{Z_{i}^{(\set_{j})}\in (X_{i}^{(\set_{j})}, W_{i}^{(\set_{j})})}E\left[\left(\sum_{\ell\in\set_j}\left|Z_{i,\ell}\right|\right)^{3}\right]\right. \\
     &+\frac{1}{n}\sum_{i=1}^{n}L^{\mathbf{q}}_{\loss_{i}}\sum_{j=1}^{k}\max_{Z_{i}^{(\set_{j})}\in (X_{i}^{(\set_{j})}, W_{i}^{(\set_{j})})}
E\left[\left(\sum_{\ell\in\set_{j}}\left|Z_{i,\ell}\right|\right)^{3}\right] \\
    &+\frac{1}{n}\sum_{i=1}^{n}L^{\mathbf{q}}_{\loss_{i}}\sum_{j=1}^{k}\max_{Z_{i}^{(\set_{j})}\in (X_{i}^{(\set_{j})}, W_{i}^{(\set_{j})})}
E\left[\left(\sum_{\ell\in\set_{j}}\left|Z_{i, \ell}\right|\right)^{\bar{\mathbf{q}}_{0}+3}\right] \\
&+\left.\frac{1}{n}\sum_{i=1}^{n}L^{\mathbf{q}}_{\loss_{i}}\sum_{j=1}^{k}\max_{Z_{i}^{(\set_{j})}\in (X_{i}^{(\set_{j})}, W_{i}^{(\set_{j})})}E\left[\left(\sum_{\ell\in\set_{j}}\left|Z_{i,\ell}\right|\right)^{3}\right]\sum_{k_{1},\cdots,k_{\bar{\mathbf{q}}_{0}}\notin \set_{j}} E\left[C^{j}_{ik_{1}}\cdots C^{j}_{ik_{\bar{\mathbf{q}}_{0}}}\right]\right\}.
\end{align*}
Then, we need to derive the upper bounds for the four terms.

By definition, for each $j\in\{1,\cdots, k\}$, we have
\begin{align*}  E\left[\left(\sum_{\ell\in\set_{j}}\left|X_{i, \ell}\right|\right)^{3}\right]  =&E\left[\sum_{\ell\in\set_{j}}\sum_{\ell\in\set_{j}}\sum_{\ell^{\prime\prime}\in\set_{j}}|X_{i, \ell}X_{i, \ell^{\prime}}X_{i,\ell^{\prime\prime}}|\right] \\
\leq & |\set_{j}|^{3}\max_{\ell,\ell^{\prime},\ell^{\prime\prime}\in\set_{j}} E\left[|X_{i, \ell}X_{i, \ell^{\prime}}X_{i,\ell^{\prime\prime}}|\right].
\end{align*}
By the Cauchy-Schwartz inequality, for any $\ell,\ell^{\prime},\ell^{\prime\prime}\in\set_{j}$, we have  
\begin{align*}
    E\left[|X_{i, \ell}X_{i, \ell^{\prime}}X_{i,\ell^{\prime\prime}}|\right]  &\leq \|X_{i, \ell}X_{i, \ell^{\prime}}\|_{2}\|X_{i, \ell^{\prime\prime}}\|_{2} \\
    &\leq\|X_{i,\ell}\|_{4} \|X_{i, \ell^{\prime}}\|_{4}\|X_{i,\ell^{\prime\prime}}\|_{4}  \\
    &\leq\max_{\ell\in\set_{j}}\|X_{i,\ell}\|^{3}_{4}.
\end{align*}
Similarly, for each $j\in\{1,\cdots, k\}$, we also have
\begin{align*}    &E\left[\left(\sum_{\ell\in\set_{j}}\left|X_{i, \ell}\right|\right)^{\bar{\mathbf{q}}_{0}+3}\right]  
\leq  |\set_{j}|^{\bar{\mathbf{q}}_{0}+3}\max_{\ell\in\set_{j}}\|X_{i,\ell}\|^{\bar{\mathbf{q}}_{0}+3}_{2^{(\bar{\mathbf{q}}_{0}+4)/2}}.
\end{align*}
Then, from Lemma \ref{dependent_group_sum} and \ref{general_inequality_dependent}, we obtain 
\begin{align*}
    \sum_{j=1}^{k}\max_{Z_{i}^{(\set_{j})}\in (X_{i}^{(\set_{j})}, W_{i}^{(\set_{j})})}
E\left[\left(\sum_{\ell\in\set_{j}}\left|Z_{i,\ell}\right|\right)^{3}\right]&\lesssim \frac{pd^{2}}{n^{3/2}} \max\left\{\max_{1\leq k\leq p}\|X_{i,k}\|^{3}_{4}, \max_{1\leq k\leq p}\|W_{i,k}\|^{3}_{4}\right\},
\end{align*}
and
\begin{align*}
&\sum_{j=1}^{k}\max_{Z_{i}^{(\set_{j})}\in (X_{i}^{(\set_{j})}, W_{i}^{(\set_{j})})}
E\left[\left(\sum_{\ell\in\set_{j}}\left|Z_{i,\ell}\right|\right)^{\bar{\mathbf{q}}_{0}+3}\right]\\
&\lesssim \frac{pd^{\bar{\mathbf{q}}_{0}+2}}{n^{(\bar{\mathbf{q}}_{0}+3)/2}}  \max\left\{\max_{1\leq k\leq p}\|X_{i,k}\|^{\bar{\mathbf{q}}_{0}+3}_{2^{(\bar{\mathbf{q}}_{0}+4)/2}} , \max_{1\leq k\leq p}\|W_{i,k}\|^{\bar{\mathbf{q}}_{0}+3}_{2^{(\bar{\mathbf{q}}_{0}+4)/2}}\right\},
\end{align*}
Furthermore, we have 
\begin{align*}
&\sum_{j=1}^{k}\max_{Z_{i}^{(\set_{j})}\in (X_{i}^{(\set_{j})}, W_{i}^{(\set_{j})})}E\left[\left(\sum_{\ell\in\set_{j}}\left|Z_{i,\ell}\right|\right)^{3}\right]\sum_{k_{1},\cdots,k_{\bar{\mathbf{q}}_{0}}\notin \set_{j}} E\left[C_{ik_{1}}\cdots C_{ik_{\bar{\mathbf{q}}_{0}}}\right] \\
&\lesssim_{\tau, q} \frac{pd^{3+(\bar{\mathbf{q}}_{0}/2)}}{n^{5/2}} \max\left\{\max_{1\leq k\leq p}\|X_{i,k}\|^{3}_{4}, \max_{1\leq k\leq p}\|W_{i,k}\|^{3}_{4}\right\}\\
& \qquad \times \left(\max\{\max_{1\leq k\leq p}\|X_{i,k}\|^{(\bar{\mathbf{q}}_{0})}_{2^{(\bar{\mathbf{q}}_{0}/2)}}, \max_{1\leq k\leq p}\|W_{i,k}\|^{(\bar{\mathbf{q}}_{0})}_{2^{(\bar{\mathbf{q}}_{0}/2)}}\}\right) \\
&+\frac{pd^{2}}{n^{3/2}}\max\left\{\max_{1\leq k\leq p}\|X_{i,k}\|^{3}_{4}, \max_{1\leq k\leq p}\|W_{i,k}\|^{3}_{4}\right\}\left(\lambda_{\max}(\Sigma)\right)^{\bar{\mathbf{q}}_{0}/2}.
\end{align*}
Therefore, the required result holds.
\end{proof}

\subsection{Proof of Lemma \ref{step2} and Lemma \ref{step3}}

\begin{proof}[Proof of Lemma \ref{step2}]
We construct a finite set \(\mathcal{S}_{n,\delta}\) such that \(|\mathcal{S}_{n,\delta}| \leq \lceil 2L_{n}/\delta \rceil^{p}\). Fix \(\delta \in (0,1)\). Define \(\delta \mathbf{Z}\) as the set of integer multiples of \(\delta\), i.e., \(\delta \mathbf{Z} := \{\dots, -2\delta, -\delta, 0, \delta, 2\delta, \dots\}\). We then define \(\mathcal{S}_{n,\delta}\) as the intersection of \(\delta \mathbf{Z}\) with \(\mathcal{S}_{n}\). By construction, the number of points in each dimension of \(\mathcal{S}_{n,\delta}\) is bounded above by \(\lceil 2L_{n}/\delta \rceil\), implying that \(|\mathcal{S}_{n,\delta}| \leq \lceil 2L_{n}/\delta \rceil^{p}\). Let \(w_{*}\) and \(w_{*,\delta}\) denote the arguments of the minimum in \(\mathcal{S}_{n}\) and \(\mathcal{S}_{n,\delta}\) for \(H_{\loss_{\rho}}(w, S)\), respectively. By the definition of \(\mathcal{S}_{n,\delta}\), it follows that \(\|w_{*} - w_{*,\delta}\| \leq \delta\). Then, we obtain
\begin{align*}
&\left| \mathbb{E}\left[ g\left(\min_{w \in \mathcal{S}_{n,\delta}} H_{\loss_{\rho}}(w, \mathbf{X}) \right)\right] - \mathbb{E}\left[ g\left(\min_{w \in \mathcal{S}_{n}} H_{\loss_{\rho}}(w, \mathbf{X}) \right)\right] \right| \\
&\leq \|g^{\prime}\|_{\infty} \mathbb{E}\left[\left(H_{\loss_{\rho}}(w_{*,\delta}, \mathbf{X}) - H_{\loss_{\rho}}(w_{*}, \mathbf{X})\right)_{+}\right] \\
&\leq \|g^{\prime}\|_{\infty}\left(\frac{1}{n} \mathbb{E}\left[\left|\sum_{i=1}^{n}\left(\loss_{i;\rho}((\mathbf{X}w_{*,\delta})_{i}) - \loss_{i;\rho}((\mathbf{X}w_{*})_{i})\right)\right|\right] + \sup_{\substack{w, w^{\prime} \in [-L_{n}, L_{n}]^{p} \\ \|w - w^{\prime}\|_{\infty} \leq \delta}} |\regularizer (w) - \regularizer (w^{\prime})|\right).
\end{align*}
From Assumption \ref{ass::loss}, we have
\[
|\loss_{i;\rho}((\mathbf{X}w_{*,\delta})_{i}) - \loss_{i;\rho}((\mathbf{X}w_{*})_{i})|
\leq 2\mathcal{D}_{\loss}(\rho)\left|(\mathbf{X}(w_{*} - w_{*,\delta}))_{i}\right| L_{\loss_{i}}^{q_{1}}\left(1 + |(\mathbf{X}w_{*})_{i}|^{q_{0}q_{1}} + |(\mathbf{X}w_{*,\delta})_{i}|^{q_{0}q_{1}}\right).
\]
Thus, the inequality yields
  \begin{align*}
        &\frac{1}{n}E\left[\left|\sum_{i=1}^{n}(\loss_{i;\rho}((\mathbf{X}w_{*,\delta})_{i})-\loss_{i;\rho}((\mathbf{X}w_{*})_{i}))\right|\right] \\
        &\leq 2\mathcal{D}_{\loss}(\rho)E\left[\|\mathbf{X}(w_{*}-w_{*,\delta})\|_{\infty}(1+\|\mathbf{X}w_{*}\|_{\infty}^{q_{0}q_{1}}+\|\mathbf{X}w_{*,\delta}\|_{\infty}^{q_{0}q_{1}})\frac{1}{n}\sum_{i=1}^{n}L^{q_{1}}_{\loss_{i}}\right] \\
        &\leq  2\mathcal{D}_{\loss}(\rho)\delta E\left[\left(\max_{1\leq i\leq n}\sum_{j=1}^{p}|X_{i,j}|\right)(1+\|\mathbf{X}w_{*}\|_{\infty}^{q_{0}q_{1}}+\|\mathbf{X}w_{*,\delta}\|_{\infty}^{q_{0}q_{1}})\frac{1}{n}\sum_{i=1}^{n}L^{q_{1}}_{\loss_{i}}\right] \\
        &\leq  2\mathcal{D}_{\loss}(\rho)\delta L^{q_{0}q_{1}}_{n}Av(\{L^{q_{1}}_{\loss_{i}}\})\left(E\left[\max_{1\leq i\leq n}\sum_{j=1}^{p}|X_{i,j}|\right]+E\left[\left(\max_{1\leq i\leq n}\sum_{j=1}^{p}|X_{i,j}|\right)^{q_{0}q_{1}+1}\right]\right),
    \end{align*}
where the second inequality follows from \(\|w_{*} - w_{*,\delta}\| \leq \delta\) and the third inequality follows from the \(\ell_{\infty}\) bound of the parameter space \(\mathcal{S}_{n}\). Therefore, the required result holds.
    
\end{proof}

\begin{proof}[Proof of Lemma \ref{step3}]
    It simply follows from Proposition 4.1 in \cite{han2023universality}.
\end{proof}

\section{Proof of Theorem \ref{Theorem_3.1}}

Following steps taken in \citet{thrampoulidis2018precise}, we  describe the reduction of the original optimization problem to  the convex-concave minimax scalar optimization defined in equation \eqref{eq::TAH_main}. Then, we verify the convergence result of the error of M-estimators. Suppose Assumptions \ref{ass::moment} and \ref{ass::eigen} hold throughout this appendix. Appendix \ref{sec::aux} provides proofs of all the lemmas stated in this section.

\subsection{Setting}

We restate the regularization problem we study
\begin{equation*}
    \hat{\theta}\in\argmin_{\theta\in\R^p}\left\{\mathcal{L}(\mathbf{y}-\mathbf{G}\theta)+\lambda f(\theta)\right\}, \quad \text{with}\quad Y_{i}=G_{i}^{\top}\theta_0+\xi_i,
\end{equation*}
where we define $\mathcal{L}(\mathbf{y}-\mathbf{G}\theta)=\sum_{i=1}^{n}\loss_{0}(Y_{i}-G_{i}^{\top}\theta)$ where $G_{i}\sim \mathcal{N}(0,\Sigma)$ and  $\xi_i$ is an independent noise variable.

We study the asymptotic behavior of $\|\hat{\theta}-\theta_0\|_{\Sigma}/\sqrt{p}$. For simplicity, 
let $w$ denote $(\theta-\theta_0)/\sqrt{p}$ Because $\mathbf{y}=A\Sigma^{1/2}\theta_{0}+\xi$, we  obtain
\begin{equation}\label{TAH_64}
        \hat{w}\in\argmin_{w\in\R^p}\frac{1}{n}\left\{\mathcal{L}(\xi-\sqrt{p}\mathbf{A}\Sigma^{1/2}w)+\lambda f(\theta_{0}+\sqrt{p}w)\right\}.
\end{equation}
where we define $\mathbf{A}$ as $(A_1,\cdots, A_n)^{\top}$ where $A_i\sim \mathcal{N}(0,I)$. For the establishment of Theorem \ref{Theorem_3.1},  we need to verify Lemma A.1 to A.5 in  \citet{thrampoulidis2018precise} with suitable modifications.

\subsection{Boundedness of error}

First, we reduce the region of minimization to some compact sets. Define $\mathcal{S}_{w}$ as $\{w\in\mathbb{R}^{p}|\|w\|_{\Sigma}\leq K_{\alpha}\}$ where
\begin{equation}\label{TAH_65}
K_{\alpha}:=\alpha_{*}+\zeta,  
\end{equation}
for a constant $\zeta>0$ and consider the bounded version of \eqref{TAH_64}:
\begin{equation}\label{TAH_66}
        \hat{w}^{B}\in\argmin_{w\in\mathcal{S}_{w}}\frac{1}{n}\left\{\mathcal{L}(\xi-\sqrt{p}\mathbf{A}\Sigma^{1/2}w)+\lambda f(\theta_{0}+\sqrt{p}w)\right\}.
\end{equation}
To make $\mathcal{S}_{w}$ compact, we need Assumptions \ref{ass::moment} and \ref{ass::eigen}. Then, we state the equivalence of the convergence of $\|\hat{w}^{B}\|_{\Sigma}$ and $\|\hat{w}\|_{\Sigma}$ in probability. 
\begin{lemma}\label{Lemma_A.1} For the two optimizations in \eqref{TAH_64} and \eqref{TAH_66}, let $\hat{w}$ and $\hat{w}^{B}$ be optimal solutions.  If $\|\hat{w}^{B}\|_{\Sigma}\xrightarrow[]{P}\alpha^{*}$, then $\|\hat{w}\|_{\Sigma}\xrightarrow[]{P}\alpha^{*}$.
\end{lemma}
The proof of Lemma \ref{Lemma_A.1} only utilizes the property of semi-norm and the convexity of the objective function. Thus, this result trivially holds.

\subsection{Identifying (PO)}

By utilizing the duality with constraints, $\mathbf{v}=\xi-\sqrt{p}\mathbf{A}\Sigma^{1/2}w$, we obtain
\begin{equation}\label{TAH_67}
        \hat{w}=\argmin_{w\in\mathcal{S}_{w},\mathbf{v}}\max_{\mathbf{u}}\frac{1}{\sqrt{p}}\left\{-\mathbf{u}^{\top}(\sqrt{p}\mathbf{A})\Sigma^{1/2}w+\mathbf{u}^{\top}\xi-\mathbf{u}^{\top}\mathbf{v}\right\}+\frac{1}{p}\left\{\mathcal{L}(\mathbf{v})+\lambda f(\theta_{0}+\sqrt{p}w)\right\}.
\end{equation}
To make this setting applicable to CGMT, we consider the following bounded (PO). 
\begin{equation}\label{TAH_68}
        \hat{w}=\argmin_{w\in\mathcal{S}_{w},\mathbf{v}}\max_{\mathbf{u}\in\mathcal{S}_{u}}\frac{1}{\sqrt{p}}\left\{-\mathbf{u}^{\top}(\sqrt{p}\mathbf{A})\Sigma^{1/2}w+\mathbf{u}^{\top}\xi-\mathbf{u}^{\top}\mathbf{v}\right\}+\frac{1}{p}\left\{\mathcal{L}(\mathbf{v})+\lambda f(\theta_{0}+\sqrt{p}w)\right\},
\end{equation}
where we define $\mathcal{S}_{u}:=\{\mathbf{u}\in\mathbb{R}^{n}|\|\mathbf{u}\|_{2}\leq K_{\beta}\}$ and $K_{\beta}$ a sufficiently large constant. 
By Assumption \ref{ass::TAH_1}-\eqref{ass::TAH_1_complement}, we can verify the equivalence of \eqref{TAH_67} and \eqref{TAH_68}:
\begin{lemma}\label{Lemma_A.2} If Assumption \ref{ass::TAH_1}-\eqref{ass::TAH_1_complement} holds, then there exists a sufficiently large constant $K_{\beta}$, such that the
optimization problem in \eqref{TAH_68} is equivalent to that in \eqref{TAH_66}, with probability approaching 1 as $n\rightarrow\infty$.
\end{lemma}
The proof of Lemma \ref{Lemma_A.2} also trivially holds because it heavily depends on Assumption \ref{ass::TAH_1}. \eqref{ass::TAH_1_complement}. Moreover, we can verify $\|\mathbf{A}\Sigma^{1/2}\|_{2}\leq C$ with high probability because of Assumption \ref{ass::moment} and \ref{ass::eigen}. 

\subsection{Analyzing (AO)}

Through the Fenchel conjugate, \eqref{TAH_68} can be reduced to 
\begin{equation}\label{TAH_69}
\hat{w}=\argmin_{w\in\mathcal{S}_{w},\mathbf{v}}\max_{\mathbf{u}\in\mathcal{S}_{u},s}\frac{1}{\sqrt{p}}\left\{-\mathbf{u}^{\top}(\sqrt{p}\mathbf{A})\Sigma^{1/2}w+\mathbf{u}^{\top}\xi-\mathbf{u}^{\top}\mathbf{v}\right\}+\frac{1}{p}\left\{\mathcal{L}(\mathbf{v})+\lambda \mathbf{s}^{\top}\theta_0+\lambda \sqrt{p}\mathbf{s}^{\top}w-\lambda f^{*}(\mathbf{s})\right\}.
\end{equation}
where we define $f^{*}(\mathbf{s})$ as $\sup_{\mathbf{x}}\mathbf{s}^{\top}\mathbf{x}-f(\mathbf{x})$. Then, consider the following problem as the (AO). Because the function above itself is convex, it is easy to deal with:
\begin{equation}\label{TAH_71}
\begin{split}
    \phi(\mathbf{g},\mathbf{h}):=\max_{0\leq \beta\leq K_{\beta}, \mathbf{s}}\min_{\|w\|_{\Sigma}\leq K_{\alpha}, \mathbf{v}}\max_{\|\mathbf{u}\|_{2}=\beta}\frac{1}{\sqrt{p}}(\|w\|_{\Sigma}\mathbf{g}+\xi-\mathbf{v})^{\top}\mathbf{u}-\frac{1}{\sqrt{p}}\|\mathbf{u}\|_{2}\mathbf{h}^{\top}\Sigma^{1/2}w \\
    +\frac{1}{p}\left\{\mathcal{L}(\mathbf{v})+\lambda \mathbf{s}^{\top}\theta_0+\lambda \sqrt{p}\mathbf{s}^{\top}w-\lambda f^{*}(\mathbf{s})\right\}.
\end{split}
\end{equation}
where we define $\mathbf{g}\sim \mathcal{N}(0,I_n)$ and $\mathbf{h}\sim \mathcal{N}(0,I_p)$, respectively. Therefore, we obtain the following lemma:
\begin{lemma}\label{Lemma_A.3}
Let $\hat{w}(\mathbf{A})$ be an optimal solution to problem \eqref{TAH_64}, and consider the auxiliary optimization (AO) problem in \eqref{TAH_71}. Let $\alpha^*$ be as defined in Theorem \ref{Theorem_3.1}. For any $\varepsilon > 0$, define the set $\mathcal{S} := \left\{ w \in \mathbb{R}^p : \left| \| w \|_\Sigma - \alpha^* \right| < \varepsilon \right\}.$ Let $\phi_{\mathcal{S}^c}(\mathbf{g}, \mathbf{h})$ denote the optimal cost of the AO problem in \eqref{TAH_71}, but with the additional constraint that $w \notin \mathcal{S}$.

Assume that for every $K_\alpha > \alpha^*$ and sufficiently large $K_\beta$, there exist constants $\overline{\phi}$ and $\overline{\phi}_{\mathcal{S}^c}$ satisfying $\overline{\phi} < \overline{\phi}_{\mathcal{S}^c}$, such that for all $\eta > 0$, the following holds with probability approaching 1 as $n \to \infty$:
\begin{itemize}
    \item [(a)] $\phi(\mathbf{g}, \mathbf{h}) < \overline{\phi} + \eta$,
     \item [(b)] $\phi_{\mathcal{S}^c}(\mathbf{g}, \mathbf{h}) > \overline{\phi}_{\mathcal{S}^c} - \eta$.
\end{itemize}
Then it follows that:
$$
\lim_{n \to \infty} \Pr\left( \left| \| \hat{w}(\mathbf{A}) \|_\Sigma - \alpha^* \right| < \varepsilon \right) = 1.
$$
\end{lemma}
We can prove the above result by the main theorem of  \citet{akhtiamov2024novel} that generalizes the CGMT to the setting with a general covariance matrix.

\subsection{Scalarization}

To simplify the (AO) problem, we consider the transformation of the optimization into another optimization problem involving only scalar variables. Through techniques of scalarization, we get the following representation. 

\begin{equation}\label{TAH_72}
\begin{split}
    \inf_{\substack{0\leq \alpha\leq K_{\alpha} \\ \tau_g>0}}\sup_{\substack{0\leq \beta\leq K_{\beta} \\ \tau_h>0}} &\frac{\beta \tau_g}{2}+\frac{1}{p}e_{\mathcal{L}}\left(\alpha \mathbf{g}+\xi;\frac{\tau_g}{\beta}\right)\\
    &-\begin{cases}
        \frac{\alpha\tau_h}{2}+\frac{\beta^{2}\alpha}{2\tau_h}\frac{\|\mathbf{h}\|^{2}}{p}-\lambda\frac{1}{p}e_{g}\left(\frac{\beta\alpha}{\tau_h}\mathbf{h}+\Sigma^{1/2}\theta_0;\frac{\alpha\lambda}{\tau_h}\right), & \text{$\alpha>0$} \\
        \frac{\lambda}{p}f(\theta_0), &\text{$\alpha=0$} \\
    \end{cases}.
\end{split}
\end{equation}
where we define
    $e_{\omega}(\mathbf{u};\tau):=\min_{\mathbf{v}}\{\frac{1}{2\tau}\|\mathbf{u}-\mathbf{v}\|^{2}_{2}+\omega(\mathbf{v})\}$, and $g(\theta)=f(\Sigma^{-1/2}\theta)$.
Then, we obtain the following lemma:
\begin{lemma}\label{Lemma_A.4}
 The following statements are true regarding the two minimax
optimization problems in \eqref{TAH_71} and \eqref{TAH_72}:
\begin{itemize}
    \item [(i)] They have the same optimal cost.
    \item [(ii)] The objective function \eqref{TAH_72} is continuous on its domain, (jointly) convex in $(\alpha,\tau_g)$ and (jointly)
concave in $(\beta,\tau_h)$.
    \item [(iii)] The order of inf-sup in \eqref{TAH_72} can be flipped without changing the optimization.
\end{itemize}
\end{lemma}

\subsection{Convergence analysis}

Finally, we analyze the limit of the objective function with four scalar variables defined in \eqref{TAH_72}. We work with the following function $\phi(\mathbf{g},\mathbf{h},\xi,\theta_0)$:
\begin{align}
    &\phi(\mathbf{g},\mathbf{h},\xi,\theta_0)=\inf_{\substack{0\leq \alpha\leq K_{\alpha} \\ \tau_g>0}}\sup_{\substack{0\leq \beta\leq K_{\beta} \\ \tau_h>0}} \mathcal{R}_{p}(\alpha,\tau_g,\beta,\tau_h;\mathbf{g},\mathbf{h},\xi,\theta_0),\label{TAH_73}
\end{align}
where we define
\begin{align}
        \mathcal{R}_{p}:=&\frac{\beta \tau_g}{2}+\frac{1}{p}\left\{e_{\mathcal{L}}\left(\alpha \mathbf{g}+\xi;\frac{\tau_g}{\beta}\right)-\mathcal{L}(\xi)\right\} \\
        -&\begin{cases}
        \frac{\alpha\tau_h}{2}+\frac{\beta^{2}\alpha}{2\tau_h}\frac{\|\mathbf{h}\|^{2}}{p}-\frac{\lambda}{p}\left\{e_{g}\left(\frac{\beta\alpha}{\tau_h}\mathbf{h}+\Sigma^{1/2}\theta_0;\frac{\alpha\lambda}{\tau_h}\right)-f(\theta_0)\right\}, \quad &\text{$\alpha>0$} \\
        0, \quad &\text{$\alpha=0$} \\
    \end{cases}.
    \notag
\end{align}
\begin{lemma}\label{Lemma_A.5}
    Let $\mathcal{R}_{p}(\alpha,\tau_g,\beta,\tau_h):= \mathcal{R}_{p}(\alpha,\tau_g,\beta,\tau_h;\mathbf{g},\mathbf{h},\xi,\theta_0)$ be defined as in \eqref{TAH_73} and 
    \begin{equation}\label{TAH_74}
        \phi_{\mathcal{A}}:=\phi_{\mathcal{A}}(\mathbf{g},\mathbf{h},\xi,\theta_0)=\inf_{\substack{\alpha\in\mathcal{A} \\ \tau_g>0}}\sup_{\substack{0\leq \beta\leq K_{\beta} \\ \tau_h>0}} \mathcal{R}_{p}(\alpha,\tau_g,\beta,\tau_h)
    \end{equation}
    for $\mathcal{A}\subset[0,\infty)$. Moreover, we define $\overline{\phi}_{\mathcal{A}}$ as the following deterministic convex program:
    \begin{equation}\label{TAH_75}
    \begin{split}
        \overline{\phi}_{\mathcal{A}}:=&\inf_{\substack{\alpha\in\mathcal{A} \\ \tau_g>0}}\sup_{\substack{0\leq \beta\\ \tau_h>0}} \mathcal{\mathcal{D}}(\alpha,\tau_g,\beta,\tau_h) \\
        :=&\begin{cases}
            \frac{\beta \tau_g}{2}+\tau_0 L\left(\alpha,\frac{\tau_g}{\beta}\right),\quad &\text{$\beta>0$} \\
            -\tau_0 L_0,\quad &\text{$\beta=0$} \\
        \end{cases}-\begin{cases}
        \frac{\alpha\tau_h}{2}+\frac{\beta^{2}\alpha}{2\tau_h}-\lambda F\left(\frac{\alpha\beta}{\tau_h},\frac{\alpha\lambda}{\tau_h}\right), \quad &\text{$\alpha>0$} \\
        0, \quad &\text{$\alpha=0$} \\
    \end{cases}
    \end{split}
    \end{equation}
    where $L$ and $F$ defined in Assumption \ref{ass::TAH_1} \eqref{ass::TAH_1_L}-\eqref{ass::TAH_1_F} and we define $\tau_0$ as the limit of $n/p$. If Assumption \ref{ass::TAH_1}  \eqref{ass::TAH_1_L}-\eqref{ass::TAH_1_F} and \ref{ass::TAH_2} hold, then, 
    \begin{itemize}
        \item [(a)]  $\mathcal{R}_{p}(\alpha,\tau_g,\beta,\tau_h)\xrightarrow[p]{}\mathcal{\mathcal{D}}(\alpha,\tau_g,\beta,\tau_h)$ 
 for all $(\alpha,\tau_g,\beta,\tau_h)$, and, $\mathcal{D}(\alpha,\tau_g,\beta,\tau_h)$ is convex in $(\alpha,\tau_g)$ and concave in $(\beta,\tau_h)$.

 \item[(b)]
Assume $\alpha_{*}$ is the unique minimizer in \eqref{TAH_75} with $\mathcal{A}:= [0,\infty)$. For any $\varepsilon>0$, define $\mathcal{S}_{\varepsilon}:= \{\alpha||\alpha-\alpha_{*}|<\varepsilon\}$. Then, for any sufficiently large constants $K_{\alpha}>\alpha_{*}$ and $K_\beta >0$, and for all $\eta>0$, it holds with probability approaching 1 as $n\rightarrow\infty$:
\begin{itemize}
    \item [(i)] $\phi_{[0,K_{\alpha}]}<\overline{\phi}_{[0,\infty)}+\eta$,
        \item [(ii)] $\phi_{[0,K_{\alpha}]\backslash\mathcal{S}_\varepsilon}>\overline{\phi}_{[0,\infty)\backslash\mathcal{S}_\varepsilon}-\eta$,
            \item [(iii)] $\overline{\phi}_{[0,\infty)\backslash\mathcal{S}_\varepsilon}>\overline{\phi}_{[0,\infty)}$.
\end{itemize}
    \end{itemize}
\end{lemma}

\subsection{Proof of Theorem \ref{Theorem_3.1}}

Finally, we conclude the proof of Theorem \ref{Theorem_3.1} through the result of Lemma \ref{Lemma_A.1} to Lemma \ref{Lemma_A.5}.

\begin{proof}[Proof of Theorem \ref{Theorem_3.1}]

Take any $\varepsilon>0$ and define $\mathcal{S}_{\varepsilon}$ as $\{w\in\mathbb{R}^{p}||\|w\|_{\Sigma}-\alpha_{*}|<\varepsilon\}$. Let $K_{\alpha}>\alpha_{*}$ and arbitrary large $K_{\beta}>0$. From Lemma \ref{Lemma_A.4} (i), $\phi(\mathbf{g},\mathbf{h})$ has the same optimal cost as that of \eqref{TAH_72}. From the result of Lemma \ref{Lemma_A.5}, there exist constants $\overline{\phi}_{[0,\infty)}$
 and     $\overline{\phi}_{[0,\infty)\backslash\mathcal{S}_{\varepsilon}}$ with $\overline{\phi}_{[0,\infty)\backslash\mathcal{S}_{\varepsilon}}>\overline{\phi}_{[0,\infty)}$, which satisfies conditions of Lemma \ref{Lemma_A.3}. Therefore, we obtain
 \begin{equation*}
     \lim_{n\rightarrow\infty}\Pr(|\|\hat{w}\|_{\Sigma}-\alpha_{*}|<\varepsilon)=1.
 \end{equation*}
\end{proof}

\section{Proof for applications}

\subsection{Proof of Lemma \ref{robus_linf}}

\begin{proof}
Let $\hat{w}^{(s)}$ denote the column leave-one-out estimator, i.e., $\hat{w}^{(s)}:=\arg\min_{w\in\mathbf{R}^{p}:w_{s}=0}H(w)$.
 Without assuming Assumption \ref{ass::robust}, as in the proof of Proposition 3.11 in \citet{han2023universality}, we obtain
 \begin{equation*}
     |\hat{w}_{\mathbf{X},s}|\lesssim \lambda^{-1}\left|\sum_{i=1}^{n}X_{i,s}\psi^{\prime}_{0}(X^{\top}_{i,-s}\hat{w}^{(s)}_{\mathbf{X},-s}-\xi_{i})\right|+\|\theta_{0}\|_{\infty},
 \end{equation*}
where we define $\hat{w}_{\mathbf{X},s}$ and $\hat{w}^{(s)}_{\mathbf{X},-s}$ as the $s-$th element of $\hat{w}_{\mathbf{X}}$ and the one removing the $s-$th element of $\hat{w}^{(s)}$ from the original vector, respectively. Therefore, for $k\geq2$, we obtain
\begin{align*}
    E[|\hat{w}_{\mathbf{X},s}|^{k}]\lesssim_{k}& \lambda^{-k} E\left[\left|\sum_{i=1}^{n}X_{i,s}\psi^{\prime}_{0}(X^{\top}_{i,-s}\hat{w}^{(s)}_{\mathbf{X},-s}-\xi_{i})\right|^{k}\right]+\|\theta_{0}\|^{k}_{\infty} \\
    \overset{\text{(i)}}{\leq}& (L_{0}/\lambda)^{k}E\left[\left|\sum_{i=1}^{n}X_{i,s}\right|^{k}\right]+\|\theta_{0}\|^{k}_{\infty} \\
       \overset{\text{(ii)}}{\lesssim_{k}}& (L_{0}/\lambda)^{k}E\left[\left|\sum_{i=1}^{n}\varepsilon_{i}X_{i,s}\right|^{k}\right]+\|\theta_{0}\|^{k}_{\infty} \\
      \overset{\text{(iii)}}{\lesssim_{k}}& (L_{0}/\lambda)^{k}E\left[\left|\sum_{i=1}^{n}X^{2}_{i,s}\right|^{k/2}\right]+\|\theta_{0}\|^{k}_{\infty},
\end{align*}
where $\{\varepsilon_{i}\}$ are i.i.d. Rademacher random variables independent of $\mathbf{X}$. Inequality (i) follows from Assumption \ref{ass::robust}-\ref{ass::robust::loss}. We applied the symmetrization inequality of Theorem 3.1.21 in \citet{gine2016mathematical} and Khintchine’s inequality to inequality (ii) and inequality (iii), respectively.
From Jensen's inequality, we obtain
\begin{equation*}
    E\left[\left|\sum_{i=1}^{n}X^{2}_{i,s}\right|^{k/2}\right]\leq \left(E\left[\left|\sum_{i=1}^{n}X^{2}_{i,s}\right|^{\lceil k/2\rceil }\right]\right)^{k/2{\lceil k/2\rceil}}.
\end{equation*}
By definition, $X_{i,s}=X^{0}_{i,s}/\sqrt{n}$ holds. Then, we have
\begin{equation*}
    E\left[\left|\sum_{i=1}^{n}X^{2}_{i,s}\right|^{\lceil k/2\rceil }\right]=E\left[\left(\frac{1}{n}\sum_{i=1}^{n}(X^{0}_{i,s})^{2}\right)^{\lceil k/2\rceil }\right].
\end{equation*}
By applying  H\"older  inequality, we have
\begin{equation*}
   E\left[\left(\frac{1}{n}\sum_{i=1}^{n}(X^{0}_{i,s})^{2}\right)^{\lceil k/2\rceil }\right]\leq \frac{n^{\lceil k/2\rceil-1}}{n^{\lceil k/2\rceil}}  E\left[\sum_{i=1}^{n}(X^{0}_{i,s})^{2\lceil k/2\rceil }\right].
\end{equation*}
Because $X^{0}_{i,s}$ is i.i.d. with respect to $i$, we have
\begin{align*}
    \frac{n^{\lceil k/2\rceil-1}}{n^{\lceil k/2\rceil}}  E\left[\sum_{i=1}^{n}(X^{0}_{i,s})^{2\lceil k/2\rceil }\right]=&\frac{1}{n}  E\left[\sum_{i=1}^{n}(X^{0}_{i,s})^{2\lceil k/2\rceil }\right] \\
    =&E[(X^{0}_{i,s})^{2\lceil k/2\rceil }].
\end{align*}
Therefore, the required result holds.
\end{proof}

\subsection{Proof of Theorem \ref{universality_robust}}

\begin{proof}
It holds from Lemma \ref{robus_linf} and Markov's inequality that
\begin{equation*}
   \Pr(\|\hat{w}^{R}_{\mathbf{X}}\|_{\infty}>L_{n})\leq L_{n}^{-k^{*}}E\max_{j\in [n]}|\hat{w}^{R}_{\mathbf{X},j}|^{k^{*}}\lesssim_{L_{0},pk^{*}\lambda}(pL^{-k^{*}}_{n})((M_{k^{*};\bfX})^{k/2\lceil k/2 \rceil}+\|\theta_{0}\|^{k^{*}}_{\infty}).
\end{equation*}
As in \citet{han2023universality}, we apply Theorem \ref{main_universality_structure}.
In the robust case, $L_{\ell_{i}}=CL_{0}(1+|\xi_{i}|)$, $\mathbf{q}=0$, $\mathcal{D}_{\loss}(\rho)=CL_0/\rho^{2}$,
$\mathcal{M}_{\loss}(\rho)= CL_0\rho$, and $\bar{\rho}= 1/(CL_{0})$ hold from Lemma B.1 in \citet{han2023universality}. Then, by setting $\rho$ and $\delta$ appropriately, we have
\begin{equation*}
    r_{R}(L_{n})\lesssim_{L_{0}}\left(L_{n}\log^{2/3}(n)\sigma^{1/3}_{n}\right)^{1/7}.
\end{equation*}
Therefore, there exists $K>0$ depending on $k^{*}, M_{k^{*};\bfX}, \tau,\lambda$ such that
\begin{equation*}
    \Pr(\hat{w}^{R}_{\mathbf{X}}\in \mathcal{S}_{n})
    \leq 4\varepsilon_{n}+K\left\{(nL^{-k^{*}}_{n})(1\vee \|\theta_{0}\|_{\infty}^{k^{*}})+(1\vee \rho^{-3}_{0})\left(L_{n}\log^{2/3}(L_{n})\sigma^{1/3}_{n}\right)^{1/7}\right\}.
\end{equation*}
Now setting $L_{n}:=(n/d^{4})^{1/6-\delta}$ where we define $\delta:=\frac{1}{6}\left(\frac{7\varepsilon k^{*}-42}{7\varepsilon k^{*}+1}\right)$. Therefore, we obtain
\begin{equation*}
    \Pr(\hat{w}^{R}_{\mathbf{X}}\in \mathcal{S}_{n})\leq 4\varepsilon_{n}+K(1+\|\theta_{0}\|^{k^{*}}_{\infty}+\rho^{-3}_{0})n^{-{1}/{45}}.
\end{equation*}
\end{proof}

\subsection{Proof of Theorem \ref{universality_robust_CGMT}}
\begin{proof}

All the conditions for Theorem \ref{main_universality_structure}  and Theorem \ref{Theorem_3.1} trivially hold from Theorem \ref{universality_robust} and assumptions of Theorem \ref{universality_robust} as in \citet{han2023universality}. 
\end{proof}

\subsection{Proof of Theorem \ref{universality_Ridge_Lasso_CGMT}}
\begin{proof}

All the conditions for Theorem \ref{main_universality_structure}  and Theorem \ref{Theorem_3.1} trivially hold from assumptions of Theorem \ref{universality_Ridge_Lasso_CGMT} as in Theorem \ref{universality_robust_CGMT}. 
\end{proof}

\section{Proof of auxiliary results}\label{sec::aux}

\subsection{Proof of Lemma \ref{Lipschitz_higherorder_d-dependent}}

\begin{proof}
Define \(\bar{H}_{\beta} := \beta \cdot \bar{H}\) and \(F_{\beta}(A) := F_{\beta}(\bar{H}(w, A)_{w \in \mW})\). As an initial step, we apply the result from Lemma 5, reducing the difference to the third derivatives of the functions involved. As shown in Appendix \ref{appen:preliminaries}, the third derivative of the function \(g(n^{-1}F_{\beta}(\cdot))\) can be bounded by the derivatives of the loss function \(\bar{H}\). For any $i\in\{1,\cdots,n\}$ and $j,k,\ell\in\{1,\cdots,d\}$, we have the following upper bound of derivatives of $\bar{H}$:
\begin{align*}
    |\langle\partial_{i,j} \bar{H}\rangle_{\bar{H}_{\beta}}|&\leq L_{n}D_{\loss}(\rho)(1+\langle\bar{H}^{(i)}(w,A)^{q_{1}}\rangle_{\bar{H}_{\beta}}), \\
     |\langle\partial_{i,j}\partial_{i,k} \bar{H}\rangle_{\bar{H}_{\beta}}|&\leq L^{2}_{n}D_{\loss}(\rho)(1+\langle\bar{H}^{(i)}(w,A)^{q_{2}}\rangle_{\bar{H}_{\beta}}), \\
      |\langle\partial_{i,j}\partial_{i,k}\partial_{i,\ell} \bar{H}\rangle_{\bar{H}_{\beta}}|&\leq L^{3}_{n}D_{\loss}(\rho)(1+\langle\bar{H}^{(i)}(w,A)^{q_{3}}\rangle_{\bar{H}_{\beta}}),
\end{align*}
where we define $\bar{H}^{(i)}(w,A):=\loss_{i:\rho}((Aw)_{i})\geq0$. so $\bar{H}(w,A)=\sum_{i=1}^{n}\bar{H}^{(i)}(w,A)+n\regularizer (w)$.
Therefore, any $i\in\{1,\cdots,n\}$ and $j,k,\ell\in\{1,\cdots,d\}$, we have the following bounds:
\begin{align*}
    |\langle\partial_{i,j} \bar{H}\rangle_{\bar{H}_{\beta}}\langle\partial_{i,k} \bar{H}\rangle_{\bar{H}_{\beta}}\langle\partial_{i,\ell} \bar{H}\rangle_{\bar{H}_{\beta}}|&\leq L^{3}_{n}D^{3}_{\loss}(\rho)(1+\langle\bar{H}^{(i)}(w,A)^{q_{1}}\rangle_{\bar{H}_{\beta}})^{3} \\
    &\leq2^2 L^{3}_{n}D^{3}_{\loss}(\rho)(1+\langle\bar{H}^{(i)}(w,A)^{3q_{1}}\rangle_{\bar{H}_{\beta}}), \\
    |\langle\partial_{i,j} \bar{H}\rangle_{\bar{H}_{\beta}}\langle\partial_{i,k} \bar{H}\partial_{i,\ell} \bar{H}\rangle_{\bar{H}_{\beta}}|&\leq \langle|w_{j}\partial\loss_{i:\rho}((Aw)_{i})|\rangle_{\bar{H}_{\beta}}\langle|w_{k}||w_{\ell}||\partial\loss_{i:\rho}((Aw)_{i})|^2\rangle_{\bar{H}_{\beta}} \\
    &\leq L^{3}_{n} \langle|\partial\loss_{i:\rho}((Aw)_{i})|^3\rangle_{\bar{H}_{\beta}} \\
    &\leq 2^2 L^{3}_{n}D^{3}_{\loss}(\rho)(1+\langle\bar{H}^{(i)}(w,A)^{3q_{1}}\rangle_{\bar{H}_{\beta}}), \\
    |\langle\partial_{i,j} \bar{H}\partial_{i,k} \bar{H}\partial_{i,\ell} \bar{H}\rangle_{\bar{H}_{\beta}}|
    &\leq L^{3}_{n} \langle|\partial\loss_{i:\rho}((Aw)_{i})|^3\rangle_{\bar{H}_{\beta}} \\
    &\leq 2^2 L^{3}_{n}D^{3}_{\loss}(\rho)(1+\langle\bar{H}^{(i)}(w,A)^{3q_{1}}\rangle_{\bar{H}_{\beta}}), \\
    |\langle\partial_{i,j}\partial_{i,k} \bar{H}\rangle_{\bar{H}_{\beta}}\langle\partial_{i,\ell} \bar{H}\rangle_{\bar{H}_{\beta}}|&\leq \langle|w_{j}w_{k}\partial^2 \loss_{i:\rho}((Aw)_{i})|\rangle_{\bar{H}_{\beta}}\langle|w_{\ell}\partial\loss_{i:\rho}((Aw)_{i})|\rangle_{\bar{H}_{\beta}} \\
    &\leq L_{n}^{3}\langle|\partial^2 \loss_{i:\rho}((Aw)_{i})|\rangle_{\bar{H}_{\beta}}\langle|\partial\loss_{i:\rho}((Aw)_{i})|\rangle_{\bar{H}_{\beta}} \\
    &\leq L_{n}^{3}D^{2}_{\loss}(\rho)\langle1+|\loss_{i:\rho}((Aw)_{i})|^{q_{2}}\rangle_{\bar{H}_{\beta}}\langle1+|\loss_{i:\rho}((Aw)_{i})|^{q_{1}}\rangle_{\bar{H}_{\beta}} \\
    &\lesssim  L_{n}^{3}D^{2}_{\loss}(\rho)\langle1+|\loss_{i:\rho}((Aw)_{i})|^{q_{1}+q_{2}}\}\rangle_{\bar{H}_{\beta}} \\
    &\lesssim L_{n}^{3}D^{3}_{\loss}(\rho)(1+\langle\bar{H}^{(i)}(w,A)^{q_{1}+q_{2}}\rangle_{\bar{H}_{\beta}}), \\
    |\langle\partial_{i,j}\partial_{i,\ell} \bar{H}\partial_{i,k} \bar{H}\rangle_{\bar{H}_{\beta}}|&\leq |w_{j}w_{\ell}w_{k}|\langle|\partial\loss^2_{i:\rho}((Aw)_{i})||\partial\loss_{i:\rho}((Aw)_{i})|\rangle_{\bar{H}_{\beta}} \\
  &\leq L^3_{n}D^{2}_{\loss}(\rho)\langle(1+|\loss_{i:\rho}((Aw)_{i})|^{q_{2}})(1+|\loss_{i:\rho}((Aw)_{i})|^{q_{1}})\rangle_{\bar{H}_{\beta}} \\
 &\lesssim  L_{n}^{3}D^{2}_{\loss}(\rho)\langle1+|\loss_{i:\rho}((Aw)_{i})|^{q_{1}+q_{2}}\}\rangle_{\bar{H}_{\beta}} \\
    &\lesssim L_{n}^{3}D^{3}_{\loss}(\rho)(1+\langle\bar{H}^{(i)}(w,A)^{q_{1}+q_{2}}\rangle_{\bar{H}_{\beta}}), \\
  |\langle\partial_{i,j}\partial_{i,k}\partial_{i,\ell} \bar{H}\rangle_{\bar{H}_{\beta}}|&\leq |w_{j}||w_{k}||w_{\ell}|\langle|\partial^{3} \loss_{i:\rho} ((Aw)_{i})|\rangle_{\bar{H}_{\beta}} \\
  &\leq L_{n}^{3}D_{\loss}(\rho)\langle1+|\loss_{i:\rho} ((Aw)_{i})|^{q_{3}}\rangle_{\bar{H}_{\beta}} \\
  &\leq L_{n}^{3}D_{\loss}(\rho)(1+\langle\bar{H}^{(i)}(w,A)^{q_{3}}\rangle_{\bar{H}_{\beta}}.
\end{align*}
Hence, for $\beta\geq1$, we have the following upper bound for the $\partial_{i,j}\partial_{i,k}\partial_{i,\ell}g(n^{-1}F_{\beta}(A))$:
\begin{align*}
    |\partial_{i,j}\partial_{i,k}\partial_{i,\ell}g(n^{-1}F_{\beta}(A))|\lesssim K_{g}n^{-1}\beta^{2}L_{n}^{3}D^{3}_{\loss}(\rho)\left[1+\langle\bar{H}^{(i)}(w,A)^{\mathbf{q}}\rangle_{\bar{H}_{\beta}}\right],
\end{align*}
where we define $\mathbf{q}=\max\{3q_{1},2q_{1}+2q_{2},q_{3}\}$. So, for any $j$ and $\ell,\ell^{\prime},\ell^{\prime\prime}\in B_j$, we obtain
\begin{align*}
   &E[| \partial_{i,\ell}\partial_{i,\ell^{\prime}}\partial_{i,\ell^{\prime\prime}}g(n^{-1}F_{\beta}(\mathbf{Z}_{i,j}^{0}+tX^{(\set_{j})}_{i}))||X^{(\set_{j})}_{i}] \\
   \lesssim &K_{g}n^{-1}\beta^{2}L_{n}^{3}D^{3}_{\loss}(\rho)\left[1+E[\langle(\bar{H}^{(i)}(w,\mathbf{Z}_{i,j}^{0}+tX^{(\set_{j})}_{i}))^{\mathbf{q}}\rangle_{\bar{H}_{\beta}}|X^{(\set_{j})}_{i}]\right].
\end{align*}
By using Chebyshev's association inequality, we have
\begin{equation*}
    E[\langle(\bar{H}^{(i)}(w,\mathbf{Z}_{i,j}^{0}+tX^{(\set_{j})}_{i}))^{\mathbf{q}}\rangle_{\bar{H}_{\beta}}|X^{(\set_{j})}_{i}]\leq E[\langle(\bar{H}^{(i)}(w,\mathbf{Z}_{i,j}^{0}+tX^{(\set_{j})}_{i}))^{\mathbf{q}}\rangle_{\bar{H}_{\beta;-i}}|X^{(\set_{j})}_{i}], 
\end{equation*}
where we define $H_{\beta;-i}:=\beta (\bar{H}(w,C)-\bar{H}^{(i)}(w,C))$ and $C$ is a $n\times d$ matrix
. Therefore, we have
\begin{equation}
\begin{split}
     &E[| \partial_{i,\ell}\partial_{i,\ell^{\prime}}\partial_{i,\ell^{\prime\prime}}g(n^{-1}F_{\beta}(\mathbf{Z}_{i,j}^{0}+tX^{(\set_{j})}_{i}))||X^{(\set_{j})}_{i}] \\
   & \lesssim K_{g}n^{-1}\beta^{2}L_{n}^{3}D^{3}_{\loss}(\rho)\left[1+E[\langle(\loss_{i;\rho}(((\mathbf{Z}_{i,j}^{0}+tX^{(\set_{j})}_{i})^{\top}w)_{i}))^{\mathbf{q}}\rangle_{\bar{H}_{\beta;-i}}|X^{(\set_{j})}_{i}]\right].
\end{split}
\label{upperbound:g-phi_rho}
\end{equation}

In the second step, we derive an upper bound for the moments of the random variables by leveraging the properties of the loss function. Assumption \ref{ass::loss} (1) and (2) yield
\begin{equation}
    \begin{split}
    &E[\langle(\loss_{i;\rho}(((\mathbf{Z}_{i,j}^{0}+tX^{(\set_{j})}_{i})^{\top}w)_{i}))^{\mathbf{q}}\rangle_{\bar{H}_{\beta;-i}}|X^{(\set_{j})}_{i}] \\
    &\leq E[\langle(\|\loss_{i;\rho}(((\mathbf{Z}_{i,j}^{0}+tX^{(\set_{j})}_{i})^{\top}w)_{i})-\loss_{i}(((\mathbf{Z}_{i,j}^{0}+tX^{(\set_{j})}_{i})^{\top}w)_{i})\|_{\infty} \\
    & \quad +\loss_{i}(((\mathbf{Z}_{i,j}^{0}+tX^{(\set_{j})}_{i})^{\top}w)_{i}))^{\mathbf{q}}\rangle_{\bar{H}_{\beta;-i}}|X^{(\set_{j})}_{i}] \\
    &\lesssim_{q}(1+E[\langle(\loss_{i}(((\mathbf{Z}_{i,j}^{0}+tX^{(\set_{j})}_{i})^{\top}w)_{i}))^{\mathbf{q}}\rangle_{\bar{H}_{\beta;-i}}|X^{(\set_{j})}_{i}])  \\
    &\lesssim_{q}(1+E[\langle L^{\mathbf{q}}_{\loss_{i}}(1+|((\mathbf{Z}_{i,j}^{0}+tX^{(\set_{j})}_{i})^{\top}w)_{i}|^{\mathbf{q}_{0}})\rangle_{\bar{H}_{\beta;-i}}|X^{(\set_{j})}_{i}])  \\
    &\lesssim_{q} L^{\mathbf{q}}_{\loss_{i}}\left(1+E\left[\langle| ((\mathbf{Z}_{i,j}^{0})^{\top}w)_{i}|^{\mathbf{q}_{0}} \rangle_{\bar{H}_{\beta;-i}}|X^{(\set_{j})}_{i}\right]+E\left[\langle |t(X^{(\set_{j})}_{i})^{\top}w|^{\mathbf{q}_{0}}\rangle_{\bar{H}_{\beta;-i}}|X^{(\set_{j})}_{i}\right]\right),
    \end{split}
    \label{upperbound:phi_rho-XandZ}
\end{equation}
where we define $\mathbf{q}_{0}=q_{0}\mathbf{q}$ and $\bar{\mathbf{q}}_{0}:=2\lceil\mathbf{q}_{0}/2\rceil$, respectively.
By trivial calculation, for any $1\leq j\leq k$ we have
\begin{equation}
    |t(X^{(\set_{j})}_{i})^{\top}w|^{\mathbf{q}_{0}}\leq t^{\mathbf{q}_{0}}L^{\mathbf{q}_{0}}_{n}\left(\sum_{m\in \set_j}|X_{i, m}|\right)^{\mathbf{q}_{0}}.\label{upperbound:X_cond_X}
\end{equation}
Moreover, let $C^{j}_{i}$ denote the $i$-th row of $\mathbf{Z}_{i,j}^{0}$. Then, we have
\begin{equation}
    \begin{split}
         E\left[\langle| ((\mathbf{Z}_{i,j}^{0})^{\top}w)_{i}|^{\bar{\mathbf{q}}_{0}} \rangle_{\bar{H}_{\beta;-i}}|X^{(\set_{j})}_{i}\right] 
    \leq & E\left[\sum_{k_{1},\cdots,k_{\bar{\mathbf{q}}_{0}}\notin \set_{j}}C^{j}_{ik_{1}}\cdots C^{j}_{ik_{\mathbf{q}_{0}}}\langle w_{k_{1}}\cdots w_{k_{\bar{\mathbf{q}}_{0}}} \rangle_{\bar{H}_{\beta;-i}}\right] \\
    \leq &\sum_{k_{1},\cdots,k_{\bar{\mathbf{q}}_{0}}\notin \set_{j}} E\left[C^{j}_{ik_{1}}\cdots C^{j}_{ik_{\bar{\mathbf{q}}_{0}}}\right]E\left[\langle w_{k_{1}}\cdots w_{k_{\bar{\mathbf{q}}_{0}}} \rangle_{\bar{H}_{\beta;-i}}\right]\\
    \leq &L_{n}^{\bar{\mathbf{q}}_{0}}\sum_{k_{1},\cdots,k_{\bar{\mathbf{q}}_{0}}\notin \set_{j}} E\left[C^{j}_{ik_{1}}\cdots C^{j}_{ik_{\bar{\mathbf{q}}_{0}}}\right],
    \end{split}
    \label{upperbound:Z_cond_X}
\end{equation}
where the second inequality follows from i.i.d. property of $X_{i}$ and $W_{i}$.  Therefore, from \eqref{upperbound:g-phi_rho} to \eqref{upperbound:Z_cond_X}, we have 
\begin{equation}\label{upperbound:final}
\begin{split}
       &E[| \partial_{i,\ell}\partial_{i,\ell^{\prime}}\partial_{i,\ell^{\prime\prime}}g(n^{-1}F_{\beta}(\mathbf{Z}_{i,j}^{0}+tX^{(\set_{j})}_{i}))||X^{(\set_{j})}_{i}] \\
     &\lesssim K_{g}n^{-1}\beta^{2}L_{n}^{3}D^{3}_{\loss}(\rho)\left[1+L^{\mathbf{q}}_{\loss_{i}}L_{n}^{\bar{\mathbf{q}}_{0}}\left(1+\sum_{k_{1},\cdots,k_{\bar{\mathbf{q}}_{0}}\notin \set_{j}} E\left[C^{j}_{ik_{1}}\cdots C^{j}_{ik_{\bar{\mathbf{q}}_{0}}}\right]+t^{\bar{\mathbf{q}}_{0}}\left(\sum_{m\in \set_j}|X_{i, m}|\right)^{\bar{\mathbf{q}}_{0}}\right)\right].
\end{split}    
\end{equation}
The second term expresses the sum of all combinations of $\mathbf{{\bar{q}}}_{0}$ random variables exerted from the $i$-th row of $\mathbf{Z}_{i,j}^{0}$ except those in $\set_{j}$. Combining Lemma \ref{Lipschitz_higherorder} with \eqref{upperbound:final}, we attain 
\begin{align*}
&\left|E\left[\int^{1}_{0}(1-t)^2\sum_{\ell \in B_j} \sum_{\ell' \in B_j} \sum_{\ell'' \in B_j} X_{i,\ell} X_{i,\ell'} X_{i,\ell''}\partial_{i, \ell}\partial_{ i, \ell'} \partial_{ i, \ell''}g(n^{-1}F_{\beta}(\mathbf{Z}_{i,j}^{0}+tX^{(\set_{j})}_{i}))dt\right]\right| \\
&\leq E\left[\int^{1}_{0}(1-t)^2\left(\sum_{\ell\in B_j}\left|X_{i,\ell}\right|\right)^{3}E\left[\left| \partial_{i,(\set_{j})_{\ell}}\partial_{i,(\set_{j})_{\ell^{\prime}}}\partial_{i,(\set_{j})_{\ell^{\prime\prime}}}g(n^{-1}F_{\beta}(\mathbf{Z}_{i,j}^{0}+tX^{(\set_{j})}_{i}))\right|X^{(\set_{j})}_{i}\right] dt\right] \\
&\lesssim K_{g}n^{-1}\beta^{2}L_{n}^{3}D^{3}_{\loss}(\rho)E\left[\left(\sum_{\ell\in B_j}\left|X_{i,\ell}\right|\right)^{3}\right] \\
&+K_{g}n^{-1}\beta^{2}L_{n}^{\bar{\mathbf{q}}_{0}+3}D^{3}_{\loss}(\rho)L^{\mathbf{q}}_{\loss_{i}}
E\left[\left(\sum_{\ell\in B_j}\left|X_{i,\ell}\right|\right)^{3}\left(1+\left(\sum_{\ell\in B_j}\left|X_{i,\ell}\right|\right)^{\bar{\mathbf{q}}_0}+\sum_{k_{1},\cdots,k_{\bar{\mathbf{q}}_{0}}\notin \set_{j}} E\left[C^{j}_{ik_{1}}\cdots C^{j}_{ik_{\bar{\mathbf{q}}_{0}}}\right]\right)\right].
\end{align*}
Therefore, the required result holds.
\end{proof}

\subsection{Proof of Lemma \ref{general_inequality_dependent}}
\begin{proof}
Let $C^{j}_{i}$ denote the $i$-th row of $\mathbf{Z}_{i,j}^{0}$. By definition, there exists $\alpha\in\mathbb{N}$ such that $\bar{\mathbf{q}}_{0}=2\alpha$.  Throughout this proof, we assume $\bar{\mathbf{q}}_{0}\leq p$. The summation, $\sum_{k_{1},\cdots,k_{\bar{\mathbf{q}}_{0}}\notin \set_{j}} $, expresses the sum of all combinations of random variables $\mathbf{{\bar{q}}}_{0}$ with repetition, exerted from the $i$-th row of $\mathbf{Z}_{i,j}^{0}$ except those in $\set_{j}$. Let $\ell$ denote the number of different dependent index cells in which random variables are.

We divide this summation into three parts based on $\ell$, i.e. (i)$1\leq\ell\leq\bar{\mathbf{q}}_{0}-\alpha-1$, (ii) $\ell=\bar{\mathbf{q}}_{0}-\alpha$  and (iii) $(\bar{\mathbf{q}}_{0}+1)-\alpha\leq \ell\leq \bar{\mathbf{q}}_{0}$, and denote corresponding summation as $A_{1}$, $A_{2}$ and $A_{3}$, respectively.
\begin{equation*}
     \sum_{k_{1},\cdots,k_{\bar{\mathbf{q}}_{0}}\notin \set_{j}} E\left[C^{j}_{ik_{1}}\cdots C^{j}_{ik_{\bar{\mathbf{q}}_{0}}}\right]=A_{1}+A_{2}+A_{3}.
\end{equation*}

We consider three cases above and derive upper bounds for them.
\begin{enumerate}
    \item[(i)] First, we consider the number of terms included in this case. After deriving the uniform upper bound of expectations, we can upper bound $A_{1}$ by the multiplication of the number of terms and the uniform bound. 

    Let $N_{1}$ denote the number of terms included in this case. To derive the number of terms included in this cell, we consider the following steps:
    \begin{enumerate}
        \item[Step 1.] Fix $\ell$ and pick up a specific $k_{1}$ from $[p]$ except elements in $\set_{j}$. The possible ways to select $C^{j}_{ik_{1}}$ is $p-|\set_{j}|$. Let $\set^{*}_{1,\ell}$ denote the dependent index cell of $C^{j}_{ik_{1}}$.
        \item[Step 2.] Set $\gamma_{1,\ell}$ as the number of variables included in the same dependent index cell as $C^{j}_{ik_{1}}$. Because the maximum cardinality of dependent index cells is $d$, the possible ways to select variables is $d^{(\gamma_{1,\ell})-1}$. Denote those variables as $C^{j}_{i,k_{2}}, \cdots, C^{j}_{i,k_{\gamma_{1,\ell}}}$. 
        \item[Step 3.] Pick up a specific $k_{\gamma_{1,\ell}+1}$ from $[p]$ except elements in $\set_{j}$ and $\set^{*}_{1,\ell}$. The possible ways to select $C^{j}_{ik_{(\gamma_{1,\ell})+1}}$ is $p-(|\set_{j}|+|\set^{*}_{1,\ell}|)$. Let $\set^{*}_{2,\ell}$ denote the dependent index cell of $C^{j}_{ik_{(\gamma_{1,\ell})+1}}$. 
        \item[Step 4.] Set $\gamma_{2,\ell}$ as the number of variables included in the same dependent index cell as $C^{j}_{ik_{(\gamma_{1,\ell})+1}}$. As in Step 2, the possible ways to select variables are $d^{(\gamma_{2,\ell})-1}$. Denote those variables as $C^{j}_{ik_{(\gamma_{1,\ell})+2}}, \cdots, C^{j}_{ik_{(\gamma_{1,\ell})+(\gamma_{2,\ell})}}$. 
        \item[Step 5.] Iterate this process until we define $\set^{*}_{\ell,\ell}$ and $\gamma_{\ell,\ell}$
    \end{enumerate}
    Because the possible ways to choose $\gamma_{1,\ell},\cdots,\gamma_{\ell,\ell}$ only depends on $q$, we have 
    \begin{align*}
        N_{1}&\lesssim_{q}\sum_{\ell=1}^{\bar{\mathbf{q}}_{0}-\alpha-1}p^{\ell}d^{\sum_{\ell^{*}=1}^{\ell}(\gamma_{\ell^{*},\ell}-1)} =\sum_{\ell=1}^{\bar{\mathbf{q}}_{0}-\alpha-1}p^{\ell}d^{\bar{\mathbf{q}}_{0}-\ell},
    \end{align*}
    where the last equality follows from the definition of $\gamma_{\ell^{*},\ell}$, i.e. $\sum_{\ell^{*}=1}^{\ell}\gamma_{\ell^{*},\ell}=\bar{\mathbf{q}}_{0}$. Moreover, because $d\leq p$ holds, for any $1\leq \ell \leq \bar{\mathbf{q}}_{0}-\alpha-2$, we have
    \begin{equation*}
        p^{\ell}d^{\bar{\mathbf{q}}_{0}-\ell} \leq p^{\bar{\mathbf{q}}_{0}/2-1}d^{\bar{\mathbf{q}}_{0}/2+1}. 
    \end{equation*}
    Therefore, we have 
    \begin{equation*}
        N_{1}\lesssim_{q}p^{\bar{\mathbf{q}}_{0}/2-1}d^{\bar{\mathbf{q}}_{0}/2+1}. 
    \end{equation*}
    Apparently, the upper bound of the uniform expectations is $\max_{k_{1},\cdots,k_{\bar{\mathbf{q}}_{0}}\in[p]^{\bar{\mathbf{q}}_{0}}}E\left[C^{j}_{ik_{1}}\cdots C^{j}_{ik_{\bar{\mathbf{q}}_{0}}}\right]$. Therefore, we have
    \begin{equation*}
        A_{1}\lesssim_{q} p^{\bar{\mathbf{q}}_{0}/2-1}d^{\bar{\mathbf{q}}_{0}/2+1}\max_{k_{1},\cdots,k_{\bar{\mathbf{q}}_{0}}\in[p]^{\bar{\mathbf{q}}_{0}}}E\left[C^{j}_{ik_{1}}\cdots C^{j}_{ik_{\bar{\mathbf{q}}_{0}}}\right].
    \end{equation*}
    
    \item[(ii)] When there exists at least one dependent index cell that has only one element, that term becomes zero due to the expectation of that element. From the definition of $\ell$, there is no dependent index cell where more than two random variables belong. Thus, the expectation can be divided into pairs of random variables. Moreover, we only have ${}_{\bar{\mathbf{q}}_{0}} \rm C_{2}$ patters of pairs belonging to the same dependent cell from $C^{j}_{ik_{1}}\cdots C^{j}_{ik_{\bar{\mathbf{q}}_{0}}}$. Hence,  to derive the upper bound, all we have to do is to consider the case when $k_{t}$ and $k_{t+1}$ are in the same dependent index cell where $t$ is odd, i.e.
    \begin{align*}
         &\sum_{k_{1},\cdots,k_{\bar{\mathbf{q}}_{0}}\notin \set_{j}} E\left[C^{j}_{ik_{1}}\cdots C^{j}_{ik_{\bar{\mathbf{q}}_{0}}}\right] \\
         &=\sum_{\substack{\ell_{1}=1 \\ \ell_{1}\neq j}}^{k}\sum_{\ell_{2}\neq \ell_{1}, j}^{k}\cdots \sum_{\ell_{\alpha}\neq \ell_{1}, \ell_{2},\cdots, \ell_{\alpha-1}, j}^{k} \sum_{\{k_{1},k_{2}\}\in \set_{\ell_{1}}}\cdots\sum_{\{k_{\bar{\mathbf{q}}_{0}-1},k_{\bar{\mathbf{q}}_{0}}\}\in \set_{\ell_{\alpha}}}E\left[C^{j}_{ik_{1}}\cdots C^{j}_{ik_{\bar{\mathbf{q}}_{0}}}\right].
    \end{align*}
By definition of the dependent structure, we have
\begin{align*}
    &\sum_{\substack{\ell_{1}=1 \\ \ell_{1}\neq j}}^{k}\sum_{\ell_{2}\neq \ell_{1}, j}\cdots \sum_{\ell_{\alpha}\neq \ell_{1}, \ell_{2},\cdots, \ell_{\alpha-1}, j} \sum_{\{k_{1},k_{2}\}\in \set_{\ell_{1}}}\cdots\sum_{\{k_{\bar{\mathbf{q}}_{0}-1},k_{\bar{\mathbf{q}}_{0}}\}\in \set_{\ell_{\alpha}}}E\left[C^{j}_{ik_{1}}\cdots C^{j}_{ik_{\bar{\mathbf{q}}_{0}}}\right]  \\ 
    &=\sum_{\substack{\ell_{1}=1 \\ \ell_{1}\neq j}}^{k}\sum_{\ell_{2}\neq \ell_{1}, j}\cdots \sum_{\ell_{\alpha}\neq \ell_{1}, \ell_{2},\cdots, \ell_{\alpha-1}, j} \sum_{\{k_{1},k_{2}\}\in \set_{\ell_{1}}}\cdots\sum_{\{k_{\bar{\mathbf{q}}_{0}-1},k_{\bar{\mathbf{q}}_{0}}\}\in \set_{\ell_{\alpha}}}E\left[C^{j}_{ik_{1}}C^{j}_{ik_{2}}\right]\cdots E\left[C^{j}_{ik_{\bar{\mathbf{q}}_{0}-1}}C^{j}_{ik_{\bar{\mathbf{q}}_{0}}}\right] \\ 
     &=\left(\frac{1}{n}\right)^{\alpha} \sum_{\substack{\ell_{1}=1 \\ \ell_{1}\neq j}}^{k}\sum_{\ell_{2}\neq \ell_{1}, j}\cdots \sum_{\ell_{\alpha}\neq \ell_{1}, \ell_{2},\cdots, \ell_{\alpha-1}, j} \boldsymbol{\iota}_{|\set_{\ell_{1}}|}^{\top}\Sigma_{|\set_{\ell_{1}}|}\boldsymbol{\iota}_{|\set_{\ell_{1}}|}\boldsymbol{\iota}_{|\set_{\ell_{2}}|}^{\top}\Sigma_{|\set_{\ell_{2}}|}\boldsymbol{\iota}_{|\set_{\ell_{2}}|}\cdots \boldsymbol{\iota}_{|\set_{\ell_{\alpha}}|}^{\top}\Sigma_{|\set_{\ell_{\alpha}}|}\boldsymbol{\iota}_{|\set_{\ell_{\alpha}}|} \\
     &\leq \left(\frac{1}{n}\right)^{\alpha} (\max_{1\leq i\leq k}\lambda_{\max}(\Sigma_{|\set_{i}|}))^{\alpha}\sum_{\substack{\ell_{1}=1 \\ \ell_{1}\neq j}}^{k}\sum_{\ell_{2}\neq \ell_{1}, j}\cdots \sum_{\ell_{\alpha}\neq \ell_{1}, \ell_{2},\cdots, \ell_{\alpha-1}, j} |\set_{\ell_{1}}|\cdots |\set_{\ell_{\alpha}}| \\
     &\leq \left(\frac{p}{n}\right)^{\alpha} (\max_{1\leq i\leq k}\lambda_{\max}(\Sigma_{|\set_{i}|}))^{\alpha},
    \end{align*}
    where we define $\Sigma_{|\set_{i}|}$ as a variance-covariance matrix of the cell $\set_{i}$. The last inequality is valid from the definition of the partition. Let $\boldsymbol{\iota}_{|\set_{i}|}$ denote a $|\set_{i}|\times 1$ vector of ones. Therefore, we obtain
    \begin{equation*}
        A_{2}\lesssim_{q,\tau} (\max_{1\leq i\leq k}\lambda_{\max}(\Sigma_{|\set_{i}|}))^{\bar{\mathbf{q}}_{0}/2}.
    \end{equation*}
    \item[(iii)] In this case, at least one cell consists of just one element of $C^{j}_{i1},\cdots, C^{j}_{i\bar{\mathbf{q}}_{0}}$, which implies we have $E[C^{j}_{ik_{1}}\cdots C^{j}_{i\bar{\mathbf{q}}_{0}}]=0$. Therefore, $A_{3}$ is equal to 0.
\end{enumerate}

In summary, the discussion above yields
\begin{align*}
&\sum_{k_{1},\cdots,k_{\bar{\mathbf{q}}_{0}}\notin \set_{j}} E\left[C^{j}_{ik_{1}}\cdots C^{j}_{ik_{\bar{\mathbf{q}}_{0}}}\right] \\
&\lesssim_{q,\tau} (\max_{1\leq i\leq k}\lambda_{\max}(\Sigma_{|\set_{i}|}))^{\bar{\mathbf{q}}_{0}/2}+
    p^{\bar{\mathbf{q}}_{0}/2-1}d^{\bar{\mathbf{q}}_{0}/2+1}\max_{k_{1},\cdots,k_{\bar{\mathbf{q}}_{0}}\in[p]^{\bar{\mathbf{q}}_{0}}}E\left[C^{j}_{ik_{1}}\cdots C^{j}_{ik_{\bar{\mathbf{q}}_{0}}}\right].
\end{align*}
As in the proof of Proposition \ref{step4}, we have
   \begin{equation*}        \max_{k_{1},\cdots,k_{\bar{\mathbf{q}}_{0}}\in[p]^{\bar{\mathbf{q}}_{0}}}E\left[C^{j}_{ik_{1}}\cdots C^{j}_{ik_{\bar{\mathbf{q}}_{0}}}\right] \leq n^{-\bar{\mathbf{q}}_{0}/2}\max\{\max_{1\leq k\leq p}(\|X_{i,k}\|_{2^{(\bar{\mathbf{q}}_{0}/2)}})^{(\bar{\mathbf{q}}_{0})}, \max_{1\leq k\leq p}(\|W_{i,k}\|_{2^{(\bar{\mathbf{q}}_{0}/2)}})^{(\bar{\mathbf{q}}_{0})}\}.
    \end{equation*}
    Therefore, the required result holds.
 \end{proof}

\subsection{Proof of Lemma \ref{Lemma_A.1}}

\begin{proof}[Proof of Lemma \ref{Lemma_A.1}]

The required result immediately holds by  Lemma A.1 of \citet{thrampoulidis2018precise} with replacement of $\|\cdot\|$ by $\|\cdot\|_{\Sigma}$. 
\end{proof}

\subsection{Proof of Lemma \ref{Lemma_A.2}}

\begin{proof}[Proof of Lemma \ref{Lemma_A.2}]
We can prove Lemma \ref{Lemma_A.2} by following the proof of  Lemma A.2 of \citet{thrampoulidis2018precise} by using $\mathbf{A}\Sigma^{1/2}$ instead of $\mathbf{A}$.
\end{proof}

\subsection{Proof of Lemma \ref{Lemma_A.3}}

\begin{proof}[Proof of Lemma \ref{Lemma_A.3}]
This naturally follows from the application of Theorem 4 in \citet{akhtiamov2024novel}.
\end{proof}

\subsection{Proof of Lemma \ref{Lemma_A.4}}

\begin{proof}[Proof of Lemma \ref{Lemma_A.4}]

(i): First, we consider the optimization in $\mathbf{u}$ of \eqref{TAH_71}. Using the fact that $\max_{\|\mathbf{u}\|_{2}=\beta}\mathbf{u}^{\top}\mathbf{t}=\beta\|\mathbf{t}\|_{2}$ holds for all $\beta\geq0$, we obtain
\begin{equation*}
\max_{0\leq \beta\leq K_{\beta}, \mathbf{s}}\min_{\|w\|_{\Sigma}\leq K_{\alpha}, \mathbf{v}}\frac{\beta}{\sqrt{p}}\|\|w\|_{\Sigma}\mathbf{g}+\xi-\mathbf{v}\|_{2}-\frac{\beta}{\sqrt{p}}\mathbf{h}^{\top}\Sigma^{1/2}w+\frac{1}{p}\mathcal{L}(\mathbf{v})+\frac{\lambda}{p} \mathbf{s}^{\top}\theta_0+\frac{\lambda}{ \sqrt{p}}\mathbf{s}^{\top}w-\frac{\lambda}{p} f^{*}(\mathbf{s}).
\end{equation*}

Second, we consider the optimization over $w$. Fixing $\|w\|_{\Sigma}$ as $\alpha$ and using a similar argument to the above discussion, we get
\begin{equation}\label{TAH_87}
\max_{0\leq \beta\leq K_{\beta}, \mathbf{s}}\min_{0\leq \alpha \leq K_{\alpha}, \mathbf{v}}\frac{\beta}{\sqrt{p}}\|\alpha\mathbf{g}+\xi-\mathbf{v}\|_{2}+\frac{1}{p}\mathcal{L}(\mathbf{v})-\frac{\alpha}{\sqrt{p}}\|\beta\mathbf{h}-\lambda\Sigma^{-1/2}\mathbf{s}\|_{2}+\frac{\lambda}{p} \mathbf{s}^{\top}\theta_0-\frac{\lambda}{p} f^{*}(\mathbf{s}).
\end{equation}
Let $M_{\Sigma^{-1/2}}(\alpha,\beta,\mathbf{v},\mathbf{s})$ denote the objective function of \eqref{TAH_87}.  $M_{I}(\alpha,\beta,\mathbf{v},\mathbf{s})$ becomes equal to the objective function (87) in \citet{thrampoulidis2018precise}. By definition of $M_{\Sigma^{-1/2}}$, we have
\begin{equation*}
    M_{\Sigma^{-1/2}}(\alpha,\beta,\mathbf{v},\mathbf{s})=M_{I}(L(\alpha,\beta,\mathbf{v},\mathbf{s})^{\top}),
\end{equation*}
where we define
\begin{equation*}
L=\begin{pmatrix}
   1 & 0 & \mathbf{0}_{dim(\mathbf{s})} & \mathbf{0}_{dim(\mathbf{s})} \\
     0 & 1 & \mathbf{0}_{dim(\mathbf{s})} & \mathbf{0}_{dim(\mathbf{s})} \\
       0 & 0 & \mathbf{I}_{dim(\mathbf{s})} & \mathbf{0}_{dim(\mathbf{s})} \\
         0 & 0 & \mathbf{0}_{dim(\mathbf{s})} & \Sigma^{-1/2} \\
\end{pmatrix}.
\end{equation*}
Because $M_{I}$ is jointly convex in $(\alpha,\mathbf{v})$ and jointly concave in $(\beta,\mathbf{s})$,  $\min_{\mathbf{v}}M_{\Sigma^{-1/2}}(\alpha,\beta,\mathbf{v},\mathbf{s})$ is convex in $\alpha$ and jointly concave in $(\beta,\mathbf{s})$. Then, from Corollary 3.3 of \citet{sion1958general}, we can flip the order of $\max_{\beta,\mathbf{s}}$ and $\min_{\alpha}$, i.e. \eqref{TAH_87} can be reduced to 
\begin{equation*}
\min_{0\leq \alpha \leq K_{\alpha}}\max_{0\leq \beta\leq K_{\beta}}\max_{\mathbf{s}}\min_{\mathbf{v}}M_{\Sigma^{-1/2}}(\alpha,\beta,\mathbf{v},\mathbf{s}).
\end{equation*}
Moreover, because $M_{\Sigma^{-1/2}}(\alpha,\beta,\mathbf{s},\mathbf{v})$ is separable in $\mathbf{s}$ and $\mathbf{v}$, we can flip the order of minimization and maximization of $\mathbf{v}$ and $\mathbf{s}$. 

Furthermore, we simplify the optimization problem through the square-root trick. We apply the fact that $\sqrt{\chi}=\inf_{\tau>0}\{\frac{\tau}{2}+\frac{\chi}{2\tau}\}$ to both terms $\frac{1}{\sqrt{p}}\|\alpha\mathbf{g}+\xi-\mathbf{v}\|_{2}$ and $\frac{1}{\sqrt{p}}\|\beta\mathbf{h}-\lambda\Sigma^{-1/2}\mathbf{s}\|_{2}$.:
\begin{equation}\label{TAH_88}
\begin{split}
   \min_{0\leq \alpha \leq K_{\alpha}} \max_{0\leq \beta\leq K_{\beta}}\inf_{\tau_g>0}\sup_{\tau_h>0}\frac{\beta\tau_g}{2}+&\frac{1}{p}\min_{\mathbf{v}}\left\{\frac{\beta}{2\tau_g}\|\alpha\mathbf{g}+\xi-\mathbf{v}\|^2_{2}+\mathcal{L}(\mathbf{v})\right\} \\
   -&\frac{\alpha \tau_h}{2}-\frac{1}{p}\min_{\mathbf{s}}\left\{\frac{\alpha}{2\tau_h}\|\beta\mathbf{h}-\lambda\Sigma^{-1/2}\mathbf{s}\|^2_{2}-\lambda \mathbf{s}^{\top}\theta_0+\lambda f^{*}(\mathbf{s})\right\}.
\end{split}
\end{equation}

Finally, we summarize the objective function and optimization problem through the use of Moreau envelope. Because the difference in \eqref{TAH_88} and (88) in \citet{thrampoulidis2018precise} does not depend on $\beta$ and $\tau_g$, we can exchange the order of optimization in $\beta$ and $\tau_g$. Moreover, it follows from Assumptions \ref{ass::moment} and \ref{ass::eigen} that
\begin{align*}
    &\min_{\mathbf{s}}\left\{\frac{\alpha}{2\tau_h}\|\beta\mathbf{h}-\lambda\Sigma^{-1/2}\mathbf{s}\|^2_{2}-\lambda \mathbf{s}^{\top}\theta_0+\lambda f^{*}(\mathbf{s})\right\} \\
    &=\min_{\mathbf{s}}\left\{\frac{\alpha}{2\tau_h}\|\beta\mathbf{h}-\lambda\mathbf{s}\|^2_{2}-\lambda \mathbf{s}^{\top}\Sigma^{1/2}\theta_0+\lambda f^{*}(\Sigma^{1/2}\mathbf{s})\right\}.
\end{align*}
Let $g(x)$ denote $f(\Sigma^{-1/2}x)$. Then, we get
\begin{align*}
    f^{*}(\Sigma^{1/2}\mathbf{s})=&\max_{\mathbf{x}}\mathbf{s}^{\top}\Sigma^{1/2}\mathbf{x}-f(\mathbf{x}) \\
    =&\max_{\mathbf{x}}\mathbf{s}^{\top}\Sigma^{1/2}\mathbf{x}-g(\Sigma^{1/2} \mathbf{x}) \\
    =&\max_{\mathbf{x}}\mathbf{s}^{\top}\mathbf{x}-g(\mathbf{x}) \\
    =&g^{*}(\mathbf{s}).
\end{align*}
Therefore, as in \citet{thrampoulidis2018precise}, we have
\begin{align}
    &\min_{\mathbf{s}}\left\{\frac{\alpha}{2\tau_h}\|\beta\mathbf{h}-\lambda\Sigma^{-1/2}\mathbf{s}\|^2_{2}-\lambda \mathbf{s}^{\top}\theta_0+\lambda f^{*}(\mathbf{s})\right\}\\
    &=\min_{\mathbf{s}}\left\{\frac{\alpha}{2\tau_h}\|\beta\mathbf{h}-\lambda\mathbf{s}\|^2_{2}-\lambda \mathbf{s}^{\top}\Sigma^{1/2}\theta_0+\lambda g^{*}(\mathbf{s})\right\} \\
    &=-\frac{\tau_h}{2\alpha}\|\Sigma^{1/2}\theta_0\|^{2}_{2}-\beta\mathbf{h}^{\top}\Sigma^{1/2}\theta_0+\lambda e_{g^{*}}\left(\frac{\beta}{\lambda}\mathbf{h}+\frac{\tau_h}{\alpha \lambda}\Sigma^{1/2}\theta_0;\frac{\tau_h}{\alpha\lambda}\right) \label{TAH_89} \\
    &=\frac{\beta^{2}\alpha}{2\tau_h}\|\mathbf{h}\|^{2}_{2}-\lambda e_{g}\left(\frac{\beta\alpha}{\tau_h}\mathbf{h}+\Sigma^{1/2}\theta_0;\frac{\alpha\lambda}{\tau_h}\right). \label{TAH_90}
\end{align}

(ii) From the property of Moreau envelope, the continuity follows. Especially, at $\alpha=0$, as per Theorem 10.3 of \citet{rockafellar1997convex}, the limit of the RHS in \eqref{TAH_90} converges to $-\lambda g(\Sigma^{1/2}\theta_0)$ equal to $-\lambda f(\theta_0)$.

(iii) We can directly apply the result of Lemma B.4 in \citet{thrampoulidis2018precise} to our setting because a function $h(\mathbf{h},\mathbf{s})=(\mathbf{h},\Sigma^{-1/2}\mathbf{s})^{\top}$ is linear.
\end{proof}

\subsection{Proof of Lemma \ref{Lemma_A.5}}

\begin{proof}[Proof of Lemma \ref{Lemma_A.5}]

(a) From Assumption \ref{ass::TAH_1}, it holds that the normalized Moreau envelope function in \eqref{TAH_73} converges in probability to $L$ and $F$, respectively. Moreover, the convergence part holds because  the probability limit of $\|\mathbf{h}\|_{2}/p$ is $1$ by the weak law of large numbers. 

$\mathcal{R}_{p}$ is convex-concave due to the result of Lemma \ref{Lemma_A.4} (i). Since convexity and concavity are preserved by the pointwise limits, the same holds for $\mathcal{D}$.

(b) The proof follows that of Lemma A.5 in \citet{thrampoulidis2018precise}. That proof requires some convex and concave conditions, as well as the application of Lemma B.1, Lemma B.2 and Lemma D.1 in \citet{thrampoulidis2018precise}. These conditions naturally follow from the convexity and concavity of $\mathcal{R}_{p}$ and  $\mathcal{D}$ that are  guaranteed from the previous part. Moreover, three aforementioned lemmas provide general results irrespective of the form of the covariance matrix. Therefore, the required result holds. 
\end{proof}

\appendix

\bibliographystyle{plainnat}
\bibliography{main}
\end{document}